\documentclass[12pt]{amsart}

\usepackage{soul}

\usepackage[hmargin=2.8cm,bmargin=2.8cm,tmargin=3cm]{geometry}
\usepackage{soul}

\usepackage{amsmath,amssymb,amsthm,amsfonts,enumerate,url}
\usepackage{mathrsfs}
\usepackage{graphicx}

\usepackage{hyperref}

\usepackage{tikz}
\usetikzlibrary{automata,positioning,arrows.meta,shapes.misc}

\tikzset{cross/.style={cross out, draw=black, minimum size=2*(#1-\pgflinewidth), inner sep=1pt, outer sep=0pt},
cross/.default={1pt}}

\usepackage{pgfplots}
\usepgfplotslibrary{polar}
\usepackage{float}

\usepackage[utf8]{inputenc}

\theoremstyle{plain}
\newtheorem{theorem}{Theorem}[section]
\newtheorem{lemma}[theorem]{Lemma}
\newtheorem{proposition}[theorem]{Proposition}
\newtheorem{corollary}[theorem]{Corollary}

\theoremstyle{definition}
\newtheorem{definition}[theorem]{Definition}
\newtheorem{example}[theorem]{Example}

\newtheorem{remark}[theorem]{Remark}

\numberwithin{equation}{section}
\numberwithin{figure}{section}

\newcommand\Graph{\mathcal{G}}
\newcommand\GraphH{\mathcal{H}}
\newcommand\NewGraph{\widetilde{\mathcal{G}}}

\newcommand\spec{\sigma}

\newcommand\newa{\tilde{a}}
\newcommand\newX{\widetilde{X}}

\newcommand{\cy}{\mathfrak{c}} %Length of the longest cycle of a graph
\newcommand{\dcp}{\mathscr{D}_{\mathcal{G}}} %Doubly connected part of a graph
\newcommand{\dcpgraph}[1]{\mathscr{D}_{#1}} %Doubly connected part of the graph #1

\newcommand{\R}{\mathbb{R}}

\newcommand{\E}{\mathbb{E}}

\newcommand{\mG}{\mathsf{G}}

\DeclareMathOperator{\dist}{dist}
\DeclareMathOperator{\var}{var}
\DeclareMathOperator{\aufspan}{span}

\newcommand{\nd}[3]{\partial_\nu {#1}|_{#2}(#3)} %Normal derivative #1=function #2=edge #3=vertex

\newcommand{\pstick}[3]{\mathcal{P}[#1,#2,#3]} %Pumpkin-on-a-stick! #1=\ell pumpkin length, #2=k pumpkin thickness, #3=d pumpkin dist
\newcommand{\pdb}[3]{\mathcal{D}[#1,#2,#3]} %Pumpkin dumbbell, #1=\ell_1 left pumpkin length, #2=\ell_2 right pumpkin, #3=k thickness
\newcommand{\db}[2]{\mathcal{D}[#1,#2]} %P-dumbbell with k=2, #1=\ell_1 left pumpkin length, #2=\ell_2 right pumpkin
\newcommand{\tp}[1]{\mathcal{L}[#1]} %Tadpole (lasso) with #1=loop length

\title{Surgery principles for the spectral analysis of quantum graphs} 



\author[G.~Berkolaiko]{Gregory Berkolaiko}
\author[J.~B.~Kennedy]{James B.~Kennedy}
\author[P.~Kurasov]{Pavel Kurasov}
\author[D.~Mugnolo]{Delio Mugnolo}

\address{Gregory Berkolaiko, Department of Mathematics, Texas A{\&}M University, College Station, TX 77843-3368, USA}
\email{gregory.berkolaiko@math.tamu.edu}

\address{James B.~Kennedy, Grupo de F\'isica Matem\'atica, Faculdade de Ci\^encias, Universidade de Lisboa, Campo Grande, Edif\'{i}cio~C6, \mbox{P-1749-016} Lisboa, Portugal}
\email{jbkennedy@fc.ul.pt}

\address{Pavel Kurasov, Department of Mathematics, Stockholm University, SE-106 91 Stockholm, Sweden}
\email{kurasov@math.su.se}

\address{Delio Mugnolo, Lehrgebiet Analysis, Fakult\"at Mathematik und Informatik, Fern\-Universit\"at in Hagen, D-58084 Hagen, Germany}
\email{delio.mugnolo@fernuni-hagen.de}

\thanks{The authors would like to thank the anonymous referees for their 
careful reading of, and helpful suggestions for, the draft version of this article. 
The work of G.B. was partially supported by the NSF under grant DMS-1410657.
J.B.K.~was supported by the Funda{\c{c}}{\~a}o para a Ci{\^e}ncia e a
Tecnologia, Portugal, via the program ``Investigador FCT'', reference
IF/01461/2015, and project PTDC/MAT-CAL/4334/2014. 
P.K. was partially supported by the Swedish Research Council (Grant D0497301).
D.M. was partially supported by the Deutsche Forschungsgemeinschaft (Grant 397230547).
All four authors were partially supported by the Center for
Interdisciplinary Research (ZiF) in Bielefeld, Germany, within the
framework of the cooperation group on ``Discrete and continuous models
in the theory of networks''.}

\begin{document}

\begin{abstract}
  We present a systematic collection of spectral surgery principles for the
  Laplacian on a compact metric graph with any of the usual 
  vertex conditions (natural, Dirichlet or $\delta$-type), which show how 
  various types of changes of a local or localised nature to a graph impact on 
  the spectrum of the Laplacian.  Many of these principles are entirely
  new; these include ``transplantation'' of volume within a graph
  based on the behaviour of its eigenfunctions, as well as
  ``unfolding'' of local cycles and pendants. In other cases we establish 
  sharp generalisations, extensions and refinements of known
  eigenvalue inequalities resulting from graph modification, such as
  vertex gluing, adjustment of vertex conditions and introducing new
  pendant subgraphs.

  To illustrate our techniques we derive a new eigenvalue estimate
  which uses the size of the doubly connected part of a compact 
  metric graph to estimate the lowest non-trivial eigenvalue of the 
  Laplacian with natural vertex conditions. This quantitative isoperimetric-type
  inequality interpolates between two known estimates --- one assuming
  the entire graph is doubly connected and the other making no
  connectivity assumption (and producing a weaker bound) --- and
  includes them as special cases.
\end{abstract}

\maketitle

\tableofcontents

%%%%%%%%%%%%%%%%%%%%%%%%%%%%%%%%%%%%%%%%%%%%%%%%%%%%%%%%%%%%%%%%%%%%%%%%%%%%%%%%%%%%%%%%%

\section{Introduction}
\label{sec:intro}

The topic of eigenvalue estimates for various kinds of differential
operators -- especially for the Laplacian -- is a well-established
one. The usual goal is to deduce estimates depending only on  
simple geometric properties of the underlying object, most commonly 
a domain or a manifold, without having to try to compute the eigenvalues 
or eigenfunctions explicitly.

In the last decade there has been a pronounced growth of interest in
such eigenvalue estimates in the particular case of quantum graph
Laplacians; we refer in particular to
\cite{AiScSmWa_arx17,Ari16,BanLev_ahp17,BeKeKuMu_arx17,BonDed12,Fri_aif05,16PAMS,KeKuMaMu_ahp16,KMN,KurNab_jst14,Roh_pams17,RohSei_arx18}
for the Laplacian and its linear generalisations, and
\cite{AdaSerTil15,AdaSerTil_cv15,AdaSerTil_cmp17} among others for a
class of nonlinear Schr\"odinger operators on graphs.  Particular
attention has been paid to the lowest non-trivial eigenvalue 
because, for example, it gives the optimal rate of convergence to the 
equilibrium of solutions to the corresponding heat equation. In the case of 
natural vertex conditions (also known as standard and Kirchhoff-continuity 
in the literature), this eigenvalue equals the spectral gap, and in addition 
to its role in the heat equation, 
predicts bifurcation of the ground state from constant in nonlinear 
Schr\"odinger equation on the graph \cite{Ada16,MarPel15}, and has a
non-trivial relation to graph connectivity: for the discrete
Laplacian the spectral gap is also referred to as the {\it algebraic
  connectivity} \cite{Fie_cmj73}.  The corresponding eigenvectors, also 
called {\it Fiedler vectors}, are interesting from the point of view of 
clustering problems \cite{Lux07} and the hot spots conjecture \cite{Ev11}, 
and similar applications are expected in the case of quantum graph 
Laplacians \cite{KeKuLeMu18,KeRo18}.  Higher eigenvalue
estimates can play a role in the study of spectral minimal
partitions \cite{BaBeRaSm12,KeKuLeMu18}, nodal count
statistics \cite{Ber_cmp08,BanBerWey_jmp15}, \cite[Theorem 7.8]{Ber_cm17} 
and even quantum chaos \cite{BeBoKe01}.

It has become increasingly clear that a central role in estimating
eigenvalues is played by what we shall call \emph{surgery operations}:
basic changes to the geometry of a graph, such as lengthening an edge
or gluing together vertices, that have a predictable effect on one or
several eigenvalues.  The current work is dedicated to developing and cataloguing
these tools in their sharpest form, as well as illustrating their potency with
some carefully selected applications (further applications that have
been discovered in the course of preparing this manuscript will be
published elsewhere).  Wherever feasible we treat general eigenvalues
(i.e., not just the lowest ones) and more general vertex conditions,
in particular Dirichlet and Robin-type couplings.  Much care is
dedicated to treating the cases of extremality, i.e. the cases when an
inequality becomes an equality. 
However, we restrict ourselves to the case of \emph{compact} 
graphs, that is, graphs with a finite number of edges, each of finite length; 
this guarantees that our Laplacian operators have discrete spectrum.
Most techniques presented here should be extendable, with the same proofs, to 
the eigenvalues below the essential spectrum of Schr\"odinger-type 
operators on non-compact graphs.

After introducing our notation and recalling the basic definitions and
properties of quantum graphs in Section~\ref{sec:prelim}, we will
collect all the surgical principles in Section~\ref{sec:tools},
classifying them into three types.  Section~\ref{sec:vertex-tools}
treats operations related to the vertices: cutting and gluing them, or
changing the vertex condition.  Notably, the main theorem of this
section, Theorem~\ref{thm:changing_vc}, contains a new, complete
characterisation of equality when cutting and gluing vertices, which
in turn uses a characterisation of equality in the Courant--Fischer
minimax principles which seems to be very little known (see
Lemma~\ref{lem:min-max-equality}).  In
Section~\ref{sec:op_increase_volume} we look at operations that
increase the total volume (length) of the graph.  Here, we study the
effect of inserting a graph at a given vertex, of which previously
studied operations of attaching a pendant graph and lengthening an
edge are special cases.  In Section~\ref{sec:op_transfer_volume} we
consider operations that transfer edges of the graph from one part to
another: in this case, we are primarily interested in the 
lowest non-trivial eigenvalue. Among others, we introduce the notions of 
transplantation and unfolding of edges.  
Theorem~\ref{thm:transferring_vol} summarises the spectral consequences 
of these operations, which for the most part we believe to be entirely new.  
In Section~\ref{sec:limitations} we give a few examples illustrating the 
necessity of our assumptions and indicating possible further extensions and 
generalisations.

Some of the surgery operations we consider have appeared elsewhere,
but in weaker forms.  In \cite{BerKuc_sg12,KMN,KurNab_jst14,Roh_pams17},
eigenvalue estimates were derived for certain basic surgical
operations of quantum graphs, namely gluing vertices and attaching
edges; the recent preprint \cite{RohSei_arx18} deals with
these operations for more general self-adjoint vertex conditions. 
However, even with these operations, to date little attention has been 
paid to characterising the cases of equality.  More sophisticated surgery
operations where the set of edges is changed were investigated in
\cite{BeKeKuMu_arx17,KeKuMaMu_ahp16,BanLev_ahp17}, often relying on 
the symmetrisation technique first applied to quantum graphs by
L. Friedlander \cite{Fri_aif05} (see \cite{16PAMS} for a comparison with
other techniques).  Some estimates of the types not considered
here, but which could be derived from the more fundamental results
collected in Section~\ref{sec:tools}, appeared as Lemma~4.2 and 
Lemma~4.5 of \cite{BanBerWey_jmp15} and as Theorem~1.3 (edge switching
transformation) of \cite{AiScSmWa_arx17}.

The proofs of all our surgery results are the subject of
Section~\ref{sec:proofs}.  The remaining sections are devoted to a
demonstration of what can be achieved using the new (or significantly
improved) surgery principles.  Our main goal in this direction is a
sharpened isoperimetric-type inequality for the first non-trivial
eigenvalue (spectral gap) of the Laplacian with natural vertex 
conditions.  For a general compact metric graph of total length $L$,
this eigenvalue was shown by Nicaise \cite{Nic_bsm87} to be no smaller
than $\pi^2/L^2$, with equality if and only if $\Graph$ is a path
(i.e., interval); see \cite{Sol02,Fri_aif05,KurNab_jst14} for further
proofs. Recently, Band and L\'evy \cite{BanLev_ahp17} obtained a
stronger lower bound under the assumption that the graph is doubly
connected: the non-trivial eigenvalue is no lower than $4\pi^2/L^2$;
see also~\cite{BeKeKuMu_arx17} for a sharper estimate in the case of
higher connectivities.

Here we will obtain results that interpolate between these two
inequalities, while containing them as special, limit cases: we will prove
lower bounds on the spectral gap in terms of the size of the
\emph{doubly connected part} of the graph $\Graph$.  This is the
largest subset of $\Graph$ each of whose connected components is
itself doubly connected (see
Definition~\ref{def:doubly-connected-part} for more details).  Our main
theorems in this context are Theorem~\ref{thm:doubly-connected} and
Theorem~\ref{thm:doubly-connected-component}. The former bounds the
spectral gap of a graph $\Graph$ from below in terms of a
\emph{dumbbell graph} (see Definition~\ref{def:examples-of-graphs})
with the same (or smaller) sized doubly connected component. The
latter gives a complementary but equally sharp bound in terms of the 
length of the longest cycle in $\Graph$, leading to a comparison with a
well-chosen \emph{tadpole graph}.
In fact, these results appear to be considerably stronger than the 
best available analogues for discrete Laplacians, cf.\
Proposition~\ref{prop:discrete-tadpole} and
Corollary~\ref{cor:circumference} for more details.

Theorems~\ref{thm:doubly-connected}
and~\ref{thm:doubly-connected-component} may be viewed as 
\emph{quantitative isoperimetric inequalities}, which
make an appearance in spectral geometry of domains in higher
dimensions \cite{BraDeP17}.  Such inequalities give not just a sharp
bound on an eigenvalue in terms of the total volume (or in our case length 
of the graph), but also a correction term which takes into account some
measure of the difference of a given domain from the optimising one
(the size of the doubly connected component or the length of the longest 
cycle, in our case). The
interested reader may wish to combine our results with the results of
\cite{Roh_pams17}, in which complementary improved estimates are
obtained for tree graphs.

Along the way to our main applications, in Section~\ref{sec:pumpkins}
we will give several smaller, more specialised applications of our
techniques to so-called \emph{pumpkin chain} graphs, which give
demonstrations of how individual surgery techniques can be used and
combined to manipulate special classes of graphs.  As a simple example, 
the ``unfolding of edges'' principle from
Section~\ref{sec:op_transfer_volume} shows that the spectral gap of
any tadpole or dumbbell graph is a monotonically decreasing function
of the length of the loop(s) for fixed total graph length, without the need 
for explicit calculations based on the secular equation for the eigenvalues; 
see Propositions~\ref{prop:pumpkin-on-a-stick}
and~\ref{prop:pumpkin-dumbbell} for more details. The examples we
present in this section have also been chosen because they provide
exactly the auxiliary results needed for the proofs of the
isoperimetric inequalities in Section~\ref{sec:sizeof}.

Another application of the techniques developed here will appear in the forthcoming 
paper \cite{Ken18}; we intend to present further applications elsewhere.

%%%%%%%%%%%%%%%%%%%%%%%%%%%%%%%%%%%%%%%%%%%%%%%%%%%%%%%%%%%%%%%%%%%%%%%%%%%%%%%%%%%%%%
\section{Preliminaries and notation}
\label{sec:prelim}

We shall begin with our notation. Let
$\Graph = (\mathcal{V},\mathcal{E})$ be a graph with vertex set
$\mathcal{V}$ and edge set $\mathcal{E}$.  We turn it into a metric
graph by identifying each edge $e\in \mathcal{E}$ with the interval
$[0,|e|]$, where $|e|>0$ is the length of the edge.  We will denote
vertices by letters such as $v$, $u$ and $w$; we shall write
$e \sim v$ to mean that the vertex $v$ is incident with the edge
$e$. In a slight abuse of notation, we will also write $e \sim vw$ to
mean that $e$ is an edge connecting $v$ and $w$.  We will 
always assume the graph is \emph{compact}, by which we mean that 
there is a finite number $E = |\mathcal{E}|$ of edges, each edge of finite 
length; this terminology is in keeping with \cite{BerKuc_msm13}. 
We denote by $|\Graph|$ the total length of the graph, i.e.\ the sum of the 
lengths of the edges of the graph. 
$\Graph$ is allowed to contain loops as well as multiple edges between 
given pairs of vertices.  Often, but not always, we will assume the graph is 
connected; whenever we do so we will state this assumption explicitly.

We shall be interested in the spectrum of the Laplacian $-\Delta$ on
$\Graph$ equipped at each vertex with one of the following vertex
conditions: more precisely, the operator is $-\frac{d^2}{dx^2}$ on
each edge applied to functions which are in the Sobolev space $H^2(e)$
on each edge $e\in \mathcal E$, and which satisfy
\begin{itemize}
\item \emph{natural}\footnote{Also known as standard, Neumann, continuity/Kirchhoff, or
    Neumann--Kirchhoff; observe that on a degree-one vertex natural
    conditions agree with common Neumann ones.} conditions on a subset
  $\mathcal{V}_N \subset \mathcal{V}$: at $v \in \mathcal{V}_N$, we
  demand continuity of the functions and that the sum of the normal
  derivatives at each vertex is zero (the Kirchhoff or ``current
  conservation'' condition):
  \begin{displaymath}
    \sum_{e \sim v} \nd{f}{e}{v} =0
  \end{displaymath}
  for each $v \in \mathcal{V}_N$, where $\nd{f}{e}{v}$ is the normal
  derivative of $f$ on $e$ at $v$, with $\nu = \nu_e(v)$ pointing
  \emph{outward} (away from the edge $e$, towards the vertex);
\item \emph{Dirichlet} conditions on a subset
  $\mathcal{V}_D \subset \mathcal{V}$: at $v \in \mathcal{V}_D$, any
  functions in the domain of $\Delta$ should take on the value zero.
\item \emph{$\delta$ (or Kirchhoff--Robin)} conditions on a subset
  $\mathcal{V}_R \subset \mathcal{V}$: for each $v\in \mathcal{V}_R$
  there is a $\gamma = \gamma(v) \in \R$, $\gamma\neq 0$ such that the
  functions $f \in D(\Delta)$ are continuous at $v$ and the
  derivatives at $v$ satisfy
  \begin{equation}
    \label{eq:robin_def}
    \sum_{e \sim v} \nd{f}{e}{v} + \gamma f(v) = 0,
  \end{equation}
  where, again, $\nu$ is the outer unit normal to the edge. 
  We sometimes refer to $\gamma$ as the \emph{strength} of the
  $\delta$-condition, or as the \textit{$\delta$-potential} at $v$. We
  will write
  \begin{displaymath}
    \mathcal{V}_R^- := \left\{v \in \mathcal{V}_R: \gamma(v) < 0\right\}, \qquad
    \mathcal{V}_R^+ := \left\{v \in \mathcal{V}_R: \gamma(v) > 0\right\},
  \end{displaymath}
  so that $\mathcal{V}_R = \mathcal{V}_R^- \cup \mathcal{V}_R^+$.
\end{itemize}
We thus assume that
$\mathcal{V} = \mathcal{V}_N \cup \mathcal{V}_D \cup \mathcal{V}_R$, 
where any of the three sets on the right-hand side may be empty.
We see immediately that the $\delta$-condition with $\gamma(v)=0$
corresponds to the natural condition.  Furthermore, Dirichlet
conditions correspond formally to $\delta$-conditions of strength
$\gamma=\infty$.  This correspondence may be made rigorous
\cite{BerKuc_sg12}, although we will not need it here.  We refer to
\cite[Chapter~1]{BerKuc_msm13} and \cite[Chapters~2 and~3]{Mug14} for
more details regarding the definitions and elementary properties of
function spaces on graphs and the Laplace-type operators defined on
them and also to \cite{Ber_cm17} for an elementary introduction to
spectral properties of quantum graphs.

\begin{remark}
  \label{rem:dummy}
  Any point in the interior of an edge may be declared to be a vertex
  of degree 2 with natural conditions without affecting the spectral 
  properties of the operator (cf.\ the discussion just after Assumption~3.1
  in \cite{BeKeKuMu_arx17}).  We will refer to this as introducing a
  ``dummy'' vertex.  Conversely, any vertex $v$ of degree 2 with
  natural conditions may be suppressed.  Likewise, the operator is
  not modified if a subset of the elements of $\mathcal V_D$
  are identified to form one single Dirichlet vertex.
\end{remark}

The corresponding quadratic form is given by the Dirichlet integral
\begin{equation}
\label{eq:dirichlet-form}
	a(f)=\int_{\Graph} |f'|^2\,\textrm{d}x + \sum_{v \in \mathcal{V}_R}\gamma(v) |f(v)|^2,
\end{equation}
with the domain formed by all functions from the Sobolev space
$ H^1(e_j) $ on every edge, which are in addition continuous at the
vertices, and zero at all vertices in $\mathcal{V}_D$. If
$\mathcal{V}_D = \emptyset$, then this form domain shall be denoted by
$H^1(\Graph)$, as is customary; if $\mathcal{V}_D \neq \emptyset$ we shall 
denote it by $H^1_0 (\Graph;\mathcal{V}_D)$, or just $H^1_0 (\Graph)$ if there 
is no danger of confusion about which set $\mathcal{V}_D$ is to be
understood. We will write $C(\Graph)$ for the space of continuous functions 
on $\Graph$, i.e., continuous on every edge, and at every vertex, so that 
$H^1 (\Graph), H^1_0 (\Graph;\mathcal{V}_D) \subset C(\Graph)$ for 
any $\mathcal{V}_D$.

As is well known, under this set of assumptions the Laplacian described 
above is self-adjoint and semi-bounded, and has trace class resolvent; in 
particular, its spectrum consists of a sequence of real eigenvalues of finite 
multiplicity, which we denote by
\begin{equation}
  \label{eq:eigenvalues}
  \lambda_1 \leq \lambda_2 \leq \lambda_3 \leq \ldots,
\end{equation}
where each is repeated according to its multiplicity. The corresponding
eigenfunctions may be chosen to form an orthonormal basis of
$L^2(\Graph)$, and may additionally without loss of generality all be
chosen real, as we shall do without further comment throughout the
paper. We will tend to use letters such as $\psi$ to denote
eigenfunctions, and will refer to an eigenfunction corresponding to
eigenvalue $\lambda$ as a $\lambda$-eigenfunction. By standard
Kre\u{\i}n--Rutman theory, if $\Graph$ is connected, the first
eigenvalue $\lambda_1$ is always simple and the corresponding
eigenfunction, unique up to scalar multiples, can be chosen strictly
positive a.e.; in fact it can be shown to be strictly positive
everywhere outside $\mathcal{V}_D$.

The eigenvalues, and their eigenfunctions, depend on both the metric
and topological structure of the graph and the $\delta$-coupling
parameters; for brevity, in keeping with the custom of 
considering a quantum graph to be a triple consisting of a metric graph, 
a differential operator and vertex conditions, we shall write $\lambda_k 
= \lambda_k(\Graph)$ to reflect
this dependence and correspondingly $\spec(\Graph) = \{\lambda_k
(\Graph):k\geq 1\}$ for the spectrum. To distinguish the important
case of only natural conditions, $\mathcal{V} = \mathcal{V}_N$, where
$\lambda_1 = 0$, we shall often use the superscript ``$N$'' if we wish
to emphasise the presence of only natural conditions at the vertices:
in this case, we will write (assuming $\Graph$ is connected),
\begin{equation}
\label{eq:natural-eigenvalues}
	0 = \lambda_1^N < \lambda_2^N \leq \lambda_3^N \leq \ldots
\end{equation}
Note that the more general notation \eqref{eq:eigenvalues} also covers
this case. We will likewise use the notation
\begin{equation}
\label{eq:dirichlet-eigenvalues}
	0 < \lambda_1^D < \lambda_2^D \leq \lambda_3^D \leq \ldots
\end{equation}
in place of the $\lambda_k$ if we want to emphasise that there is at
least one Dirichlet vertex. In this case all non-Dirichlet vertices
are assumed to be equipped with natural conditions. 
We mention explicitly that if one or more vertices are equipped 
with a negative coupling condition, i.e., if $\mathcal{V}_R^- \neq \emptyset$, 
then there may be negative eigenvalues; in this case, it is also possible that 
$\lambda_1 (\Graph) = 0$ but the corresponding eigenfunction is not identically 
constant. See also \cite{ExJe}.

At any rate, these eigenvalues are generally not explicitly computable
as the relevant secular equation is transcendental even on graphs as
simple as stars, cf.~\cite[\S~4, Thm.]{Bel_laa85} 
or~\cite[Sec 5]{Ber_cm17}. Instead, in the tradition of
spectral geometry, we can attempt to understand how the eigenvalues
change depending on the underlying graph.  For example, one may be
interested to know if a certain graph minimises or maximises a given
eigenvalue among all graphs with certain fixed geometric quantities
(length, diameter etc.), or if, for a given graph with reflection
symmetry, the first non-trivial eigenfunction is symmetric or
anti-symmetric.  To answer such questions one needs to be able to make
comparisons.  One of the aims of this work is to catalogue the types
of alterations to the geometry of a graph that affect the eigenvalue(s)
in a predictable way; we shall generally refer to this as
``surgery''. 

We will often be concerned with the first non-trivial eigenvalue in
particular.  The starting point is always the variational
characterisation of $\lambda_1 (\Graph)$ and $\lambda_2 (\Graph)$,
namely
\begin{equation}
\label{eq:varchar-principal}
	\lambda_1 (\Graph) = \inf \left\{\frac{a(f)}{\int_\Graph |f|^2\,\textrm{d}x}: 0 \neq f \in H^1 (\Graph) \text{ or } 
	H^1_0 (\Graph; \mathcal{V}_D) \right\},
\end{equation}
($H^1$ or $H^1_0$ as appropriate, and where $a$ is given by \eqref{eq:dirichlet-form}). 
In the case of a connected graph and pure natural
conditions, $\mathcal{V}=\mathcal{V}_N$, where $\lambda_1^N=0$ and the
associated eigenfunction is simply the constant function, it is more
natural to consider $\lambda_2^N$, which in this case, since $a(f)$
reduces to the integral of the derivatives, is given by
\begin{equation}
\label{eq:varchar}
	\lambda_2^N (\Graph) = \inf \left\{ \frac{\int_{\Graph} |f'|^2\,\textrm{d}x}{\int_{\Graph} |f|^2\,\textrm{d}x} :
	0 \neq f \in H^1 (\Graph) \text{ and } \int_{\Graph} f\,\textrm{d}x = 0 \right\};
\end{equation}
the condition $\int_{\Graph} f\,\textrm{d}x = 0$ represents the orthogonality 
in $L^2(\Graph)$ of $f$ to the constant eigenfunctions of $\lambda_1^N(\Graph)$. 
We therefore introduce the following notation.
\begin{definition}
  \label{def:mu}
  Given a \emph{connected} graph $\Graph$ with vertex conditions of the types listed
  above, we denote by $\mu (\Graph)$ its first eigenvalue
  $\lambda_1 (\Graph)$ if
  $\mathcal{V}_D \cup \mathcal{V}_R \neq \emptyset$, or its second
  eigenvalue $\lambda_2^N (\Graph)$ if
  $\mathcal{V}_D \cup \mathcal{V}_R = \emptyset$ (and thus
  $\lambda_1^N (\Graph) = 0$).
\end{definition}

The higher eigenvalues may be characterised by corresponding minimax
and maximin principles of Courant--Fischer type, cf.~\eqref{eq:minimax_full} 
and~\eqref{eq:maximin_full}. In all cases, for a
given function $f \in H^1(\Graph)$, the quotient appearing in
\eqref{eq:varchar-principal} or \eqref{eq:varchar} is called the
\emph{Rayleigh quotient} of $f$, and equality is achieved if and only
$f$ is an eigenfunction associated with the corresponding eigenvalue.

%%%%%%%%%%%%
\subsection{Examples of graphs}

Here we introduce terminology for several classes of graphs that come
up often in applications, in particular, as sharp cases of eigenvalue
estimates. In all cases, we assume that any vertices of 
degree two are suppressed, cf.~Remark~\ref{rem:dummy}.

\begin{definition}
\label{def:examples-of-graphs}
  \begin{enumerate}
  \item The \emph{path graph} is a graph consisting of two vertices and one edge.
  \item A \emph{loop} is a graph consisting of one edge whose
    endpoints are the same vertex. 
  \item The \emph{star graph} is a graph consisting of $E$ edges all
    having exactly one vertex in common. We refer to this graph 
    as an \emph{$E$-star} to emphasise the number of edges.
  \item The \emph{tadpole graph} (also called ``lasso graph'') is a graph
    consisting of a loop attached to a single edge.
  \item The \emph{flower graph} consists of $E$ loops attached to a
    single vertex.  A special case is the \emph{figure-8 graph}
    which has two loops.
  \item The \emph{dumbbell graph} has three edges and two vertices; 
    it consists of an edge joining two loops.
  \item The \emph{pumpkin graph} (also called ``mandarin graph'') is 
    a graph consisting of two vertices and $E$ parallel edges of possibly 
    different lengths running between them. We will also write 
    \emph{$E$-pumpkin} if we wish to emphasise the number of edges.
  \item In examples (3), (5) and (7), the prefix \emph{equilateral} 
    is applied if all edges have the same length.
  \item A \emph{pumpkin chain} is built up out of pumpkins glued
    sequentially at the vertices.  More precisely, a
    \emph{$[m_1,\ldots,m_n]$-pumpkin chain} is a connected graph
    consisting of $n+1$ vertices $v_1,\ldots, v_{n+1}$ and, for each  $k=1,\ldots,n$, 
    $m_k$ parallel edges running between the vertices $v_k$ and $v_{k+1}$. 
    We will denote the subpumpkin consisting of the vertices $v_k$ and $v_{k+1}$ and 
    the $m_k$ edges joining them by $\mathcal{P}_k$ and refer to it as the 
    $k$th constituent pumpkin of the chain. 
    Both vertices $v_1$ and $v_{n+1}$, which we sometimes denote by $v_-$ and
    $v_+$, and the pumpkins $\mathcal{P}_1$ and $\mathcal{P}_n$ attached to them, 
    shall be called \emph{terminal}.  We shall call a constituent pumpkin of a pumpkin
    chain \emph{trivial} if it only has one edge, and
    \emph{non-trivial} otherwise. A pumpkin chain is \emph{locally equilateral} if 
    each constituent pumpkin is equilateral (i.e., all edges have 
    the same length), although the lengths of edges in different pumpkins 
    may be different.
  \end{enumerate}
\end{definition}

We remark that the tadpole graph can be viewed as the $[1,2]$- or $[2,1]$-pumpkin
chain, the figure-8 graph is the $[2,2]$-pumpkin chain, and the dumbbell 
is the $[2,1,2]$-pumpkin chain.

%%%%%%%%%%%%%%%%%%%%%%%%%%%%%%%%%%%%%%%%%%%%%%%%%%%%%%%%%%%%%%%%%%%%%
\section{A surgeon's toolkit}
\label{sec:tools}

The aim of this section is to catalogue the effects of elementary
surgical transformations on the spectrum of the graph.  In each case we
will first describe how the connectivity and the metric features of
$\Graph$ should be changed in order to produce a new graph
$\NewGraph$, and then how functions on $\Graph$ can be lifted to
functions on $\NewGraph$ by canonically assigning conditions in the
vertices of $\NewGraph$. Here and throughout we adopt the usual
conventions for arithmetic involving $\infty$, e.g.,
$\gamma + \infty = \infty$ if $\gamma \in \R$;
$\infty + \infty = \infty$ and so on.

We reiterate that all graphs throughout this section are taken to
satisfy the assumptions described at the beginning of 
Section~\ref{sec:prelim}; all 
results of this section will be proved in Section~\ref{sec:proofs}.

%%%%%%%%%%%%
\subsection{Operations changing vertex conditions}
\label{sec:vertex-tools}

We have already remarked that cutting through/gluing together
Dirichlet vertices is a trivial operation (Remark~\ref{rem:dummy});
let us now consider what happens for more general vertex conditions.
Recall that all three kinds of vertex conditions we are considering
can be regarded as $\delta$-conditions with parameter
$\gamma\in (-\infty,\infty]$.

\begin{definition}[Gluing vertices]\label{def:gluing}
  Let $\NewGraph$ be obtained from $\Graph$ by identifying the
  vertices $v_1,v_2,\ldots,v_m$ to obtain a new vertex $ v_0.$ If the
  $\delta$-conditions with the strengths
  $\gamma(v_j) \in (-\infty,\infty]$ were imposed at $v_j$,
  $j=1,\ldots,m$, then the new vertex $v_0$ is to be equipped with the
  $\delta$-condition of strength
  \begin{equation}
    \label{eq:adding_delta_strengths}
    \gamma(v_0) = \gamma(v_1) + \gamma(v_2) + \ldots + \gamma(v_m).  
  \end{equation}
  We will refer to the corresponding surgery transformation as {\bf
    gluing vertices}.
\end{definition}

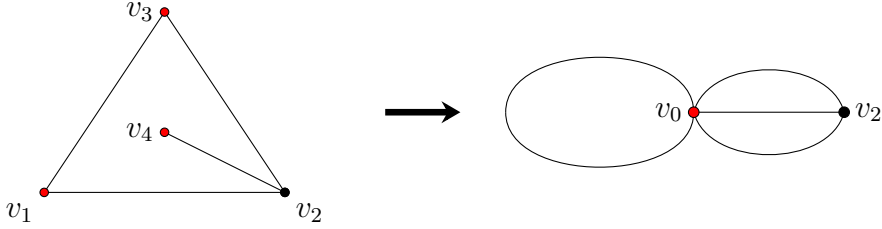
\begin{figure}[H]
\label{fig:joining-vertices-1}
\begin{minipage}[l]{5cm}
\begin{tikzpicture}[scale=0.8]
\coordinate (a) at (0,0);
\coordinate (b) at (4,0);
\coordinate (c) at (2,3);
\coordinate (d) at (2,1);
\draw (a) -- (b);
\draw (c) -- (b);
\draw (a) -- (c);
\draw (b) -- (d);
\draw[fill=red] (0,0) circle (2pt);
\draw[fill] (4,0) circle (2pt);
\draw[fill=red] (2,3) circle (2pt);
\draw[fill] (4,0) circle (2pt);
\draw[fill=red] (2,1) circle (2pt);
\node at (a) [anchor=north east] {$v_1$};
\node at (b) [anchor=north west] {$v_2$};
\node at (c) [anchor=east] {$v_3$};
\node at (d) [anchor=east] {$v_4$};
\end{tikzpicture}
\end{minipage}
\begin{minipage}{1.5cm}
\begin{tikzpicture}%[scale=0.8]
\draw[-{Stealth[scale=0.5,angle'=60]},line width=2.5pt] (4,0) -- (5,0);
\end{tikzpicture}
\end{minipage}
\begin{minipage}{6cm}
\begin{tikzpicture}%[scale=0.8]
\coordinate (c) at (1,0);
\coordinate (d) at (3.5,0);
\coordinate (e) at (5.5,0);
\coordinate (f) at (5.4,1);
\coordinate (g) at (5.4,-1);
\draw[bend left=90]  (c) edge (d);
\draw[bend right=90]  (c) edge (d);
\draw (d) -- (e);
\draw[bend left=75] (d) edge (e);
\draw[bend right=75] (d) edge (e);
\draw[fill=red] (3.5,0) circle (2pt);
\draw[fill] (5.5,0) circle (2pt);
\node at (d) [anchor=east] {$v_0$};
\node at (e) [anchor=west] {$v_2$};
\end{tikzpicture}
\end{minipage}
\caption{The graph $\NewGraph$ (right) is obtained from $\Graph$ (left) 
by gluing the vertices $v_1,v_3,v_4$. Conversely, the graph on the left is 
one of the possible graphs obtainable from the graph on the 
right upon cutting through $v_0$, in this case producing the vertices 
$v_1,v_3,v_4$.}
\end{figure}

\begin{definition}[Cutting through vertices]
  \label{def:cutting}
  The converse operation to gluing the vertices, i.e.\ splitting a
  vertex $v_0$ into $m$ vertices $v_1,v_2,\ldots,v_m$ (called
  \emph{descendants} of $v_0$) with $\delta$-type conditions
  satisfying \eqref{eq:adding_delta_strengths} is called
  \textbf{cutting through the vertex $ v_0$}.

If in addition we are given a certain function $\psi$ satisfying
$\delta$-conditions at $v_0$ we can choose $\gamma(v_1),\ldots,\gamma(v_m)$ 
so that the same function $\psi$ satisfies the conditions
at the new vertices.  Namely, we let
\begin{equation}
  \label{eq:gamma_choice_along_psi}
  \gamma(v_i) = - \frac{1}{ \psi(v_0)} \sum_{e \sim v_i} 
  \nd{\psi}{e}{v_0},
  \qquad
  i=1,\ldots,m
\end{equation}
where the summation is over the edges that are attached to the
relevant descendant of $v_0$. In particular, when $\psi(v_0)=0$, we
impose Dirichlet conditions at all vertices $v_1,\ldots,v_m$.  The corresponding
surgery transformation will be called {\bf cutting through the vertex
  $ v_0$ along the function $ \psi$}.
\end {definition}

In general, when cutting through a vertex $v_0$ the
edges incident with it may be assigned to the new vertices $v_1,\ldots,v_m$ 
in several possible ways. In other words, even if $m$ is fixed, the graph 
$\NewGraph$ created from $\Graph$ by cutting through a 
vertex is in general not unique.

\begin{remark}
\label{rem:canonical-identification}
Suppose $\NewGraph$ is created from $\Graph$ by gluing $m$ vertices 
$v_1,\ldots,v_m$ to form $v_0$. Then there is a natural isomorphism 
$\Phi: L^2(\Graph) \to L^2(\NewGraph)$, that is, we can make the 
identification
\begin{displaymath}
	L^2(\Graph) \simeq \bigoplus_{e\in \mathcal{E}}L^2(0,\ell_e) 
	\simeq L^2 (\NewGraph),
\end{displaymath}
where $\mathcal{E}$ is the common set of edges of the two graphs. Moreover, if 
$f \in C(\Graph)$ (in particular if $f \in H^1(\Graph)$) satisfies $f(v_1)=\ldots= 
f(v_m)$, then also $\Phi(f) \in C(\NewGraph)$ (corresp.~in $H^1(\NewGraph)$). 
In this case, we identify $\Phi(f)$ with $f$ and speak of a ``canonical 
identification'' of the two; in this way, $C(\NewGraph)$ and $H^1 (\NewGraph)$ 
may be regarded as subspaces of $C(\Graph)$ and $H^1(\Graph)$ of codimension 
$m-1$, respectively. From now on, we will \emph{always} make this identification, 
that is, whenever we glue together vertices, we will suppress the notation $\Phi$ 
and identify $C(\NewGraph)$ and $H^1(\NewGraph)$ (and its subspaces) with 
subspaces of $C(\Graph)$ and $H^1(\Graph)$ (and its subspaces), respectively.
\end{remark}

By regarding cutting through a vertex as removing continuity conditions 
from it, if one wishes one may also view this operation as changing the 
conditions at a single, ``generalised'' vertex.  Another example of 
changing conditions is the operation of varying the $\delta$-potential at 
a vertex $v$.  Both types of operations are finite rank perturbations of the 
quantum graph operator and result in the interlacing of the eigenvalues of 
the two graphs.

\begin{theorem}[Changing vertex conditions]
  \label{thm:changing_vc}
  If the graph $\NewGraph$ is obtained from $\Graph$ by either
  \begin{enumerate}
  \item \label{item:gluing_vertices} gluing two vertices,
  \\ \indent
  or 
  \item \label{item:changing_delta} increasing the strength of the 
    $\delta$-condition at a
    single vertex from $\gamma$ to $\gamma' \in (\gamma,\infty]$,
  \end{enumerate}
  then their eigenvalues satisfy the interlacing inequalities
  \begin{equation}
    \label{eq:interlacing}
    \lambda_{k}(\Graph) \leq \lambda_{k}(\NewGraph) 
    \leq \lambda_{k+1}(\Graph) \leq \lambda_{k+1}(\NewGraph),
    \qquad k \geq 1.
  \end{equation}
  If a given value $\Lambda$ has multiplicities $m$ and $\widetilde{m}$
  in the spectra of $\Graph$ and $\NewGraph$, respectively, then 
  $|m-\widetilde{m}| \leq 1$ and, with the identification in 
  Remark~\ref{rem:canonical-identification}, the intersection of the 
  respective $\Lambda$-eigenspaces has dimension $\min(m,\widetilde{m})$.
\end{theorem}

The inequality of type \eqref{eq:interlacing} is both well known (cf.,
e.g., \cite[Theorems~3.1.8 and~3.1.10]{BerKuc_msm13}, \cite{KMN},
\cite{RohSei_arx18}) and easy to obtain from the variational
principles.  However, the conclusive treatment of the cases of
equality is, to the best of our knowledge, new.  It is also of
tremendous value for characterising the extremal cases of the
inequalities contained in the subsequent sections of the present
paper.  For example, the following simple observation will be used at
least twice.

\begin{remark}
  \label{rem:catch_up}
  If, in the setting of Theorem~\ref{thm:changing_vc}, the eigenvalue
  arrangement is
  \begin{equation}
    \label{eq:catch_up}
    \lambda_{k}(\Graph) < \lambda_{k}(\NewGraph) 
    = \lambda_{k+1}(\Graph) =: \Lambda,
  \end{equation}
  then every $\lambda_{k+1}(\Graph)$-eigenfunction of $\Graph$ is also
  an eigenfunction of $\NewGraph$.  Indeed, it is easy to see that
  \eqref{eq:interlacing} will imply that $\min(m, \widetilde{m})=m$
  for the eigenvalue $\Lambda$ and the eigenspace inclusion follows.
\end{remark}

We will also often use the following special case of equality.

\begin{corollary}[Gluing level points]
  \label{cor:joining_points}
  Suppose $v_1,\ldots,v_m \in \mathcal{V}(\Graph)$ 
  and for some $k\geq 1$ there exist eigenfunctions $\psi_1,\ldots,\psi_k$ 
  corresponding to $\lambda_1 (\Graph), \ldots , \lambda_k (\Graph)$, 
  respectively, such that
  \begin{displaymath}
	\psi_1(v_1)=\ldots =\psi_1(v_m),
        \ \ldots,\ \psi_k(v_1)=\ldots=\psi_k(v_m).
  \end{displaymath}
  Let $\NewGraph$ be the graph formed from $\Graph$ by 
  gluing $v_1,\ldots,v_m$. Then
  \begin{displaymath}
	\lambda_1(\NewGraph)=\lambda_1(\Graph),
        \ \ldots,\ 
	\lambda_k(\NewGraph)=\lambda_k(\Graph).
  \end{displaymath}
  Moreover, $\psi_1,\ldots,\psi_k$ are 
  eigenfunctions on $\NewGraph$ associated with $\lambda_1(\NewGraph), 
  \ldots, \lambda_k (\NewGraph)$, respectively.
\end{corollary}

\begin{remark}
  In \cite[Proposition~3.1.6]{BerKuc_msm13} it is additionally shown
  that if an eigenvalue $\lambda_k(\Graph)$ is simple, with
  eigenfunction $\psi_k$ normalised to have $L^2$-norm one, it can be
  differentiated with respect to the strength $\gamma$ of the
  $\delta$-parameter at a given vertex $v$.  The value of the
  derivative is
  \begin{equation}
    \label{eq:eig_deriv_alpha}
    \frac{d\lambda_k}{d\gamma} = |\psi_k(v)|^2 
    = \left| \frac1\gamma \sum_{e \sim v}
      \nd{\psi_k}{e}{v} \right|^2.
  \end{equation}
\end{remark}

%%%%%%%%%%%%
\subsection{Operations increasing the volume}
\label{sec:op_increase_volume}
We will now consider operations that expand the 
graph in some way, either by scaling up a part of it or by attaching 
a new subgraph to it.

\begin{definition}[Inserting a graph at a vertex]
  \label{def:insert}
  Let $v_0$ be a vertex of $\Graph$ whose set of incident edges is
  $\{e_1,\ldots,e_k\}$ and let $\GraphH$ be another metric graph.
 Form a new graph
  $\NewGraph$ by removing $v_0$ from $\Graph$ and, for each 
  $i=1,\ldots,k$, attaching edge $e_i$ to some vertex $w=w(i)$ of $\GraphH$ 
  instead.  Let $w_1,\ldots,w_m$, $m \leq k$ be the list of vertices of 
  $\GraphH$ to which an edge has been so attached.  If $v_0$ is equipped 
  with the $\delta$-potential of strength $\gamma(v_0) \in (-\infty,\infty)$, then
  $\delta$-potentials should be placed at the vertices
  $w_1,\ldots,w_m$ in such a way that they sum to $\gamma(v_0)$. We then
  say that $\widetilde\Graph$ is formed by {\bf inserting $\mathcal H$
    into $\Graph$ at $v_0$}.
\end{definition}

\begin{figure}[H]
\begin{minipage}[l]{3cm}
\begin{tikzpicture}[scale=0.6]
\coordinate (a) at (2,0);
\coordinate (b) at (4,2);
\coordinate (c) at (2,4);
\coordinate (d) at (0,2);
\draw (c) -- (d) -- (a);
\draw[red] (a) -- (b) -- (c);
\draw[red] (b) -- (d);
\draw[fill] (a) circle (2pt);
\draw[fill=red] (b) circle (2pt);
\draw[fill] (c) circle (2pt);
\draw[fill] (d) circle (2pt);
\node at (b) [anchor=north west] {$v_0$};
\node at (2.9,3) [anchor=south west] {$e_1$};
\node at (2.9,1) [anchor=north west] {$e_3$};
\node at (2,2) [anchor=south] {$e_2$};
\draw (0.1,3.5) circle (0pt) node[right]{$\Graph$};
\end{tikzpicture}
\end{minipage}
\begin{minipage}[l]{.1cm}
\end{minipage}
\begin{minipage}[l]{3cm}
\begin{tikzpicture}[scale=0.6]
\coordinate (a) at (0,.9);
\coordinate (b) at (0,3.1);
\coordinate (c) at (1.8,2);
\draw (a) -- (b) -- (c) -- (a);
\draw[fill=red] (a) circle (2pt);
\draw[fill=red] (b) circle (2pt);
\draw[fill] (c) circle (2pt);
\node at (a) [anchor=north] {$w_2$};
\node at (b) [anchor=south] {$w_1$};
\draw (0,2) circle (0pt) node[right]{$\GraphH$};
\end{tikzpicture}
\end{minipage}
\begin{minipage}{1.5cm}
\begin{tikzpicture}[scale=0.8]
\draw[-{Stealth[scale=0.5,angle'=60]},line width=2.5pt] (4,0) -- (5,0);
\end{tikzpicture}
\end{minipage}
\begin{minipage}{6cm}
\begin{tikzpicture}[scale=0.6]
\coordinate (a) at (2,0);
\coordinate (c) at (2,4);
\coordinate (d) at (0,2);
\coordinate (a1) at (4.5,0.9);
\coordinate (b1) at (4.5,3.1);
\coordinate (c1) at (6.3,2);
\draw (c) -- (d) -- (a);
\draw[red] (a) -- (a1);
\draw (a1) -- (c1) -- (b1) -- (a1);
\draw[red] (d) -- (b1);
\draw[red] (b1) -- (c);
\draw[fill] (a) circle (2pt);
\draw[fill=red] (b1) circle (2pt);
\draw[fill=red] (a1) circle (2pt);
\draw[fill] (c1) circle (2pt);
\draw[fill] (c) circle (2pt);
\draw[fill] (d) circle (2pt);
\node at (3.1,3.5) [anchor=south] {$e_1$};
\node at (3,1.2) [anchor=north] {$e_3$};
\node at (a1) [anchor=north] {$w_2$};
\node at (b1) [anchor=south] {$w_1$};
\node at (2,2.5) [anchor=south] {$e_2$};
\draw (3.2,1.9) circle (0pt) node[right]{$\NewGraph$};
\end{tikzpicture}
\end{minipage}
\caption{Inserting $\GraphH$ into $\Graph$ at $v_0$, we obtain the
  graph $\NewGraph$ on the right.}
\label{fig:insert1}
\end{figure}

Whenever $w_1=\ldots=w_m$ we have the following special case.

\begin{definition}[Attaching a pendant graph]\label{def:attach}
  Assume that $\Graph$ and $\GraphH$ are given, with one distinguished
  vertex in each graph, say $v_1 \in \Graph$ and $w_1 \in \GraphH$.
  If $\NewGraph$ is formed by gluing together $v_1$ and
  $w_1$ in accordance to Definition~\ref{def:gluing}, we speak of {\bf
    attaching the pendant graph $\GraphH$ to $\Graph$}.
\end{definition}

\begin{figure}[H]
\begin{minipage}[l]{3.8cm}
\begin{tikzpicture}[scale=0.6]
\coordinate (a) at (0,0);
\coordinate (b) at (4,0);
\coordinate (c) at (0,4);
\coordinate (d) at (2,4);
\draw (c) -- (a) -- (b) -- (d) -- (c);
\draw[fill] (0,0) circle (2pt);
\draw[fill] (0,4) circle (2pt);
\draw[fill=red] (4,0) circle (2pt);
\draw[fill] (2,4) circle (2pt);
\node at (a) [anchor=north east] {$v_2$};
\node at (b) [anchor=north] {$v_1$};
\node at (c) [anchor=east] {$v_3$};
\node at (d) [anchor=south] {$v_4$};
\draw (1,2) circle (0pt) node[right]{$\Graph$};
\end{tikzpicture}
\end{minipage}
\begin{minipage}[l]{1.1cm}
\end{minipage}
\begin{minipage}[l]{3.8cm}
\begin{tikzpicture}[scale=0.6]
\coordinate (a) at (0,0);
\coordinate (b) at (0,4);
\coordinate (c) at (3,2);
\draw (a) -- (b) -- (c) -- (a);
\draw[fill=red] (0,0) circle (2pt);
\draw[fill] (0,4) circle (2pt);
\draw[fill] (3,2) circle (2pt);
\node at (a) [anchor=north] {$w_1$};
\node at (b) [anchor=south] {$w_2$};
\node at (c) [anchor=west] {$w_3$};
\draw (.8,2) circle (0pt) node[right]{$\GraphH$};
\end{tikzpicture}
\end{minipage}
\begin{minipage}{1.5cm}
\begin{tikzpicture}[scale=0.8]
\draw[-{Stealth[scale=0.5,angle'=60]},line width=2.5pt] (4,0) -- (5,0);
\end{tikzpicture}
\end{minipage}
\begin{minipage}{6cm}
\begin{tikzpicture}[scale=0.6]
\coordinate (a) at (0,0);
\coordinate (b) at (4,0);
\coordinate (c) at (0,4);
\coordinate (d1) at (2,4);
\coordinate (d2) at (4,4);
\coordinate (e) at (7,2);
\draw (d1) -- (c) -- (a) -- (b);
\draw (d2) -- (e) -- (b);
\draw (b) -- (d1);
\draw (b) -- (d2);
\draw[fill] (0,0) circle (2pt);
\draw[fill] (0,4) circle (2pt);
\draw[fill=red] (4,0) circle (2pt);
\draw[fill] (2,4) circle (2pt);
\draw[fill] (4,4) circle (2pt);
\draw[fill] (7,2) circle (2pt);
\node at (a) [anchor=north east] {$v_2$};
\node at (b) [anchor=north] {$v_1=w_1$};
\node at (c) [anchor=east] {$v_3$};
\node at (d1) [anchor=south] {$v_4$};
\node at (d2) [anchor=south] {$w_2$};
\node at (e) [anchor=west] {$w_3$};
\draw (2.7,3) circle (0pt) node[right]{$\NewGraph$};
\end{tikzpicture}
\end{minipage}
\caption{By gluing together $v_1,w_1$ we can attach the graph $\GraphH$ to
  $\Graph$, thus obtaining the graph $\NewGraph$ on the right.}
\label{fig:attachpend}
\end{figure}

Figure~\ref{fig:insert1} shows an example of inserting a graph at
$v_0$ and Figure~\ref{fig:attachpend} shows an example of attaching a
pendant graph.

\begin{theorem}
  \label{thm:increasing_vol}
  The following operations decrease the given eigenvalues.
  \begin{enumerate}
  \item \label{item:attaching_a_pendant}
    (Attaching a pendant graph) Suppose $\NewGraph$ is formed from
    $\Graph$ by attaching a pendant metric graph $\GraphH$ at a vertex
    $v_0 \in \mathcal{V}(\Graph)$. 
    If, for some $r$ and $k$,
    \begin{equation}
      \label{eq:pendant_eig_condition}
      \lambda_{r}(\GraphH) \leq \lambda_{k}(\Graph),
    \end{equation}
    then 
    \begin{equation}
      \label{eq:pendant_attach_result}
      \lambda_{k+r-1}(\NewGraph) \leq \lambda_k(\Graph).
    \end{equation}
    The inequality in \eqref{eq:pendant_attach_result} is strict if
    the eigenvalue $\lambda_k(\Graph)$ has an eigenfunction which does
    not vanish at $v_0$, $\lambda_k(\Graph)> \lambda_{k-1} 
    (\Graph)$ and $\lambda_k(\Graph)>\lambda_{r}(\GraphH)$.
  \item \label{item:inserting_a_graph} (Inserting a graph at a vertex)
    Suppose $\NewGraph$ is formed by inserting a graph $\GraphH$ at a
    vertex $v_0$ of $\Graph$.  Assume that only natural conditions
    were imposed at the vertices of $\GraphH$ prior to
    insertion.  Then, for all $k$ such that $\lambda_k(\Graph)\geq 0$,
    \begin{equation}
      \label{eq:insert_result}
      \lambda_k(\NewGraph) \leq \lambda_k(\Graph).
    \end{equation}
    The inequality in \eqref{eq:insert_result} is strict if 
    $\lambda_k(\Graph) > \max(0,\lambda_{k-1}(\Graph))$ and the
    eigenvalue $\lambda_k(\Graph)$ has an eigenfunction which does not
    vanish at $v_0$.
  \end{enumerate}
\end{theorem}

\begin{remark}
  \label{rem:increasing_vol}
  An important special case of
  Theorem~\ref{thm:increasing_vol}(\ref{item:attaching_a_pendant}) is
  when the conditions are natural at all vertices of $\GraphH$, while
  $\Graph$ has non-negative spectrum (this holds in particular when
  $\mathcal{V}_R^- = \emptyset$).  In this case
  \begin{equation}
    \label{eq:pendant_natural}
    0 = \lambda_1(\GraphH) \leq \lambda_1(\Graph) \leq \lambda_{k}(\Graph)
  \end{equation}
  and Theorem~\ref{thm:increasing_vol}(\ref{item:attaching_a_pendant})
  with $r=1$ shows that attaching the pendant lowers \emph{all} eigenvalues
  of $\Graph$:
  \begin{equation}
    \label{eq:pendant_attach_result_special}
    \lambda_k (\NewGraph) \leq \lambda_k (\Graph) \qquad \text{for all }
    k \geq 1.
  \end{equation}
  The inequality \eqref{eq:pendant_attach_result_special} was noted in 
  \cite[Theorem~2]{KMN} (for $\mathcal{V}=\mathcal{V}_N$ and $k=1$) 
  and \cite[Proposition~3.1]{Roh_pams17} (for $\mathcal{V}=\mathcal{V}_N$ 
  and general $k$), and generalised in \cite[Theorem~3.5]{RohSei_arx18} to 
  more general self-adjoint vertex conditions.
\end{remark}

Several useful inequalities now follow.

\begin{corollary}
\label{cor:increasing_vol}
  \begin{enumerate}
  \item \label{item:lengthening_an_edge} (Lengthening an edge) 
    Let $\NewGraph$ be obtained from $\Graph$ by lengthening the edge
    $e$. If $\lambda_k (\Graph) \geq 0$, then
    \begin{equation}
      \label{eq:lengthening_edge}
      \lambda_k(\NewGraph) \leq \lambda_k(\Graph)
      \qquad \text{for all } k\geq 1.
    \end{equation}
    The inequality is strict if
    $\lambda_k(\Graph) > \max(0,\lambda_{k-1}(\Graph))$ and there is
    an eigenfunction corresponding to $\lambda_k (\Graph)$ which does
    not vanish identically on $e$.
  \item \label{item:adding_an_edge} (Adding an edge between existing
    vertices) Suppose there exist $v,w \in \mathcal{V}(\Graph)$ and a
    choice of $n\geq1$ first eigenfunctions $\psi_1,\ldots,\psi_n$
    such that
    \begin{equation}
      \label{eq:functions_equal}
      \psi_k(v)=\psi_k(w)     
    \end{equation}
    for all $k=1,\ldots,n$.  If $\lambda_k (\Graph) \geq 0$, then the graph 
    $\NewGraph$ formed by inserting an edge of arbitrary length between $v$ 
    and $w$ satisfies
    \begin{displaymath}
      \lambda_k(\NewGraph) \leq \lambda_k(\Graph),
      \qquad k=1,\ldots,n.
    \end{displaymath}
  \item \label{item:adding_a_long_edge} (Adding a long edge between
    existing vertices) Suppose $\NewGraph$ is formed by adding an edge 
    of length $\ell$ connecting existing vertices $v$ and $w$ of $\Graph$.  
    Then $(\pi/\ell)^2 \leq \lambda_{k_0}(\Graph)$ implies
    $\lambda_k(\NewGraph) \leq \lambda_k(\Graph)$ for all $k\geq k_0$.
  \item \label{item:shrinking_redundant} (Shrinking a redundant edge)
    Let all vertices of the connected graph $\Graph$ have natural conditions.
    Suppose there exist an eigenfunction $\psi$ associated with
    $\lambda_2^N (\Graph)$ and an edge $e\in \mathcal{E} (\Graph)$
    such that $\psi|_e \equiv 0$. Then the graph $\NewGraph$ formed by
    shrinking $e$ to a point (i.e., removing $e$ and gluing its
    incident vertices together) satisfies
    $\lambda_2^N (\NewGraph) = \lambda_2^N (\Graph)$, and
    $\psi|_{\Graph \setminus e}$ is an eigenfunction associated with
    $\lambda_2^N (\NewGraph)$ (up to the canonical identification
    described in Remark~\ref{rem:canonical-identification}).
  \end{enumerate}  
\end{corollary}

\begin{remark} 
  Part~(\ref{item:lengthening_an_edge}) of Corollary~\ref{cor:increasing_vol}
  in general does not hold for negative eigenvalues, see for example
  \cite{ExJe}.  Note that in a graph with all natural conditions
  equation \eqref{eq:functions_equal} is trivially always satisfied 
  with $n=1$, since the eigenfunction is constant; in particular, for $n=2$, it 
  suffices to check \eqref{eq:functions_equal} for $\psi_2$ only.
  Part~(\ref{item:adding_an_edge}) for $\lambda_2^N(\Graph)$ probably 
  appeared for the first time in \cite[Theorem~3]{KMN}, where it was also 
  observed that inserting an edge between two existing vertices may not
  decrease $\lambda_2^N$, unless the edge is sufficiently long. Part
  (\ref{item:adding_a_long_edge}) gives a quantification of how long
  this has to be, and also reconciles this observation with the Weyl
  asymptotics, which requires that all \emph{sufficiently high}
  eigenvalues must decrease upon the insertion of an additional edge.
\end{remark}

\begin{remark}[Hadamard formula]
  \label{rem:hadamard}
  If an eigenvalue $\lambda = \lambda(\Graph)$ is simple, there is a
  quantitative expression for its increase if an edge is lengthened,
  which we refer to as a \emph{Hadamard formula}, by way of analogy with
  similar formulae on domains:
  \begin{equation}
    \label{eq:hadamard_formula}
    \frac{d\lambda}{d|e|} = -\mathscr{E}_e := 
    -\left(\psi'(x)^2 + \lambda\psi(x)^2\right), \qquad x\in e,
  \end{equation}
  where $|e|$ is the length of the edge $e$, $\psi$ is the
  $\lambda$-eigenfunction, normalised to have $L^2(\Graph)$-norm one,
  and the expression $\mathscr E_e$, often called the \emph{Pr\"ufer
    amplitude} of the edge $e$, can be easily shown to be independent
  of the location $x\in e$.  This has the following
    immediate consequence, which we will use repeatedly and thus state
    explicitly: if we lengthen one edge $e_1$ and shorten another
    $e_2$ by the same amount, then the derivative of a simple
    eigenvalue $\lambda$ with respect to this operation at the
    identity exists and has the same sign as
  \begin{equation}
  \label{eq:simple-hadamard}
	-\mathscr{E}_{e_1} + \mathscr{E}_{e_2},
  \end{equation}
  with $\mathscr{E}_{e_1}$, $\mathscr{E}_{e_2}$ as in \eqref{eq:hadamard_formula}.

  The formula \eqref{eq:hadamard_formula} is proved in \cite[Proof of
  the Lemma]{Fri_ijm05}, \cite[Appendix~A]{CdV_ahp15} and
  \cite[Lemma~5.2]{BanLev_ahp17} for natural vertex conditions only,
  although the same proof also works if $\Graph$ has some Dirichlet
  conditions (see also \cite[Section~3.1.4]{BerKuc_msm13} for the case
  of a degree one vertex with Dirichlet conditions), and we will use
  it without further comment for these conditions. Actually, the same
  formula turns out to hold in greater generality, for \emph{general}
  self-adjoint vertex conditions (including $\delta$-type). We intend
  to return to this point in a later work.
\end{remark}

%%%%%%%%%%%%
\subsection{Operations transferring the volume}
\label{sec:op_transfer_volume}

We now list some useful new principles based on changing the geometry 
of the graph by moving edges around, while keeping the total length
constant.

\begin{definition}[Transplantation]
  \label{def:transplantation}
  Cut through some of the vertices of $\Graph$ (in the sense of
  Definition~\ref{def:cutting}) to produce two disjoint metric graphs
  $\mathcal R,\mathcal C$ (neither of which is required to be connected). 
  Assume that only natural conditions equip the vertices
  of $\mathcal{C}$. Take any connected metric graphs
  $\GraphH_1, \ldots, \GraphH_k$ with purely natural vertex conditions
  and such that $|\GraphH_1|+\ldots+|\GraphH_k|=|\mathcal C|$. If
  $\NewGraph$ is formed by inserting each $\GraphH_i$ into
  $\mathcal{R}$ at some vertices
  $v_1,\ldots,v_k\in \mathcal V\setminus\mathcal V_D$ in accordance
  with Definition~\ref{def:insert}, then we say that $\NewGraph$ is
  obtained by {\bf transplanting the subgraph $\mathcal{C}$ to the
    subgraphs $\GraphH_1,\ldots,\GraphH_k$ at $v_1, \ldots, v_k$} (for
  short, transplantation of $\mathcal{C}$ to $v_1,\ldots,v_k$).
\end{definition}

An important special case consists of transplanting one or more edges
to a vertex $v$ by either inserting a new pendant edge at $v$, as
depicted in Figure~\ref{fig:pumpkin-chain-presurg}, or lengthening
existing edges incident to $v$.

\begin{figure}[H]
\label{fig:pumpkin-chain-presurg}
\begin{minipage}{6cm}
\begin{tikzpicture}[scale=0.6]
\coordinate (a) at (0,0);
\coordinate (b) at (2,0);
\coordinate (c) at (6,0);
\coordinate (d) at (9,0);
%\coordinate (e) at (11,0);
\draw (a) -- (b);
\draw[red,bend left]  (a) edge (b);
\draw[red,bend right]  (a) edge (b);
\draw (b) -- (c);
\draw[bend right]  (b) edge (c);
\draw[bend left=70]  (b) edge (c);
\draw[bend right=70]  (b) edge (c);
\draw[bend left]  (b) edge (c);
\draw (c) -- (d);
\draw[fill] (0,0) circle (2pt);
\draw[fill] (2,0) circle (2pt);
\draw[fill] (6,0) circle (2pt);
\draw[fill] (9,0) circle (2pt);
\node at (6,0) [anchor=south] {$v_0$};
\node at (1,-1) [anchor=south] {$e_2$};
\node at (1,1) [anchor=north] {$e_1$};
\end{tikzpicture}
\end{minipage}
\begin{minipage}{1.5cm}
\begin{tikzpicture}[scale=0.8]
\draw[-{Stealth[scale=0.5,angle'=60]},line width=2.5pt] (4,0) -- (5,0);
\end{tikzpicture}
\end{minipage}
\begin{minipage}{6cm}
\begin{tikzpicture}[scale=0.6]
\coordinate (a) at (0,0);
\coordinate (b) at (2,0);
\coordinate (c) at (6,0);
\coordinate (d) at (9,0);
\coordinate (f) at (8.5,1);
\draw (a) -- (b);
\draw (b) -- (c);
\draw[bend right]  (b) edge (c);
\draw[bend left=70]  (b) edge (c);
\draw[bend right=70]  (b) edge (c);
\draw[bend left]  (b) edge (c);
\draw (c) -- (d);
\draw[red] (c) -- (f);
\draw[fill] (0,0) circle (2pt);
\draw[fill] (2,0) circle (2pt);
\draw[fill] (6,0) circle (2pt);
\draw[fill] (9,0) circle (2pt);
\draw[fill] (8.5,1) circle (2pt);
\node at (6,0) [anchor=south] {$v_0$};
\node at (6.5,.5) [anchor=south west] {$e_0$};
\end{tikzpicture}
\end{minipage}
\caption{An example of transplantation: the graph on the right is obtained by transplantation of $\{e_1,e_2\}$ to $\GraphH:=\{e_0\}$ at $v_0$.}
\end{figure}
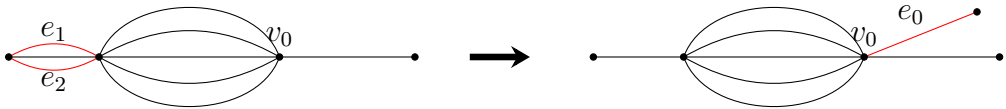

\begin{definition}[Unfolding edges]
  \label{def:unfolding_edges}
  Suppose $\Graph$ has $ k \geq 2 $ selected parallel
  edges $e_1, e_2, \ldots, e_k$ between the vertices $v_1$ and $v_2$.\footnote{We 
  explicitly allow the existence of further edges between $v_1$ and $v_2$, 
  as well as the possibility that $v_1 = v_2$.}
  Then the operation of replacing the parallel edges by
  a single edge of length $|e_1| + |e_2| + \dots + |e_k|$ running
  between $v_1$ and $v_2$ (preserving the 
  $\delta$-condition strengths at $v_1$ and $v_2$) is called {\bf unfolding 
    parallel edges}.
  
  Let $e_1, e_2, \ldots, e_k$ be pendant edges which are attached to
  the same vertex $v$ and have natural conditions at their endpoints.
  The operation of replacing $e_1, e_2, \ldots, e_k$ by a single
  pendant edge at $v$ of length $|e_1| + |e_2| + \dots + |e_k|$ (and
  again preserving the $\delta$-condition strength at $v$) 
  is called {\bf unfolding pendant edges}.
\end{definition}

\begin{figure}[H]
\label{fig:pumpkin-chain-symmetrisation}
\begin{minipage}{4.5cm}
\begin{tikzpicture}[scale=0.6]
\coordinate (a) at (0,0);
\coordinate (b) at (1.5,1.5);
\coordinate (c) at (2,4);
\coordinate (d) at (4,2);
\coordinate (e1) at (0.5,4);
\coordinate (e2) at (0.7,5.3);
\coordinate (f1) at (6,3);
\coordinate (f2) at (6,2);
\draw (a) -- (b);
\draw (b) -- (c);
\draw (b) -- (d);
\draw (c) -- (e1);
\draw (c) -- (e2);
\draw (d) -- (f1);
\draw (d) -- (f2);
\draw[red,bend left]  (c) edge (d);
\draw[red,bend right]  (c) edge (d);
\draw[fill] (a) circle (2pt);
\draw[fill] (b) circle (2pt);
\draw[fill] (c) circle (2pt);
\draw[fill] (d) circle (2pt);
\draw[fill] (e1) circle (2pt);
\draw[fill] (f1) circle (2pt);
\draw[fill] (e2) circle (2pt);
\draw[fill] (f2) circle (2pt);
\node at (c) [anchor=south] {$v_1$};
\node at (d) [anchor=north] {$v_2$};
\node at (3.5,3.3) [anchor=south] {$e_1$};
\node at (2.5,2.7) [anchor=north] {$e_2$};
\end{tikzpicture}
\end{minipage}
\begin{minipage}{1.5cm}
\begin{tikzpicture}[scale=0.8]
\draw[-{Stealth[scale=0.5,angle'=60]},line width=2.5pt] (4,0) -- (5,0);
\end{tikzpicture}
\end{minipage}
\begin{minipage}{6cm}
\begin{tikzpicture}[scale=0.6]
\coordinate (a) at (0,0);
\coordinate (b) at (1.5,1.5);
\coordinate (c) at (1,4);
\coordinate (d) at (4,1);
\coordinate (e1) at (-0.5,4);
\coordinate (e2) at (-0.2,5.3);
\coordinate (f1) at (6,2);
\coordinate (f2) at (6,1);
\draw (a) -- (b);
\draw (b) -- (c);
\draw (b) -- (d);
\draw (c) -- (e1);
\draw (c) -- (e2);
\draw (d) -- (f1);
\draw (d) -- (f2);
\draw[red]  (c) -- (d);
\draw[fill] (a) circle (2pt);
\draw[fill] (b) circle (2pt);
\draw[fill] (c) circle (2pt);
\draw[fill] (d) circle (2pt);
\draw[fill] (e1) circle (2pt);
\draw[fill] (f1) circle (2pt);
\draw[fill] (e2) circle (2pt);
\draw[fill] (f2) circle (2pt);
\node at (c) [anchor=south] {$v_1$};
\node at (d) [anchor=north] {$v_2$};
\node at (2.5,2.5) [anchor=north] {$e_0$};
\end{tikzpicture}
\end{minipage}
  \caption{Unfolding the parallel edges $e_1$ and $e_2$ in the graph on the
    left to obtain the one on the right, that is, replacing the loop
    $e_1,e_2$ with a single edge $e_0$ of equal total length.}
  \label{fig:edgesubstitution}
\end{figure}
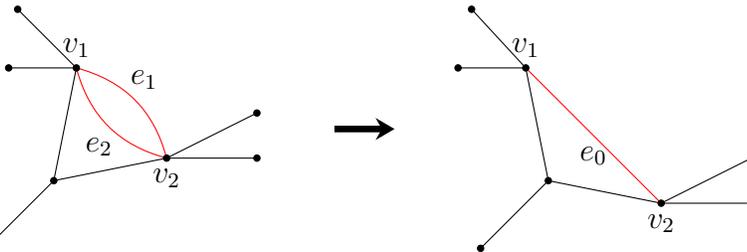

A variation on the unfolding of parallel edges is symmetrising them, possibly 
while reducing their number.

\begin{definition}
  Suppose $\Graph$ has $ k \geq 2 $ selected parallel edges
  $e_1, \ldots, e_k$ of arbitrary lengths $|e_1|,\ldots,|e_k|>0$
  between the vertices $v_1$ and $v_2$.\footnote{Here, as well, we emphasise that we allow both 
  the existence of further parallel edges between $v_1$ and $v_2$, and 
  also $v_1 = v_2$.}  We say
  $\NewGraph$ is formed from $\Graph$ by {\bf symmetrisation} of
  $e_1, \ldots, e_k$ if these parallel edges are replaced by $m\leq k$
  parallel edges, each of length $(|e_1|+\ldots+|e_k|)/m$, and the
  $\delta$-condition strengths at $v_1$ and $v_2$ are preserved.
\end{definition}

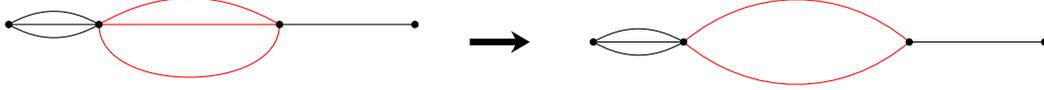
\begin{figure}[H]
\label{fig:pumpkin-chain-symmetrisation}
\begin{minipage}{6cm}
\begin{tikzpicture}[scale=0.6]
\coordinate (a) at (0,0);
\coordinate (b) at (2,0);
\coordinate (c) at (6,0);
\coordinate (d) at (9,0);
\draw (a) -- (b);
\draw[bend left]  (a) edge (b);
\draw[bend right]  (a) edge (b);
\draw[red] (b) -- (c);
\draw[red,bend right=90]  (b) edge (c);
\draw[red,bend left]  (b) edge (c);
\draw (c) -- (d);
\draw[fill] (0,0) circle (2pt);
\draw[fill] (2,0) circle (2pt);
\draw[fill] (6,0) circle (2pt);
\draw[fill] (9,0) circle (2pt);
\end{tikzpicture}
\end{minipage}
\begin{minipage}{1.5cm}
\begin{tikzpicture}[scale=0.8]
\draw[-{Stealth[scale=0.5,angle'=60]},line width=2.5pt] (4,0) -- (5,0);
\end{tikzpicture}
\end{minipage}
\begin{minipage}{6cm}
\begin{tikzpicture}[scale=0.6]
\coordinate (a) at (0,0);
\coordinate (b) at (2,0);
\coordinate (c) at (7,0);
\coordinate (d) at (10,0);
\draw (a) -- (b);
\draw[bend left]  (a) edge (b);
\draw[bend right]  (a) edge (b);
\draw[red,bend left=40]  (b) edge (c);
\draw[red,bend right=40]  (b) edge (c);
\draw (c) -- (d);
\draw[fill] (0,0) circle (2pt);
\draw[fill] (2,0) circle (2pt);
\draw[fill] (7,0) circle (2pt);
\draw[fill] (10,0) circle (2pt);
\end{tikzpicture}
\end{minipage}
\caption{An example of symmetrisation: The graph on the right is obtained by replacing the 3-pumpkin in the pumpkin chain on the left by an equilateral 2-pumpkin.}
\end{figure}

The following theorem applies to the first non-trivial eigenvalue
$\mu(\Graph)$ of a connected graph $\Graph$ (see
Definition~\ref{def:mu}): in particular, we make no assumptions on the
vertex conditions except where explicitly stated.

\begin{theorem}
  \label{thm:transferring_vol}
  The following operations decrease the first non-trivial eigenvalue
  of a connected graph $\Graph$, i.e.\ 
  \begin{equation}
    \label{eq:nontrivial_decrease}
    \mu(\NewGraph) \leq \mu(\Graph).
  \end{equation}
  \begin{enumerate}
  \item \label{item:transplanting}
    (Transplantation) Suppose for the subgraph
    $\mathcal{C} \subset \Graph$ and
    $v_1,\ldots,v_k \in \mathcal{V}_N (\Graph)$ that there exists a
    $\mu (\Graph)$-eigenfunction $\psi$ such that
    \begin{equation}
      \label{eq:smaller_than_v}
      0 \leq \min_{x\in\mathcal{C}} \psi(x) 
      \leq \max_{x\in\mathcal{C}} \psi(x) \leq \min_{i=1,\ldots,k}\psi (v_i).
    \end{equation}
    If $\mathcal{C}$, together with further graphs
    $\GraphH_1,\ldots, \GraphH_k$, satisfies the requirements of
    Definition~\ref{def:transplantation}, then
    \eqref{eq:nontrivial_decrease} holds for the graph $\NewGraph$
    obtained by transplanting $\mathcal{C}$ to
    $\GraphH_1, \ldots, \GraphH_k$ at $v_1,\ldots, v_k$ in accordance
    with Definition~\ref{def:transplantation}. The inequality
    \eqref{eq:nontrivial_decrease} is strict if 
    $\min_{x\in\mathcal{C}}\psi(x) < \max_{i=1,\ldots,k} \psi(v_i)$.
  \item \label{item:unfolding_parallel} (Unfolding parallel edges)
    Inequality \eqref{eq:nontrivial_decrease} holds for the graph
    $\NewGraph$ obtained by unfolding edges $e_1,e_2,\ldots,e_k$,
    $k\geq 2$, in accordance with
    Definition~\ref{def:unfolding_edges}.  Equality in
    \eqref{eq:nontrivial_decrease} implies that \emph{either} every
    eigenfunction of $\mu(\Graph)$ is constant on
    $e_1 \cup \ldots \cup e_k$, \emph{or} $\Graph$ is a figure-8 with
    natural conditions and $\NewGraph$ is a loop.\footnote{Since a 
    figure-8 consists of two parallel edges running from the central 
    vertex to itself, unfolding these edges, that is, replacing them with a 
    single edge of the same length, produces a loop. By Corollary~\ref{cor:joining_points},
    $\lambda_2^N$ is unaltered by this operation.}
  \item \label{item:symm_parallel} (Symmetrising parallel edges)
    Inequality \eqref{eq:nontrivial_decrease} holds whenever $k$ edges
    $e_1,\ldots, e_k$ are symmetrised to $m$ parallel edges,
    $1\leq m\leq k$, \emph{if} there is an eigenfunction $\psi$ of
    $\mu(\Graph)$ which is monotonically increasing along each of the
    edges $e_1,\ldots,e_k$ from $v_1$ to $v_2$. In this case, equality
    in \eqref{eq:nontrivial_decrease} implies that either $\psi$ is
    constant on $e_1 \cup \ldots \cup e_k$, or else
    $\NewGraph=\Graph$.  The same
    conclusions hold if this principle is applied to several pumpkin
    subgraphs within $\Graph$ simultaneously, provided that the
    conditions are satisfied separately on each pumpkin.
  \item \label{item:unfolding_pendant} (Unfolding pendant edges)
    Inequality \eqref{eq:nontrivial_decrease} holds for the graph
    $\NewGraph$ obtained by unfolding pendant
    edges. Equality in \eqref{eq:nontrivial_decrease} 
    implies that either every eigenfunction of $\mu$ is constant on 
    the pendant edges, or $\NewGraph=\Graph$.
\end{enumerate}
\end{theorem}

\begin{remark}
  Part~(\ref{item:symm_parallel}) has already appeared in the
  literature in a less general form, most recently in
  \cite{BeKeKuMu_arx17}.  The other statements are, to the best of our
  knowledge, completely new.  It would be interesting to have a
  complete characterisation of equality throughout, although this may 
  be difficult as in some cases it seems to depend on the global 
  geometry of $\Graph$, cf.~Example~\ref{ex:star-power}.
\end{remark}

\begin{remark}
  \label{rem:more-general-transplantation}
  The condition for strict inequality in
  Theorem~\ref{thm:transferring_vol}(\ref{item:transplanting}) is
  satisfied, for example, if $\psi(x)$ has non-zero variation on
  $\mathcal{C}$, i.e. if $\min_{x\in\mathcal{C}}\psi(x) <
  \max_{x\in\mathcal{C}}\psi(x)$.  This condition can only fail if
  $\psi$ is identically 0 on $\mathcal{C}$ or if $\mu = 0$ and $\psi$
  is constant on $\mathcal{C}$.

  The conclusion of
  Theorem~\ref{thm:transferring_vol}(\ref{item:transplanting}) also
  holds, with only trivial modifications of the same proof, if all
  vertex conditions are natural (so that $\mu=\lambda_2^N$ and $\psi$
  changes sign), and we have two subgraphs
  $\mathcal{C}_1 \subset \{\psi\geq 0\}$ and
  $\mathcal{C}_2 \subset \{\psi\leq 0\}$, such that $\mathcal{C}_1$ is
  transplanted to vertices $v_1,\ldots,v_{k_1}$ with
  \begin{displaymath}
	0 \leq \min_{x\in\mathcal{C}_1} \psi(x) 
	\leq \max_{x\in\mathcal{C}_1} \psi(x) \leq \min_{i=1,\ldots,k_1}\psi (v_i)
  \end{displaymath}
  and $\mathcal{C}_2$ is transplanted to vertices $w_1,\ldots,w_{k_2}$ such that
  \begin{displaymath}
	0 \geq \max_{x\in\mathcal{C}_2} \psi(x) 
	\geq \min_{x\in\mathcal{C}_2} \psi(x) \geq \max_{i=1,\ldots,k_2}\psi (w_i).
  \end{displaymath}
  We will not use this, so we do not go into details.
\end{remark}

%%%%%%%%%%%%
\subsection{Effects of surgery: some examples and counterexamples}
\label{sec:limitations}

We now give some basic examples illustrating why some of the assumptions in 
the above theorems are necessary, and why more can be expected in some cases.

\begin{example}
\label{ex:unfolding-pumpkins}
  We start with an example to show that Theorem~\ref{thm:transferring_vol}(\ref{item:unfolding_parallel}) 
  (unfolding parallel edges) does not have to apply to the higher eigenvalues. Take $\Graph$ 
  to be an equal $3$-pumpkin each of whose edges $e_1,e_2,e_3$ has length $1$. 
  Then the first few eigenvalues of $\Graph$ with natural boundary conditions are 
  $0,\pi^2,\pi^2,\pi^2$. If we unfold $e_1$ and $e_2$, we produce a loop $\NewGraph$ 
  of length $3$, whose first eigenvalues are $0,\frac{4\pi^2}{9}, \frac{4\pi^2}{9}, 
  \frac{16\pi^2}{9}$. In particular, $\lambda_4^N(\NewGraph)>\lambda_4^N(\Graph)$.
\end{example}

\begin{example}
\label{ex:star-struck}
  Theorem~\ref{thm:transferring_vol}(\ref{item:unfolding_pendant}) (unfolding \emph{pendant} edges) 
  also does not apply to the higher eigenvalues. Indeed, by a theorem of Friedlander 
  \cite[Theorem~1]{Fri_aif05}, for $k\geq 3$, the unique minimiser of 
  $\lambda_k^N$ among all graphs of total length $L>0$ is the equilateral $k$-star 
  (cf.~Definition~\ref{def:examples-of-graphs}). In particular, unfolding any two 
  of its $k$ pendant edges strictly increases $\lambda_k^N$.
\end{example}

\begin{example}
\label{ex:unfolding-story}
  On the other hand, Theorem~\ref{thm:transferring_vol}(\ref{item:unfolding_parallel}) 
  holds for all eigenvalues if one instead unfolds an \emph{odd} number of 
  edges, i.e., replaces $2k+1$ parallel edges between two vertices $v_1,v_2 \in 
  \mathcal{V}_N$ by a single edge of the same total length. This is a simple 
  application of Theorem~\ref{thm:changing_vc}(\ref{item:gluing_vertices}). We do not go into details, 
  but refer to Figure~\ref{fig:simple-unfolding} to illustrate the principle.
\begin{figure}[H]
\begin{center}
\begin{minipage}{5.5cm}
\begin{tikzpicture}
\coordinate (a1) at (-1,0);
\coordinate (a2) at (-1,1);
\coordinate (a) at (0,0);
\coordinate (b) at (3,0);
\coordinate (b1) at (4,-1);
\coordinate (b2) at (4,0);
\coordinate (b3) at (4,1);
\draw[bend left=60]  (a) edge (b);
\draw[bend right=60]  (a) edge (b);
\draw (a) -- (a1);
\draw (a) -- (a2);
\draw (a) -- (b);
\draw (b) -- (b1);
\draw (b) -- (b2);
\draw (b) -- (b3);
\draw[fill] (a1) circle (2pt);
\draw[fill] (a2) circle (2pt);
\draw[fill=green] (a) circle (2pt);
\draw[fill=red] (b) circle (2pt);
\draw[fill] (b1) circle (2pt);
\draw[fill] (b2) circle (2pt);
\draw[fill] (b3) circle (2pt);
\node at (a) [anchor=north] {$v_1$};
\node at (b) [anchor=south] {$v_2$};
\node at (0.7,0.6) [anchor=south] {$e_1$};
\node at (1.5,-0.05) [anchor=south] {$e_2$};
\node at (0.7,-0.6) [anchor=north] {$e_3$};
\end{tikzpicture}
\end{minipage}
\begin{minipage}{1.5cm}
\begin{tikzpicture}
\draw[-{Stealth[scale=0.5,angle'=60]},line width=2.5pt] (4,0) -- (5,0);
\end{tikzpicture}
\end{minipage}
\begin{minipage}{3cm}
\begin{tikzpicture}
\coordinate (c1) at (5,0);
\coordinate (c2) at (5,1);
\coordinate (c) at (6,0);
\coordinate (d) at (6.4,0);
\coordinate (e) at (8.6,0);
\coordinate (f) at (9,0);
\coordinate (f1) at (10,1);
\coordinate (f2) at (10,0);
\coordinate (f3) at (10,-1);
\draw (c1) -- (c);
\draw (c2) -- (c);
\draw[bend left=60] (c) edge (e);
\draw (e) -- (d);
\draw[bend right=60] (d) edge (f);
\draw (f) -- (f1);
\draw (f) -- (f2);
\draw (f) -- (f3);
\draw[fill] (c1) circle (2pt);
\draw[fill] (c2) circle (2pt);
\draw[fill=green] (c) circle (2pt);
\draw[fill=green] (d) circle (2pt);
\draw[fill=red] (e) circle (2pt);
\draw[fill=red] (f) circle (2pt);
\draw[fill] (f1) circle (2pt);
\draw[fill] (f2) circle (2pt);
\draw[fill] (f3) circle (2pt);
\node at (c) [anchor=north] {$v_1$};
\node at (f) [anchor=south] {$v_2$};
\node at (6.7,0.5) [anchor=south] {$e_1$};
\node at (7.5,-0.05) [anchor=south] {$e_2$};
\node at (6.8,-0.45) [anchor=north] {$e_3$};
\end{tikzpicture}
\end{minipage}
\bigskip
\begin{minipage}{6.5cm}
\begin{tikzpicture}
\draw[-{Stealth[scale=0.5,angle'=60]},line width=2.5pt] (9,0) -- (10,0);
\coordinate (c1) at (11,0);
\coordinate (c2) at (11,1);
\coordinate (f1) at (17.5,-1);
\coordinate (f2) at (17.5,0);
\coordinate (f3) at (17.5,1);
\draw (12,0) -- (16.5,0);
\draw (12,0) -- (c1);
\draw (12,0) -- (c2);
\draw (16.5,0) -- (f1);
\draw (16.5,0) -- (f2);
\draw (16.5,0) -- (f3);
\draw[fill=green] (12,0) circle (2pt);
\draw[fill] (c1) circle (2pt);
\draw[fill] (c2) circle (2pt);
\draw[fill=red] (13.5,0) circle (2pt);
\draw[fill=green] (15,0) circle (2pt);
\draw[fill=red] (16.5,0) circle (2pt);
\draw[fill] (f1) circle (2pt);
\draw[fill] (f2) circle (2pt);
\draw[fill] (f3) circle (2pt);
\node at (12,0) [anchor=north] {$v_1$};
\node at (16.5,0) [anchor=south] {$v_2$};
\node at (12.75,0) [anchor=south] {$e_1$};
\node at (14.25,0) [anchor=south] {$e_2$};
\node at (15.75,0) [anchor=south] {$e_3$};
\end{tikzpicture}
\end{minipage}
\caption{The three edges between $v_1$ and $v_2$ can be unfolded by cutting 
through the vertices in the right way.}\label{fig:simple-unfolding}
\end{center}
\end{figure}
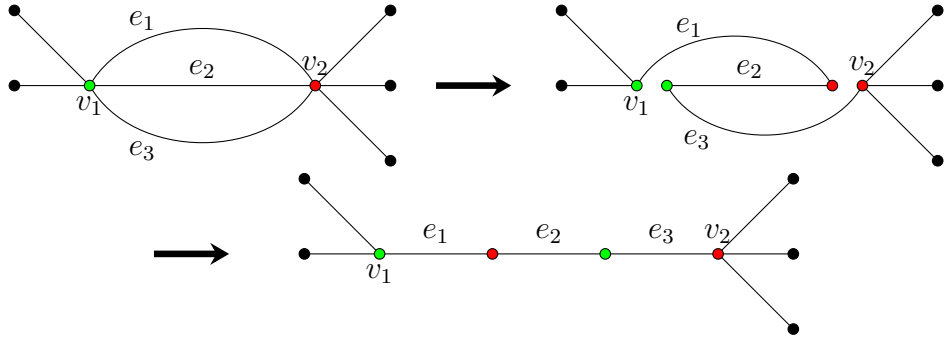
\end{example}

Likewise, Theorem~\ref{thm:transferring_vol}(\ref{item:unfolding_parallel}) holds 
for all $k\geq 1$ if the two edges $e_1$ and $e_2$ to be unfolded form a pendant 
(i.e., together they form a pendant loop); this also follows directly from 
Theorem~\ref{thm:changing_vc}(\ref{item:gluing_vertices}). Thus the general 
question of whether unfolding parallel edges decreases the higher
eigenvalues seems to be subtle.

\begin{example}
  \label{ex:loopy-pumpkin}
  We now show that the assumption of monotonicity of the eigenfunction
  in Theorem~\ref{thm:transferring_vol}(\ref{item:symm_parallel})
  (symmetrising parallel edges) cannot be dropped.  Choose
  $\varepsilon>0$ small and let $\Graph$ consist of two vertices $v_1$
  and $v_2$, joined by edges $e_1,\ldots,e_4$ of length
  $2-2\varepsilon$, $\varepsilon$, $\varepsilon$ and $1$, respectively
  ($\Graph$ is thus a $4$-pumpkin, of total length $3$).  Then, for
  $\varepsilon>0$ small enough, $\lambda_2^N(\Graph)$ is approximately
  equal to $4\pi^2/9$, the second eigenvalue of a
  figure-8 (and of a cycle) of total length $3$. If we apply
  Theorem~\ref{thm:transferring_vol}(\ref{item:symm_parallel}) to
  $e_1,e_2,e_3$ with $m=2$, we obtain a new graph $\NewGraph$
  consisting of two vertices and three parallel edges running between
  them, each of length one (an equilateral $3$-pumpkin), so that
  $\lambda_2^N (\NewGraph) = \pi^2$, the second 
  eigenvalue of a cycle of length $2$. Obviously, the configuration of
  $\Graph$ forces the eigenfunction $\psi$ to satisfy
  $\psi(v_1) \approx \psi(v_2)$; since $e_1$ has more than half the
  total length of $\Graph$, it is impossible for $\psi$ to be
  monotonic on it.
\end{example}

\begin{example}
\label{ex:star-power}
  Finally, we give an example to show that inequality in 
  Theorem~\ref{thm:transferring_vol}(\ref{item:unfolding_pendant}) 
  can be strict even if $\mu(\Graph)$ is simple and its eigenfunction 
  $\psi$ vanishes identically on the pendants. Take $\Graph$ to be a 
  $4$-star with edge lengths $1$, $1$, $1-\varepsilon$ and $1- 
  \varepsilon$ for some $\varepsilon>0$ small, connected at a central 
  vertex $v_0$. Then $\lambda_2^N (\Graph) = \pi^2/4$ with 
  eigenfunction $\psi$ supported on the two longer edges and vanishing 
  identically on the shorter ones, with $\psi(v_0)=0$. Now unfold the 
  shorter edges, so that $\NewGraph$ is a $3$-star with edge lengths 
  $1$, $1$ and $2-2\varepsilon>1$. Now since $\NewGraph$ can be 
  formed by attaching one of its edges of length $1$ as a pendant to a 
  path graph (interval) of length $3-2\varepsilon$, whose first non-trivial 
  eigenvalue is $\pi^2 /(3-2\varepsilon)^2 < \pi^2/4$, 
  Theorem~\ref{thm:increasing_vol}(\ref{item:attaching_a_pendant}) 
  implies $\lambda_2^N (\NewGraph) < \pi^2/4$.
\end{example}

%%%%%%%%%%%%%%%%%%%%%%%%%%%%%%%%%%%%%%%%%%%%%%%%%%%%%%%%%%%%%%%%%%%%%%%%%%%%%%%
\section{Proofs}
\label{sec:proofs}

%%%%%%%%%%%%
\subsection{Proof of Theorem~\ref{thm:changing_vc} and its
  corollaries}

For the proof of Theorem~\ref{thm:changing_vc}, we will need a sharp 
form of the Courant--Fischer minimax characterisation of the eigenvalues. 
If $H$ is a Hilbert space with inner product $(\,\cdot\,,\,\cdot\,)_H$, $a:D(a) 
\times D(a) \to \R$ a closed, symmetric, sesquilinear form bounded from below 
and defined on a dense and compactly embedded subspace $D(a) \subset H$,
then we have a sequence of associated eigenvalues $\lambda_1\leq 
\lambda_2 \leq \ldots$, with corresponding eigenvectors
$v_1,v_2,\ldots$, which satisfy $a(u,v_k) = \lambda_k(u,v_k)_H$ for
all $u\in D(a)$ and 
which can be chosen to form an orthonormal basis of $H$; and 
the eigenvalues can be characterised variationally as
\begin{align}
  \label{eq:minimax_full}
  \lambda_n &= \min_{\substack{X\subset D(a)\\ \dim(X)=n}} 
  \max_{0\neq u \in X} \frac{a(u,u)}{\|u\|_H^2} \\
  \label{eq:maximin_full}
            &= \max_{\substack{Y\subset D(a)\\ \dim(Y)=n-1}} 
	\min_{0\neq u \in Y^\perp} \frac{a(u,u)}{\|u\|_H^2}.
\end{align}
These formulae are well known, cf.\ \cite[Chapter~6]{CouHil_mmp53}, 
\cite[Section~I.6.10]{Kat_ptlo76},
\cite[Theorem~XIII.2]{ReeSim_mmmp78}. However, we need additionally
the following characterisation of equality,
Lemma~\ref{lem:min-max-equality}. While this is surely not new, it
does not seem to be in any of the standard references, including the
ones just cited; similar but not identical results are contained in
\cite[Theorem~2.4.3]{WeiSte}.  For completeness we include a proof.

\begin{lemma}
\label{lem:min-max-equality}
  With the notation just introduced,
  \begin{enumerate}
  \item \label{item:minimizer_contains_eigenvector} if $X$ realises
    the minimum in \eqref{eq:minimax_full} (corresp.\ if $Y$ achieves
    the maximum in \eqref{eq:maximin_full}), then $X$ (corresp.\ $Y$)
    contains an eigenvector of $\lambda_n$;
  \item \label{item:minimizer_contains_eigenspace}
    if $\lambda_n < \lambda_{n+1}$, then the intersection of all possible
    minimising $n$-dimensional subspaces $X$ in 
    \eqref{eq:minimax_full} is the eigenspace of $\lambda_n$.
  \end{enumerate}
\end{lemma}

\begin{proof}
  (\ref{item:minimizer_contains_eigenvector}) We prove the statement
  only for $X$. A simple dimension count yields that if $\dim(X)=n$,
  then we can find a vector $u\in X$ of norm 1 which is orthogonal to
  the first $n-1$ eigenvectors $v_1,\ldots,v_{n-1}$. We expand $u$ in the
  eigenvector basis,
  \begin{equation}
    \label{eq:perp-form}
    u = \sum_{k=n}^\infty \alpha_k v_k,
  \end{equation}
  for some coefficients $\alpha_k=(u,v_k)_H \in\R$; the normalisation condition reads 
  $\sum_{k=n}^\infty \alpha_k^2=1$. Moreover,
\begin{equation}
\label{eq:perp-form-form}
	a(u,u) = \sum_{k=n}^\infty \alpha_k^2\lambda_k \leq \lambda_n,
\end{equation}
the inequality following from the minimality of $X$. This is only possible if 
$\alpha_k = 0$ whenever $\lambda_k> \lambda_n$. Thus $u$ is a linear 
combination of eigenvectors having eigenvalue equal to $\lambda_n$.

(\ref{item:minimizer_contains_eigenspace}) We first show that if $X$
is an arbitrary minimising subspace, then every eigenvector of
$\lambda_n$ is in $X$; after a re-numbering of all eigenvalues equal
to $\lambda_n$ if necessary, it suffices to prove $v_n \in X$.  As in
(\ref{item:minimizer_contains_eigenvector}), we obtain a vector $u \in X$ 
of norm one having the form
\eqref{eq:perp-form}, such that \eqref{eq:perp-form-form} also holds.
But since $\lambda_n < \lambda_{n+1}$, we must have $\alpha_k=0$ for
all $k\geq n+1$.  That is, $u=v_n$.

Finally, let
$\lambda_{n-m} < \lambda_{n-m+1} = \ldots = \lambda_n <
\lambda_{n+1}$, i.e., suppose $\lambda_n$ has multiplicity $m$.  We need 
to show that a vector $u$ that belongs to every minimising $X$ must
in fact be a superposition of $v_{n-m+1},\ldots,v_n$.  We expand $u$,
\begin{equation}
  \label{eq:expansion_other_u}
  u = \sum_{k=1}^\infty \alpha_k v_k.
\end{equation}
Taking $X:=\aufspan\{v_1,\ldots,v_n\}$, we conclude that
$\alpha_k = 0$ for all $k>n$.  Let now $n>m$ and suppose without loss
of generality that $\alpha_1 \neq 0$.  Then we take
\begin{displaymath}
  X:= \aufspan \{v_1+\varepsilon v_{n+1},v_2, v_3,\ldots,
  v_n \}.
\end{displaymath}
For small enough $\varepsilon$ this subspace is minimising for
\eqref{eq:minimax_full} and cannot contain $u$ since the expansion of 
the latter contains $v_1$ but does not contain $v_{n+1}$.  Similarly, 
$\alpha_k = 0$ for all $k \leq n-m$, which completes the proof.
\end{proof}

Theorem~\ref{thm:changing_vc} is true as a special case of a more
general theorem about rank-1 perturbations of quadratic forms.  Since
the vertex conditions can enter into a form in two different ways, we need
to consider perturbations of two different types.

\begin{definition}
  \label{def:form_rank_1}
  Let $X$ be a normed space and $Z$ a closed subspace.
  We say that $Z$ is a \emph{co-dimension 1 subspace} of $X$ and write $Z
  \subset_1 X$, if the quotient space $X/Z$ is 1-dimensional.

  Let $a$ and $\newa$ be closed semi-bounded Hermitian forms with
  domains $D(a)$ and $D(\newa)$.  Then we say that $\newa$ is a
  \emph{positive rank-1 perturbation} of the form $a$ if $\newa=a$ on some
    $Z \subset_1 D(a)$ and 
  \begin{itemize}
  \item either $Z = D(\newa) \subset_1 D(a)$, 
  \item or $D(\newa) = D(a)$ and $\newa \geq a$.
  \end{itemize}
  We shall call the former case perturbation of type (R) (for
  ``restriction'') and the latter perturbation of type (V) (for ``variation'').
\end{definition}

This definition is, in particular, applicable to the form given 
by \eqref{eq:dirichlet-form} for any of the vertex conditions we are 
considering; moreover, it is easy to see that the operations in 
Theorem~\ref{thm:changing_vc} satisfy the definition of rank-1 
perturbation of the forms, where as always we make the identification of 
Remark~\ref{rem:canonical-identification}.  Indeed, gluing two vertices 
introduces a single constraint and is therefore of type (R); changing the 
strength of the $\delta$-condition corresponds to variation of the form 
except in the case when $\gamma'=\infty$, in which case the perturbation 
is again of type (R).

\begin{theorem}[Interlacing with equality characterisation]
  \label{thm:rank1}
  Let $\newa$ be a positive rank-1 perturbation of the form
  $a$. Then the eigenvalues of the two forms satisfy the 
  interlacing inequalities
  \begin{equation}
    \label{eq:interlacing_forms}
    \lambda_{k}(a) \leq \lambda_{k}(\newa) 
    \leq \lambda_{k+1}(a) \leq \lambda_{k+1}(\newa),
    \qquad k \geq 1.
  \end{equation}
  If a given value $\Lambda$ has multiplicities $m$ and $\widetilde{m}$
  in the spectra of $a$ and $\newa$, respectively, then
  $|m-\widetilde{m}|\leq 1$ and the intersection of the 
  respective $\Lambda$-eigenspaces of the two forms has dimension 
  $\min(m, \widetilde{m})$.
\end{theorem}

We remark that, in particular, the common eigenfunction(s) must belong to
the domain $D(\newa)$ which may be smaller than $D(a)$.

\begin{proof}[Proof of Theorem~\ref{thm:rank1} and hence of
  Theorem~\ref{thm:changing_vc}]
  Introduce, for convenience, the notation
  \begin{equation}
    \label{eq:max_in_minimax}
    M(a, X) := \max_{0\neq u \in X} \frac{a(u,u)}{\|u\|^2}.
  \end{equation}

  The inequalities \eqref{eq:interlacing_forms} are standard; the
  inequality $\lambda_k (a) \leq \lambda_k (\newa)$ follows from
  minimising over a smaller set $D(\newa)$ in \eqref{eq:minimax_full}
  in the type (R) case or from inequality $a \leq \newa$ in the type
  (V) case.  The inequality $\lambda_k (\newa) \leq \lambda_{k+1}(a)$
  follows from the rank of the perturbation.  Indeed, let $X$ be any
  $(k+1)$-dimensional subspace of $D(a)$ such that
  $\lambda_{k+1}(a) = M(a,X)$.  Then the subspace $X \cap Z$ (with $Z$
  as in Definition~\ref{def:form_rank_1}) has dimension at least $k$.
  Choosing an arbitrary subset $\newX$ of $X \cap Z$ of dimension $k$,
  we have $\lambda_k(\newa) \leq M(\newa, \newX) = M(a,\newX) \leq
  M(a, X) = \lambda_{k+1}(a)$.

  The case of equality is more interesting.  The fact that the
  multiplicities $m$ and $\widetilde{m}$ differ by at most 1 is a simple
  consequence of the interlacing.  Let $\Lambda \in \R$ be an
  eigenvalue of $a$ and denote by $E(a, \Lambda) \subset H$
  the corresponding eigenspace.  With $r :=m-1$ and an appropriate $k$,
  we can write
  \begin{equation}
    \label{eq:equalities_of_a}
    \lambda_{k-1}(a) < \lambda_k(a) = \ldots =\Lambda= \ldots =
    \lambda_{k+r}(a) < \lambda_{k+r+1}(a).    
  \end{equation}
  We have to consider four cases of arrangements of eigenvalues of
  $\newa$ among \eqref{eq:equalities_of_a}.  We will deal with these
  arrangements two at a time.

  Suppose the equalities line up in one of the following two ways:
  \begin{align}
    \label{eq:equalities_case1}
    \ldots &< \lambda_k(a) = \lambda_k (\newa) 
    = \ldots 
    = \lambda_{k+r} (a) = \lambda_{k+r}(\newa) < \ldots \\
    \label{eq:equalities_case2}
    \ldots < \lambda_{k-1} (\newa) &= \lambda_k (a) = 
    \lambda_k (\newa) = \ldots = \lambda_{k+r} (a) = \lambda_{k+r} 
    (\newa) < \ldots
  \end{align}
  In these two cases $m \leq \widetilde{m}$ and we will show that
  $E(a, \Lambda) \subset E(\newa, \Lambda)$; the claim then follows from 
  a dimension count argument.  Let $\newX$ be any
  subspace of dimension $k+r$ such that $M(\newa, \newX) = \Lambda$.
  Since $\newX \subset D(a)$, we can also consider $a$ on $\newX$.  We
  have
  \begin{equation}
    \label{eq:Lambda_chain2}
    \Lambda = \lambda_{k+r}(a) \leq M(a, \newX) 
    \leq M(\newa, \newX) = \Lambda,
  \end{equation}
  and therefore $M(a, \newX) = \Lambda$.  Thus $\newX$ is a minimising
  subspace for $\lambda_{k+r}(a)$ and by
  Lemma~\ref{lem:min-max-equality}(\ref{item:minimizer_contains_eigenspace})
  we get $E(a, \Lambda) \subset \newX$.  Since the intersection of all
  such $\newX$ coincides with $E(\newa, \Lambda)$, we are done.

  Now suppose the equalities line up in one of the following ways:
  \begin{align}
    \label{eq:equalities_case3}
    \ldots < \lambda_{k-1} (\newa) 
    &= \lambda_k (a) = \lambda_k (\newa)
    = \ldots = \lambda_{k+r-1} (\newa) 
    = \lambda_{k+r} (a) < \ldots \\
    \label{eq:equalities_case4}
    \ldots 
    &< \lambda_k (a) = \lambda_k (\newa) 
    = \ldots = \lambda_{k+r-1}(\newa) 
    = \lambda_{k+r} (a) < \ldots
  \end{align}
  Here we need to show that
  $E(\newa, \Lambda) \subset E(a, \Lambda)$.  Consider any minimising
  subspace $X$ of dimension $k+r$ for $a$, i.e.\ $M(a,X) = \Lambda$.
  Let $\newX \subset Z \cap X$, where $Z$ is the subspace from
  Definition~\ref{def:form_rank_1} of co-dimension 1 and therefore
  $\newX$ can be chosen to have dimension $k+r-1$.  On $\newX$ the two
  forms agree, therefore we have
  \begin{equation}
    \label{eq:Lambda_chain3}
    \Lambda = \lambda_{k+r-1}(\newa) \leq M(\newa,\newX) = M(a,\newX)
    \leq M(a,X) = \Lambda.  
  \end{equation}
  Thus $M(\newa,\newX) = \Lambda$, $\newX$ is a minimising subspace
  for $\lambda_{k+r-1}(\newa)$ and by
  Lemma~\ref{lem:min-max-equality}(\ref{item:minimizer_contains_eigenspace})
  we conclude $E(\newa, \Lambda) \subset \newX \subset X$.  As above,
  we now take the intersection over all possible $X$ and use
  Lemma~\ref{lem:min-max-equality}(\ref{item:minimizer_contains_eigenspace})
  again to conclude $E(\newa, \Lambda) \subset E(a, \Lambda)$.
\end{proof}

To continue in the spirit of this section, we will establish
Corollary~\ref{cor:joining_points} by proving a slightly more
general claim. For the rest of this subsection, $a$ continues 
to be any closed semi-bounded Hermition form, with domain $D(a)$ in a 
normed space $X$.

\begin{definition}
  \label{def:invariant_vector}
  A vector $v \in D(a)$ is \emph{invariant with respect to
    perturbation of $a$ to $\newa$} if $v \in D(\newa)$ and
  $a(v,v) = \newa(v,v)$.  
\end{definition}

Note that if $\newa$ is a rank-1 perturbation of $a$ of type (R), the
condition $a(v,v) = \newa(v,v)$ is satisfied automatically and if it
is a rank-1 perturbation of type (V), the condition $v \in D(\newa)$
is satisfied automatically.

\begin{lemma}
  \label{cor:joining_points_general}
  Let $\newa$ be a positive rank-1 perturbation of the form
  $a$.  If the first $k$ eigenvectors of $a$ are invariant with
  respect to the perturbation, then
  \begin{equation}
    \lambda_i(a) = \lambda_i(\newa), \qquad 1\leq i \leq k.
  \end{equation}
\end{lemma}

\begin{proof}[Proof of Lemma~\ref{cor:joining_points_general} and
  hence of Corollary~\ref{cor:joining_points}]
  By positivity of the perturbation (see Theorem~\ref{thm:rank1}),
  \begin{displaymath}
    \lambda_i(a)\leq\lambda_i(\newa)
  \end{displaymath}
  for $i=1,\ldots,k$. Denote the eigenvectors of $a$ by
  $v_1,\ldots,v_k$ and form $X_i = \aufspan\{v_1,\ldots,v_i\}$ with
  $i\leq k$.  By assumption, $X_i$ is contained in the form domain of
  $\newa$ and 
  \begin{equation*}
    \lambda_i(\newa) 
    \leq \max_{0\neq u \in X_i} \frac{\newa(u,u)}{\|u\|_H^2}
    = \max_{0\neq u \in X_i} \frac{a(u,u)}{\|u\|_H^2}
    = \lambda_i(a).
  \end{equation*}
  Thus we have equality of the eigenvalues.
\end{proof}

\begin{remark}
  The rank of the perturbation plays no role in the proof, but we have
  not defined general positive perturbations of the form $a$ and will
  not require them.
\end{remark}

%%%%%%%%%%%%
\subsection{Proof of Theorem~\ref{thm:increasing_vol} and its
  corollaries}

\begin{proof}[Proof of Theorem~\ref{thm:increasing_vol}]
  (\ref{item:attaching_a_pendant}) Suppose that
  \begin{equation}
    \label{eq:pendant_eig}
    \lambda_{r}(\GraphH) \leq \lambda_{k}(\Graph).
  \end{equation}
  The spectrum of the union of the two graphs (considered as a single, disconnected 
  graph $\Graph \,\dot{\cup}\, \GraphH$) is the union of the two spectra.  Therefore,
  \begin{equation}
    \label{eq:spectral_shift}
    \lambda_{k}(\Graph) = \lambda_{n}(\Graph \,\dot{\cup}\, \GraphH), 
  \end{equation}
  for some $n=n(k,r) \geq k+r$. Now attaching $\GraphH$ to $\Graph$ to form 
  $\NewGraph$ is an operation covered by 
  Theorem~\ref{thm:changing_vc}(\ref{item:gluing_vertices})
  (gluing vertices) and we get
  \begin{equation}
    \label{eq:pendant_gluing_shift}
   \lambda_{n-1}(\NewGraph) 
   \leq \lambda_{n}(\Graph \,\dot{\cup}\, \GraphH). 
  \end{equation}
  Combining this with the estimate $n\geq k+r$, we thus obtain
  \begin{equation}
    \label{eq:extended_pendant_attach_result}
    \lambda_{k+r-1}(\NewGraph) \leq \lambda_k(\Graph).
  \end{equation}

  To obtain strict inequality under the additional conditions
  stipulated in the theorem, let $\lambda_{n}(\Graph \,\dot{\cup}\, \GraphH)$ be
  the first occurence of $\Lambda := \lambda_k(\Graph)$ in the spectrum of
  $\Graph \,\dot{\cup}\, \GraphH$, namely
  $\lambda_{n-1}(\Graph \,\dot{\cup}\, \GraphH) 
    < \lambda_{n}(\Graph \,\dot{\cup}\, \GraphH)
    = \Lambda$.
  From $\Lambda > \lambda_{k-1}(\Graph)$ and
  $\Lambda > \lambda_{r}(\GraphH)$ we still have $n \geq k+r$.
  Assume we do not have strict inequality in
  \eqref{eq:pendant_gluing_shift}, i.e.
  \begin{equation*}
    \lambda_{n-1}(\Graph \,\dot{\cup}\, \GraphH) 
    < \lambda_{n-1}(\NewGraph) = \lambda_{n}(\Graph \,\dot{\cup}\, \GraphH),
  \end{equation*}
  which fits the circumstances described in Remark~\ref{rem:catch_up}.
  But the eigenspace of $\lambda_{n}(\Graph \,\dot{\cup}\, \GraphH)$ contains a
  function which vanishes identically on $\GraphH$ and is non-zero at
  $v_0$; this function cannot be contained in the eigenspace of
  $\NewGraph$ because it does not belong to its form domain.
  Therefore
  $\lambda_{n-1}(\NewGraph) < \lambda_{n}(\Graph \,\dot{\cup}\, \GraphH) =
  \Lambda$
  and thus
  $\lambda_{k+r-1}(\NewGraph) \leq \lambda_{n-1}(\NewGraph) <
  \Lambda$.

  (\ref{item:inserting_a_graph}) Denote by $\widehat{\GraphH}$
  the graph obtained from $\GraphH$ by gluing the vertices
  $w_1,\ldots,w_m$ to form a single vertex $w^\ast$.  Attach
  $\widehat{\GraphH}$ to $\Graph$ by gluing $w^\ast$ and $v_0$,
  keeping the $\delta$-condition $\gamma_0$ at $v_0$.  Call the new
  graph $\widehat{\Graph}$.  Since all vertex conditions on
  $\widehat{\GraphH}$ are natural, $\lambda_1(\widehat{\GraphH}) = 0$
  and we can use part~(\ref{item:attaching_a_pendant}) of the theorem
  with $r=1$, as long as $\lambda_k (\Graph)\geq 0$.  We conclude that
  $\lambda_k(\widehat{\Graph}) \leq \lambda_k (\Graph)$ and then cut
  through the vertex $v_0$ of $\widehat{\Graph}$ to restore the
  vertices $w_1,\ldots,w_m$ and thus create the graph $\NewGraph$.  Since
  the $\delta$-conditions at $w_1,\ldots,w_m$ sum to $\gamma_0$, 
  by Theorem~\ref{thm:changing_vc}(\ref{item:gluing_vertices}) we have
  $\lambda_k (\NewGraph) \leq \lambda_k (\widehat{\Graph})$ for all
  $k\geq 1$. To obtain strict inequality we use the strict version of
  part~(\ref{item:attaching_a_pendant}).
\end{proof}

\begin{proof}[Proof of Corollary~\ref{cor:increasing_vol}]
  (\ref{item:lengthening_an_edge}) follows directly from
  Theorem~\ref{thm:increasing_vol}(\ref{item:inserting_a_graph}) by 
  picking an arbitrary internal point of $e$, making this a dummy vertex, 
  and inserting a one-edge graph $\GraphH$ at this 
  point in accordance with Definition~\ref{def:insert}.

  (\ref{item:adding_an_edge}) Form a new graph
  $\widehat{\Graph}$ from $\Graph$ by gluing the vertices $v$ and $w$
  (remembering to add the $\delta$-potentials if present); call the
  new vertex $v_0$. Then by Corollary~\ref{cor:joining_points},
  \begin{displaymath}
    \lambda_1 (\widehat{\Graph}) = \lambda_1 (\Graph), \ldots, 
    \lambda_n (\widehat{\Graph}) = \lambda_n (\Graph).
  \end{displaymath}
  Now insert a one-edge graph $\GraphH$ at $v_0$ separating the
  vertices $v$ and $w$ again.

  (\ref{item:adding_a_long_edge}) Denote by $\GraphH$ the graph
  consisting of the long edge, with natural endpoints.  Then the
  assumption on $\ell$ implies that
  $\lambda_2^N (\GraphH) \leq \lambda_{k_0} (\Graph)$. We now glue one
  endpoint of $\GraphH$ and $v$ to form a new graph $\widehat{\Graph}$
  and apply Theorem~\ref{thm:increasing_vol}(\ref{item:attaching_a_pendant}) 
  with $r=2$ to obtain
  \begin{displaymath}
	\lambda_{k+1} (\widehat{\Graph}) \leq \lambda_k (\Graph)
  \end{displaymath}
  for all $k \geq k_0$. We now glue $w$ and the other endpoint of $\GraphH$  to form 
  $\NewGraph$; then Theorem~\ref{thm:changing_vc}(\ref{item:gluing_vertices}) implies
  \begin{displaymath}
	\lambda_k (\NewGraph) \leq \lambda_{k+1} (\widehat{\Graph})
  \end{displaymath}
  for all $k$, whence the claim.

  (\ref{item:shrinking_redundant}) By
  Corollary~\ref{cor:joining_points}, we may assume that the edge $e$ on
  which $\psi$ vanishes is pendant (possibly a loop) by gluing
  its incident vertices if necessary. Now $\NewGraph$ is formed from
  $\Graph$ by removing the pendant graph consisting of $e$; hence
  $\lambda_2^N(\Graph) \leq \lambda_2^N(\NewGraph)$ by
  Theorem~\ref{thm:increasing_vol}(\ref{item:attaching_a_pendant}). 
  But it is easy to check
  directly that the restriction of $\psi$ to $\NewGraph$ satisfies the
  eigenvalue equation and the vertex conditions.  Therefore
  $\lambda_2^N(\Graph)>0$ is still an eigenvalue of the graph
  $\NewGraph$ and it must be the second one since the first one is zero.
\end{proof}

%%%%%%%%%%%%
\subsection{Proof of Theorem~\ref{thm:transferring_vol}}

The following lemma will be useful for the proof.

\begin{lemma}
  \label{lem:variance_comparison}
  Let $F$ and $G$ be two real-valued functions defined on a probability space $X$. 
  If $F$ has zero mean and for almost all $x\in X$ either 
  \begin{equation}
    \label{eq:G_outside}
    0 \leq F(x) \leq G(x)
    \qquad\mbox{or}\qquad
    G(x) \leq F(x) \leq 0,
  \end{equation}
  then
  \begin{equation}
    \label{eq:variance_comparison}
    \var(F) \leq \var(G).
  \end{equation}
  If $0 < |F(x)| < |G(x)|$ on a set of non-zero measure, then
  \eqref{eq:variance_comparison} is strict.
\end{lemma}

\begin{proof}
  The conditions on $F$ and $G$ can be rewritten as $F(G-F) \geq 0$ a.e. We now estimate
  \begin{align*}
    \var(G) &= \E(G - \E G)^2 
    = \E\big(F + (G - F - \E G) \big)^2 \\
    &= \E\big(F^2 + 2F(G-F) - 2F \E G + (G-F-\E G)^2 \big) \\
    &= \E F^2 + 2\E F(G-F) + \E(G-F-\E G)^2
    \geq \E F^2 = \var(F),
  \end{align*}
  where we used $\E F = 0$ to get to the last line. If $0<|F|<|G|$ on a 
  set of positive measure, then also $F(G-F)>0$ on the same set, implying that 
  the last inequality in the above calculation is strict.
\end{proof}

\begin{proof}[Proof of Theorem~\ref{thm:transferring_vol}]
  (\ref{item:transplanting}) Denote by $\psi$ the eigenfunction
  on $\Graph$ satisfying the assumptions of the theorem. Construct a
  test function $\varphi$ on $\NewGraph$ by setting
  \begin{displaymath}
    \varphi (x) = \begin{cases} \psi (x) \qquad &\text{if } x \in 
      \Graph \setminus \{v_1,\ldots,v_k\} = \NewGraph \setminus 
      \bigcup_{i=1}^k \GraphH_i ,\\
      \psi (v_i) &\text{if } x \in \GraphH_i, \,i=1,\ldots,k.\end{cases}
  \end{displaymath}
  Then $\varphi \in H^1 (\NewGraph)$ since it is continuous and piecewise-$H^1$; 
  moreover, it is still zero at any Dirichlet vertices, with
  \begin{equation}
    \label{eq:transplantation-integrals}
    \int_{\NewGraph} |\varphi|^2\,\textrm{d}x \geq \int_{\Graph} |\psi|^2\,\textrm{d}x,
    \qquad \int_{\NewGraph} |\varphi'|^2\,\textrm{d}x \leq
    \int_{\Graph} |\psi'|^2\,\textrm{d}x.
  \end{equation}
  In particular, if $\mu = \lambda_1$, we see $\varphi$ is a valid
  test function with smaller Rayleigh quotient than $\psi$; thus
  $\lambda_1 (\NewGraph) \leq \lambda_1 (\Graph)$. 

  If $\mu = \lambda_2^N$, then $\varphi$ will not in general be a valid test 
  function since $\int_{\NewGraph} \varphi \,\textrm{d}x \neq 0$. Assume 
  without loss of generality that $|\Graph|=|\NewGraph|=1$ (otherwise 
  just divide everything by the total length). Noting that $\int_{\Graph} 
  \psi \,\textrm{d}x = 0$, we apply Lemma~\ref{lem:variance_comparison} 
  to $F=\psi$, $G=\varphi$ and $X = L^2 (\Graph) \simeq L^2 (\NewGraph)$ 
  (here we identify $\mathcal{C}$ with $\bigcup_{i=1}^k \GraphH_i$, 
  and spaces of $L^2$-functions on them, arbitrarily, noting that the sets have the same measure). This yields
  \begin{equation}
    \label{eq:variance-application}
    \int_{\Graph} |\psi|^2\,\textrm{d}x \leq \int_{\NewGraph}
    \left[ \varphi - \int_{\NewGraph}\varphi\,\textrm{d}x\right]^2\,\textrm{d}x.
  \end{equation}
  In particular, if we set $\tilde{\varphi}:=\varphi-\int_{\NewGraph}\varphi
  \,\textrm{d}x$, then since also $\tilde{\varphi}'=\varphi'$, we see that 
  \eqref{eq:transplantation-integrals} holds with $\tilde{\varphi}$ in place 
  of $\varphi$. Since $\tilde{\varphi}$ is now a valid test function, as above 
  we conclude that $\lambda_2^N (\NewGraph) \leq \lambda_2^N (\Graph)$. 

  If, in addition to \eqref{eq:smaller_than_v} we have
  $\min_{x\in\mathcal{C}}\psi(x) < \max_{j=1,\ldots,k} \psi(v_j)$, then
  either $\min_{x\in\mathcal{C}}\psi(x) < \min_i \psi(v_i)$ or
  $\max_{x\in\mathcal{C}}\psi(x) < \max_i \psi(v_i)$.  From continuity
  of $\psi$ we conclude that there is a set of non-zero measure on
  which $\psi(x) < \varphi(x)$ and we obtain strict inequality in
  \eqref{eq:transplantation-integrals} or in 
  \eqref{eq:variance-application}, correspondingly (in the latter case
  through the strict version of Lemma~\ref{lem:variance_comparison}).

  (\ref{item:unfolding_parallel}) Let $\psi$ be an eigenfunction
  corresponding to $\mu(\Graph)$.  Denote $\psi(v_1)=a$, $\psi(v_2)=b$
  and assume without loss of generality that $0 \leq |a| \leq b$.  On
  the edge $e_1$ locate an interval $[x_1,x_2]$ such that $\psi$ goes
  from $a$ to $b$ on this interval, while remaining in the range
  $[a,b]$.  To be more specific, one can take
  \begin{displaymath}
  \begin{aligned}
    x_1 &:= \max \{x \in [0,|e_1|]: \psi(x) \leq a \}\\
    x_2 &:= \min \{x \in [x_1,|e_1|]: \psi(x) \geq b \}.
  \end{aligned}
  \end{displaymath}

  Change the value of the function on the interval $[x_1,x_2]$ to $b$,
  thus increasing its variance, see
  Lemma~\ref{lem:variance_comparison}, but creating a discontinuity at
  $x_1$.  Insert the edge $e_2$ at the point $x_1$, restoring the
  continuity.  This process creates a test function $\varphi_1$ with
  greater variance.  We can repeat this process, further absorbing the
  edges $e_3,\ldots, e_k$ and eventually obtaining the test function
  $\varphi := \varphi_{k-1}$ on $\NewGraph$.  As in the proof of
  part~(\ref{item:transplanting}), inequalities
  \eqref{eq:transplantation-integrals}-\eqref{eq:variance-application}
  are satisfied leading to non-increase of $\mu(\Graph)$.  
  
  It is easy to see that the variance is strictly increased by this
  process as long as $a\neq b$.  To prove strict inequality in the
  case when $\varphi(v_1)=\varphi(v_2)$ for every eigenfunction
  $\varphi$ associated with $\mu(\Graph)$, suppose that there is at
  least one eigenfunction, call it $\psi$, which is non-constant on
  $e_1\cup \ldots \cup e_k$.  By working recursively, we may assume 
  that $k=2$, and in fact, by Corollary~\ref{cor:joining_points} we 
  may also assume without loss of generality that $v_1 = v_2 =: v$, 
  and $e_1$ and $e_2$ are loops at $v$.

  Then $\NewGraph$ is formed from $\Graph$ by cutting through $v$ to
  create a longer loop $e_0$ out of $e_1$ and $e_2$,
  cf.~Figure~\ref{fig:unfolding-cut}.  By
  Theorem~\ref{thm:changing_vc}, the only way we can have
  $\mu(\NewGraph)=\mu(\Graph)$ is if the image of $\psi$ under this
  cut, which we will still call $\psi$, is also an eigenfunction for
  $\mu (\NewGraph)$ on $\NewGraph$; in particular, at the point $v'$
  on $e_0$ which is the image of $v$ under the cut, it takes the value
  $\psi(v') = \psi(v)$.

\begin{figure}[H]
\begin{minipage}[l]{3.5cm}
\begin{tikzpicture}[scale=0.6]
\foreach \x in {150,180,210}{
\draw[fill] (\x:0cm) -- (\x:3.1cm);
\draw[fill] (\x:3.2cm) -- (\x:3.3cm);
\draw[fill] (\x:3.4cm) -- (\x:3.5cm);
\draw[fill] (\x:3.6cm) -- (\x:3.7cm);
\draw[fill] (0:0cm) circle (2pt) node[below]{$v$};
}
\draw (0,0) .. controls (2,1) and (1,2) .. (0,0);
\draw (0,0) .. controls (1.5,-2.5) and (2.5,-1.5) .. (0,0);
\coordinate (a) at (0,0);
\draw[fill=red] (a) circle (2pt);
\node at (-3.5,0.7) [anchor=west] {$\Graph$};
\node at (0.2,1.5) [anchor=north] {$e_1$};
\node at (2,-2) [anchor=east] {$e_2$};
\end{tikzpicture}
\end{minipage}
\begin{minipage}[l]{1.5cm}
\begin{tikzpicture}[scale=0.8]
\draw[-{Stealth[scale=0.5,angle'=60]},line width=2.5pt] (4,0) -- (5,0);
\end{tikzpicture}
\end{minipage}
\begin{minipage}[l]{3.5cm}
\begin{tikzpicture}[scale=0.6]
\foreach \x in {150,180,210}{
\draw[fill] (\x:0cm) -- (\x:3.1cm);
\draw[fill] (\x:3.2cm) -- (\x:3.3cm);
\draw[fill] (\x:3.4cm) -- (\x:3.5cm);
\draw[fill] (\x:3.6cm) -- (\x:3.7cm);
\draw[fill] (0:0cm) circle (2pt) node[below]{$v$};
}
\draw[rounded corners=1ex] (0,0) .. controls (1,2) and (2,1) .. (.5,0) .. controls (2,-1) and (1,-2) .. (0,0);
\coordinate (a) at (0,0);
\draw[fill=red] (a) circle (2pt);
\node at (-3.5,0.7) [anchor=west] {$\NewGraph$};
\node at (1.3,1) [anchor=west] {$e_0$};
\draw[fill=red] (0.65,0) circle (2pt);
\node at (0.65,0) [anchor=west] {$v'$};
\end{tikzpicture}
\end{minipage}
\caption{The graph $\Graph$ in which $v_1=v_2=v$ (left) and the graph 
$\NewGraph$ obtained by cutting through $v$ (right).}
\label{fig:unfolding-cut}
\end{figure}

  Now since $\psi$ is non-constant on the loop $e_0$ and 
  it takes the same value at least twice on this loop, we 
  can conclude that $\mu \geq 4\pi^2 / |e_0|^2 = \lambda_2^N (e_0)$. 
  Moreover, as $\psi$ will change sign on $e_0$, it cannot
  correspond to the first eigenvalue and we may assume we are dealing
  with $\mu(\Graph) = \lambda_2^N(\Graph)$.
  
  However, $\NewGraph$ can be viewed as being formed by attaching a
  pendant to the loop $e_0$ at $v$. Since $\lambda_2^N (e_0)$ has an
  eigenfunction not vanishing at $v$,
  Theorem~\ref{thm:increasing_vol}(\ref{item:attaching_a_pendant})
  yields $\lambda_2^N(\NewGraph) < \lambda_2^N(e_0) \leq 
  \lambda_2^N (\Graph)$. The only exception is when the pendant is 
  empty, that is, $\NewGraph = e_0$.

  (\ref{item:symm_parallel}) The proof uses a symmetrisation
  argument, which is an easy variant of one that has appeared several
  times in the literature \cite{Fri_aif05},
  \cite[Theorem~2.1]{BanLev_ahp17},
  \cite[Theorem~3.4]{BeKeKuMu_arx17}.

  We denote by $\mathcal{P}_1,\ldots,\mathcal{P}_n$ the pumpkin
  subgraphs to be symmetrised, by $e_{1i},\ldots,e_{k_ii}$ the edges
  of $\mathcal{P}_i$ and by $\tilde{e}_{1i},\ldots, \tilde{e}_{m_ii}$
  the edges of the symmetrised pumpkin $\widetilde{\mathcal{P}}_i$,
  $i=1,\ldots,n$.  Suppose the two vertices of $\mathcal{P}_i$ are
  $v_i^-$ and $v_i^+$ (where $\psi(v_i^-) \leq \psi(v_i^+)$ and it is
  possible that $v_i^\pm = v_j^\pm$ for some $i\neq j$).

  We now construct a test function $\psi^\ast \in H^1 (\NewGraph)$ out of $\psi$ 
  by symmetrising $\psi$ on each pumpkin separately: we set $\psi^\ast (x) = \psi (x)$ 
  if $x \not\in \mathcal{P}_1 \cup \ldots \cup \mathcal{P}_n$ and, on $\mathcal{P}_i$, 
  if $\psi\geq 0$ on $\mathcal{P}_i$, we define $\psi^\ast (x)$ to be the continuous 
  function such that $\psi^\ast(v_i^-) = \psi (v_i^-)$ and
  \begin{displaymath}
    |\{x \in \tilde{e}_{ji}: \psi^\ast(x)<t\}| = \frac{1}{m_i}
    |\{x \in \mathcal{P}_i: \psi(x)<t \}|
  \end{displaymath}
  for $t\in \R$ and $j=1,\ldots,m_i$.

  The following facts are standard and may be easily checked (cf.\ also the 
  references given above): $\psi^\ast$ is in fact an $H^1$-function such 
  that, by construction, $\psi^\ast (v) = \psi (v)$ for all $v \in \mathcal{V} 
  (\NewGraph) \simeq \mathcal{V} (\Graph)$; in particular, if $\psi$ 
  satisfies a Dirichlet condition at some vertex, then so too does $\psi^\ast$; 
  $\|\psi^\ast\|_{L^2 (\NewGraph)} = \|\psi\|_{L^2 (\Graph)}$,
  \begin{displaymath}
	\int_{\NewGraph} \psi^\ast\,\textrm{d}x = \int_\Graph \psi\,\textrm{d}x,
  \end{displaymath}
  and $\|(\psi^\ast)'\|_{L^2 (\NewGraph)} \leq \|\psi'\|_{L^2 (\Graph)}$. Thus 
  $\psi^\ast$ is a valid test function for $\mu (\NewGraph)$ with a Rayleigh 
  quotient no larger than the one of $\psi$, and so we conclude $\mu (\NewGraph) 
  \leq \mu (\Graph)$.

  If there is at least one pumpkin $\mathcal{P}_i$ on which $\psi$ is not constant,
  then the inequality $\|(\psi^\ast)'\|_{L^2 (\NewGraph)} \leq \|\psi'\|_{L^2 
  (\Graph)}$ is strict unless $k_i = m_i$ and all edges $e_1, \ldots, e_{k_i}$
  have the same length (i.e., strict unless the symmetrisation is trivial).

  (\ref{item:unfolding_pendant}) Since unfolding pendant edges
  can be done recursively, it is enough to prove the statement in the
  case of two pendant edges $e_1$ and $e_2$.  Denote their attachment
  vertex by $v_0$, and their other vertices by $v_1$ and $v_2$, respectively, and
  suppose the edge to be created is $e_0$, $|e_0|=|e_1|+|e_2|$. If
  $\mu(\Graph)=\lambda_1 (\Graph)$ (or else $\lambda_2^N$ has an
  eigenfunction which does not change sign on $e_1 \cup e_2$), then
  this follows by transplanting an edge. More precisely, suppose there is an 
  eigenfunction $\psi$ which is non-negative on $e_1\cup e_2$, and suppose it
  reaches a maximum over $e_1\cup e_2$ at $x_0 \in e_1$. We then
  transplant $e_2$ to $x_0$. The resulting graph is $\NewGraph$, and
  by part (\ref{item:transplanting}), we have $\mu (\NewGraph) \leq \mu(\Graph)$. 
  Moreover, if $\psi$ is not constant on $e_1 \cup e_2$, then 
  obviously $0 \leq \min_{x \in e_2} \psi(x) < \psi(x_0)$, so 
  (\ref{item:transplanting}) yields strict inequality.

  Now suppose $\mu = \lambda_2^N (\Graph)$,
  $\Graph \supsetneq e_1 \cup e_2$, and $\psi$ is an eigenfunction
  which changes sign on $e_1 \cup e_2$.  First we observe the {\it a priori}
  bound 
  \begin{equation}
    \label{eq:apriori_bound}
    \lambda_2^N (\NewGraph) < \lambda_2^N (e_1 \cup e_2),
  \end{equation}
  which is obtained from
  Theorem~\ref{thm:increasing_vol}(\ref{item:attaching_a_pendant}) by
  attaching the pendant
  $\mathcal{R}:= \Graph \setminus (e_1 \cup e_2)$ to the interval
  $e_1 \cup e_2$ at $v_1$.

  First consider the case $\psi(v_0)=0$.  If $\psi \equiv 0 $ on
  $e_2$, then $\psi$ changes sign on $e_1$ and takes zero value at
  $v_0$, and we get
  $\mu > \lambda_2^N (e_1) > \lambda_2^N (e_1 \cup e_2)$. By
  \eqref{eq:apriori_bound}, we are done.  Similarly we treat the case
  of $\psi$ vanishing on $e_1$.  If $\psi(v_0)=0$, but $\psi$ is not
  identically zero on either $e_1$ or $e_2$, then we can create an 
  eigenfunction of $e_1 \cup e_2$ by possibly multiplying $\psi$ by a 
  constant on $e_1$ or $e_2$. Since this does not affect the eigenvalue, 
  we obtain $\mu \geq \lambda_2^N (e_1 \cup e_2)$ and we can again 
  invoke \eqref{eq:apriori_bound}.

  So suppose $\psi(v_0)\neq 0$. Our strategy of proof is as follows:
  we cut $\Graph$ at $v_0$ along $\psi$ to create a path graph
  (interval) out of $e_1 \cup e_2$ with a $\delta$-condition at $v_0$
  (cf.~Definition~\ref{def:cutting}). We either remove this
  $\delta$-condition or shift it to an endpoint of the path, which
  will lower the eigenvalue. We then re-glue the path to the rest of
  $\Graph$ to create $\NewGraph$.

  We denote by $\mathcal{I}_{v_0,\gamma}$ the (quantum) graph
  consisting of $e_1 \cup e_2$ with a $\delta$-potential of strength
  $\gamma$ at $v_0$ and natural conditions at $v_1$ and $v_2$.  We
  obtain the value of $\gamma$ from the eigenfunction $\psi$ as
  described in Definition~\ref{def:cutting}.  By
  $\mathcal{R}_{v_0,-\gamma}$ we denote the graph $\Graph$ with $e_1$
  and $e_2$ removed, and with a $\delta$-potential $-\gamma$ at $v_0$
  and natural conditions elsewhere; see Figure~\ref{fig:pendant-split}.
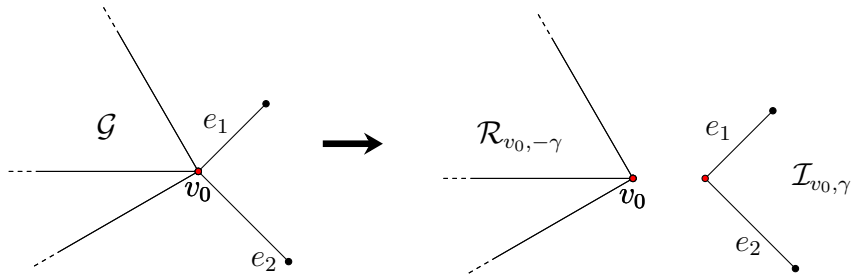
\begin{figure}[H]
\begin{minipage}[l]{4cm}
\begin{tikzpicture}[scale=0.6]
\foreach \x in {120,180,210}{
\draw[fill] (\x:0cm) -- (\x:3.6cm);
\draw[fill] (\x:3.7cm) -- (\x:3.8cm);
\draw[fill] (\x:3.9cm) -- (\x:4cm);
\draw[fill] (\x:4.1cm) -- (\x:4.2cm);
\draw[fill] (0:0cm) circle (2pt) node[below]{$v_0$};
}
\coordinate (a) at (0,0);
\coordinate (b) at (1.5,1.5);
\coordinate (c) at (2,-2);
\draw (a) -- (b);
\draw (a) -- (c);
\draw[fill=red] (a) circle (2pt);
\draw[fill] (b) circle (2pt);
\draw[fill] (c) circle (2pt);
\node at (-2.5,1.5) [anchor=north west] {$\mathcal G$};
\node at (0.4,1.5) [anchor=north] {$e_1$};
\node at (c) [anchor=east] {$e_2$};
\end{tikzpicture}
\end{minipage}
\begin{minipage}{1.5cm}
\begin{tikzpicture}[scale=0.8]
\draw[-{Stealth[scale=0.5,angle'=60]},line width=2.5pt] (4,0) -- (5,0);
\end{tikzpicture}
\end{minipage}
\begin{minipage}[l]{4cm}
\begin{tikzpicture}[scale=0.6]
\foreach \x in {120,180,210}{
\draw[fill] (\x:0cm) -- (\x:3.6cm);
\draw[fill] (\x:3.7cm) -- (\x:3.8cm);
\draw[fill] (\x:3.9cm) -- (\x:4cm);
\draw[fill] (\x:4.1cm) -- (\x:4.2cm);
\draw[fill] (0:0cm) circle (2pt) node[below]{$v_0$};
}
\draw (1.6,0) -- (3.1,1.5);
\draw (1.6,0) -- (3.6,-2);
\draw[fill=red] (0,0) circle (2pt);
\draw[fill] (3.1,1.5) circle (2pt);
\draw[fill] (3.6,-2) circle (2pt);
\draw[fill=red] (1.6,0) circle (2pt);
\node at (-3.7,1.5) [anchor=north west] {$\mathcal{R}_{v_0,-\gamma}$};
\node at (1.9,1.5) [anchor=north] {$e_1$};
\node at (3.1,-1.5) [anchor=east] {$e_2$};
\node at (3.3,0) [anchor=west] {$\mathcal{I}_{v_0,\gamma}$};
\end{tikzpicture}
\end{minipage}
\caption{The graph $\Graph$ (left), and the graphs $\mathcal{R}_{v_0,-\gamma}$ 
and $\mathcal{I}_{v_0,\gamma}$ created by cutting through $v_0$ (right).}
\label{fig:pendant-split}
\end{figure}  

  First suppose that $\gamma \geq 0$.  We have
  $\mu \geq \lambda_2\left(\mathcal{I}_{v_0,\gamma}\right)$ because
  $\psi$ is an eigenfunction for $\mathcal{I}_{v_0,\gamma}$ which 
  changes sign on $e_1 \cup e_2$.  But by
  Theorem~\eqref{eq:interlacing}, we have
  $\lambda_2\left(\mathcal{I}_{v_0,\gamma}\right) \geq 
  \lambda_2\left(\mathcal{I}_{v_0,0}\right)
  = \lambda_2^N(e_1 \cup e_2)$
  and we invoke \eqref{eq:apriori_bound} to complete this
  case.\footnote{One can in fact show that the case $\gamma \geq 0$ 
  is impossible under the assumption that $\psi$ is a second eigenfunction that
    changes sign on $e_1 \cup e_2$.}

  Finally consider the case $\gamma < 0$.  We will need the following
  auxiliary result, which we state and prove in greater generality than we
  require for the proof.  Denote by $\ell$ the length $|e_1\cup e_2|$
  of the interval graph $\mathcal{I}$.
\begin{lemma}
\label{lem:delta-shift}
With the above notation and with $\gamma<0$, the function
$x \mapsto \lambda_2 (\mathcal{I}_{x,\gamma})$ is strictly monotonically
increasing in $x \in [0,\ell/2]$. In particular, it reaches its unique
minimum at $x=0$.
\end{lemma}

\begin{proof}
Firstly, by standard Sturm--Liouville theory the eigenvalue is simple, and its 
eigenfunction $\psi$ has exactly one zero in $[0,\ell]$, say at $z$. 
Moreover, an argument using strict monotonicity with respect to domain inclusion 
and changes of $\gamma$ shows that $\psi$ cannot have an interior 
extremum. We will show using the Hadamard formula of 
Remark~\ref{rem:hadamard} that
\begin{equation}
\label{eq:delta-shift}
	\frac{d}{dx} \lambda_2 (\mathcal{I}_{x,\gamma}) > 0
\end{equation}
for all $x \in (0,\ell/2)$, which will prove the lemma.\footnote{Similar arguments, 
albeit with a different goal, will be used repeatedly in Section~\ref{sec:pumpkins}.} 
Now since increasing $x$ is equivalent to lengthening the edge $[0,x]$ and shortening 
$[x,\ell]$ by the same amount, invoking \eqref{eq:hadamard_formula} and 
\eqref{eq:simple-hadamard} and using the continuity of $\psi$ at $x$, we see 
\eqref{eq:delta-shift} is equivalent to
\begin{equation}
\label{eq:delta-derivative-comparison}
	(\psi_-)'(x)^2 < (\psi_+)'(x)^2,
\end{equation}
where $(\psi_\pm)'(x)$ denote the left ($-$) and right ($+$) derivatives of $\psi$ 
at $x$. We claim that the zero $z \in (x,\ell)$. Indeed, if not, then, since 
$\gamma<0$,
\begin{displaymath}
  \lambda_2 (\mathcal{I}_{x,\gamma}) = \lambda_1^D (0,z) = \lambda_2^N (0,2z) 
  \geq \lambda_2^N (0,\ell) > \lambda_2 (\mathcal{I}_{x,\gamma}).
\end{displaymath}
So suppose that $\psi(x)>0$; then clearly $(\psi_-)'(x), (\psi_+)'(x) < 0$. The 
vertex condition at $x$ together with $\gamma < 0$ now yields 
\eqref{eq:delta-derivative-comparison} and hence \eqref{eq:delta-shift}.
\end{proof}

  We now return to the proof of
  Theorem~\ref{thm:transferring_vol}(\ref{item:unfolding_pendant}). It
  follows using Lemma~\ref{lem:delta-shift} that
  \begin{displaymath}
    \mu = \lambda_2 (\mathcal{I}_{v_0,\gamma}) 
    > \lambda_2 (\mathcal{I}_{v_1,\gamma}),
  \end{displaymath}
  where in $\mathcal{I}_{v_1,\gamma}$ the $\delta$-potential has been
  shifted from $v_0$ to $v_1$, the degree one vertex of $e_1$. We now glue
  $\mathcal{I}_{v_1,\gamma}$ to $\mathcal{R}_{v_0,-\gamma}$ to create
  $\NewGraph$. By Theorem~\ref{thm:changing_vc}, we have
  \begin{equation}
    \label{eq:pendant-comparison}
    \lambda_2^N (\NewGraph) 
    \leq \lambda_3 (\mathcal{R}_{v_0,-\gamma} \,\dot{\cup}\, 
    \mathcal{I}_{v_1,\gamma}) \leq \mu = \lambda_2^N (\Graph),
  \end{equation}
  where the second inequality follows from $\mu > 
  \lambda_2 (\mathcal{I}_{v_1,\gamma})$ and 
  $\mu \in \spec\left(\mathcal{R}_{v_0,-\gamma}\right)$. This establishes 
  the desired inequality in this case. To show that it is actually strict, assume 
  the contrary: there is equality throughout \eqref{eq:pendant-comparison}. 
  Then we must also have 
  \begin{equation}
    \label{eq:pendant_contradiction}
    \lambda_2 (\mathcal{I}_{v_1,\gamma})
    = \lambda_2 \left( \mathcal{R}_{v_0,-\gamma} \,\dot{\cup}\, 
    \mathcal{I}_{v_1,\gamma}\right)
    < \lambda_2^N (\NewGraph) 
    = \lambda_3\left(\mathcal{R}_{v_0,-\gamma} \,\dot{\cup}\, 
    \mathcal{I}_{v_1,\gamma}\right) = \mu.
  \end{equation}
  This puts us in the circumstances of Remark~\ref{rem:catch_up} and
  we conclude that every eigenfunction of
  $\lambda_3\left(\mathcal{R}_{v_0,-\gamma} \,\dot{\cup}\,
    \mathcal{I}_{v_1,\gamma}\right)$
  is an eigenfunction of $\NewGraph$.  But the former eigenspace
  contains a function which is equal to $\psi$ on $\mathcal{R}$ and to
  zero on $\mathcal{I}$ and cannot be glued continuously since
  $\psi(v_0)\neq 0$.  This final contradiction concludes the proof.
\end{proof}

%%%%%%%%%%%%%%%
\section{Pumpkins everywhere}
\label{sec:pumpkins}

We are now ready to show how various surgery principles can be combined 
to allow a fine spectral analysis of one's graph. In this section, we will concentrate 
on the particular classes of \emph{pumpkins} and \emph{pumpkin chains} (see 
Definition~\ref{def:examples-of-graphs}), for three reasons: firstly, they are 
good examples on which to illustrate the principles; secondly, they will be 
used in a central way in our principal application in Section~\ref{sec:sizeof}; and 
thirdly (which also largely explains the first two), at least when
considering the first non-trivial eigenvalue $\mu(\Graph)$, these graphs are 
in a certain sense generic, as we shall now demonstrate using our surgery principles.

\subsection{Pumpkin chains}

We recall that a $[m_1,\ldots,m_n]$-pumpkin chain consists of vertices $v_1,\ldots, v_{n +1}$ 
together with, for each $k=1,\ldots,n$, some number $m_k$ of parallel edges running 
between $v_k$ and $v_{k+1}$. We will say that a (continuous) function $\psi$ defined 
on the given pumpkin chain is \emph{monotonically increasing along the chain} 
(or \emph{monotonically increasing} for short)  
if it satisfies $\psi(v_1) \leq \psi(v_2)\leq \ldots \leq \psi(v_{n+1})$, and 
on each edge connecting $v_k$ to $v_{k+1}$, $\psi$ is monotonically increasing 
from $v_k$ to $v_{k+1}$, $k=1,\ldots,n$. We call $\psi$ \emph{monotonically decreasing} (along 
the chain) if it satisfies the reverse inequalities, and \emph{monotonic} if it is monotonically 
increasing or decreasing.

We start out by showing that for any graph $\Graph$ there is a naturally associated 
pumpkin chain with the same first non-trivial eigenvalue.

\begin{lemma}
\label{lem:pumpkinchain}
Let $\Graph$ be compact and connected, and suppose $\mu(\Graph) \neq 0$. 
Then there is a pumpkin chain $\mathcal{P}_1$ with the same or smaller total length such that
  \begin{displaymath}
	\mu (\Graph) = \mu (\mathcal{P}_1);
  \end{displaymath}
moreover, $\mu (\mathcal{P}_1)$ has an eigenfunction which is monotonically 
increasing along the chain. Finally, if $\mathcal{P}_1$ is a path, then 
either it has shorter total length than $\Graph$, or $\Graph$ is a path itself, in which case they 
coincide.
\end{lemma} 

In fact, if there is an eigenfunction associated with $\mu (\Graph)$ which does not vanish 
identically on any edge of $\Graph$, then the $\mathcal{P}_1$ we construct has the same 
length as $\Graph$. For the proof of the lemma, we first introduce what we shall call 
\emph{critical values} or \emph{critical levels} of a function.

\begin{definition}
\label{def:critical}
Suppose $f \in C(\Graph)$. A point $x \in \Graph$ shall be called a \emph{critical 
point} (of the function $f$) if $x \in \mathcal{V} (\Graph)$ \emph{or} if $f$ attains 
a local maximum or minimum at $x$. A value $t \in \R$ shall be called a 
\emph{critical value} or \emph{critical level} (of $f$) if there exists a 
critical point $x$ of $f$ such that $f(x)=t$. The preimage $\{x \in \Graph: f(x)=t\}$ 
of a critical value $t$ will also be called the \emph{critical set} (of $f$ at the critical 
level $t$).
\end{definition}

Note that since $\Graph$ is compact any \emph{eigenfunction} has only
finitely many critical levels; and, as long as $\mu(\Graph)\neq 0$, 
apart possibly from $t=0$ every corresponding critical set is also finite.

\begin{proof}[Proof of Lemma~\ref{lem:pumpkinchain}]
  Fix any eigenfunction $\psi$ associated with $\mu := \mu (\Graph)$. We may assume 
  $\psi$ does not vanish identically on any edge of the remaining connected 
  component by shrinking the edge to zero if necessary, which decreases the total 
  length without affecting $\mu$ or $\psi$ by 
  Corollary~\ref{cor:increasing_vol}(\ref{item:shrinking_redundant}). Thus each 
  critical set of $\psi$ is finite; we already know that $\psi$ has finitely many critical 
  sets. By inserting dummy vertices as necessary we may assume that every point in 
  every critical set is a vertex. For each critical set we glue together all the vertices 
  belonging to it; Corollary~\ref{cor:joining_points} implies this leaves $\mu$ and 
  $\psi$ unchanged. The new graph $\NewGraph$ is, by construction, a 
  pumpkin chain, and $\psi$ is now (strictly) monotonic along the chain. Indeed, $\psi$ 
  does not take on the same value at any two distinct vertices, and if $v_1,v_2,v_3$ 
  are any vertices such that $\psi(v_1) < \psi(v_2) < \psi(v_3)$, then there is no edge 
  from $v_1$ to $v_3$, since otherwise $\psi$ would take on a value on that edge equal 
  to $\psi(v_2)$. It follows in particular that each vertex is connected at most to a 
  predecessor and a successor, that is, $\NewGraph$ is a pumpkin chain.

  Next observe that $\mu = \mu(\NewGraph)$, that is, $\mu$ is the smallest eigenvalue 
  of $\NewGraph$ having a non-trivial eigenfunction. In the case $\mu = \lambda_2^N 
  (\Graph)$, then $\lambda_2^N (\Graph) \leq \lambda_2^N (\NewGraph) = \mu 
  (\NewGraph)$ by Theorem~\ref{thm:changing_vc}(\ref{item:gluing_vertices}), but 
  $\psi$ is a non-constant eigenfunction of $\mu$ on $\NewGraph$, so there is equality. 
  The argument if $\mu = \lambda_1 (\Graph)$ and $\mu (\NewGraph) = \lambda_1 
  (\NewGraph)$ is similar. It is impossible for $\mu = \lambda_1 (\Graph)$ and 
  $\mu (\NewGraph) = \lambda_2^N (\NewGraph)$, since then $\psi$ would be a 
  non-constant, non-sign-changing eigenfunction of $\NewGraph$, a contradiction to the 
  theorem of Kre\u{\i}n--Rutman and $\lambda_1^N (\NewGraph)=0$ with only constants 
  as eigenfunctions.

  Finally, if $\Graph$ is not a path graph after the initial ``shrinking'' procedure, then 
  $\mathcal{P}_1$ is also not one, since the above procedure cannot decrease the 
  degree of any vertex; if $\psi$ did not vanish on any edge, then no edge has been 
  shrunk and thus $|\Graph|=|\mathcal{P}_1|$.
\end{proof}

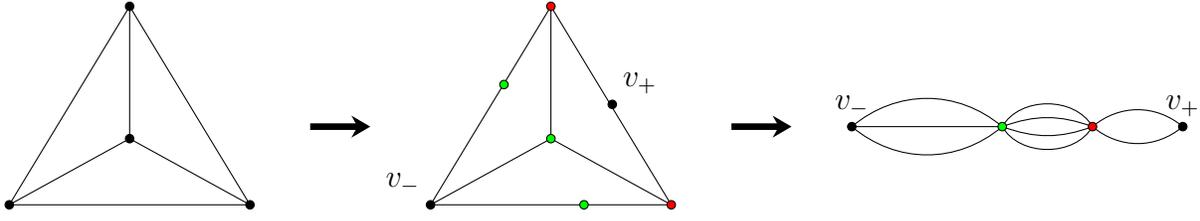
\begin{figure}[h]
\begin{tikzpicture}[scale=0.8]
\coordinate (v) at (-6,-1.3);
\coordinate (w) at (-2,-1.3);
\coordinate (y) at (-4,-.2);
\coordinate (z) at (-4,2);
\draw[fill] (v) circle (2pt);
\draw[fill] (w) circle (2pt);
\draw[fill] (y) circle (2pt);
\draw[fill] (z) circle (2pt);
\draw (v) -- (w);
\draw (w) -- (y);
\draw (y) -- (z);
\draw (z) -- (v);
\draw (v) -- (y);
\draw (w) -- (z);
\draw[-{Stealth[scale=0.5,angle'=60]},line width=2.5pt] (-1,0) -- (0,0);

\coordinate (a) at (1,-1.3);
\coordinate (b) at (5,-1.3);
\coordinate (c) at (3,-.2);
\coordinate (d) at (3,2);
\coordinate (e) at (4.02,.37);
\coordinate (e1) at (3.55,-1.3);
\coordinate (e2) at (2.22,.7);
\draw (a) -- (b);
\draw (b) -- (c);
\draw (c) -- (d);
\draw (a) -- (d);
\draw (a) -- (c);
\draw (b) -- (d);

\draw[fill] (a) circle (2pt);
\draw[fill=red] (b) circle (2pt);
\draw[fill=green] (c) circle (2pt);
\draw[fill=red] (d) circle (2pt);
\draw[fill] (e) circle (2pt);
\draw[fill=green] (e1) circle (2pt);
\draw[fill=green] (e2) circle (2pt);

\node at (a) [anchor=south east ] {$v_-$};
\node at (e) [anchor=south west] {$v_+$};

\draw[-{Stealth[scale=0.5,angle'=60]},line width=2.5pt] (6,0) -- (7,0);

\coordinate (g) at (8,0);
\coordinate (h) at (10.5,0);
\coordinate (i) at (12,0);
\coordinate (j) at (13.5,0);

\draw[bend left=60] (h) edge (i);
\draw[bend left=20] (h) edge (i);
\draw[bend left=-60] (h) edge (i);
\draw[bend left=-20] (h) edge (i);
\draw[bend left=40] (g) edge (h);
\draw (g) -- (h);
\draw[bend right=40] (g) edge (h);
\draw[bend right=40] (i) edge (j);
\draw[bend right=-40] (i) edge (j);

\node at (g) [anchor=south] {$v_-$};
\node at (j) [anchor=south] {$v_+$};

\draw[fill] (g) circle (2pt);
\draw[fill=green] (h) circle (2pt);
\draw[fill=red] (i) circle (2pt);
\draw[fill] (j) circle (2pt);
\end{tikzpicture}
\caption{Turning the quantum graph constructed upon the complete graph $K_4$ 
into a pumpkin chain. The new vertices appearing in the middle graph are those 
points where the eigenfunction attains the same values as in the central vertex 
together with the point $v_+$ where it attains its maximum.}\label{fig:creating-pumpkins}
\end{figure}

The monotonicity of the eigenfunction of $\mathcal{P}_1$ means 
we are now in a position to symmetrise each pumpkin using 
Theorem~\ref{thm:transferring_vol}(\ref{item:symm_parallel}) 
(or (\ref{item:unfolding_parallel})). Note the contrast to 
\cite[Lemma~5.4]{KeKuMaMu_ahp16}, which shows that $\lambda_2^N$ can be 
bounded from \emph{above} by the corresponding eigenvalue of a locally equilateral 
pumpkin chain with the same diameter (but generally smaller total length).

\begin{lemma}
\label{lem:locally-equilateral}
  Let $\Graph$ be compact and connected, 
  and suppose $\mu (\Graph) \neq 0$.
  Then there is a locally equilateral pumpkin chain $\mathcal{P}_2$ with 
  the same or smaller total length such that
  \begin{displaymath}
	\mu (\Graph) \geq \mu (\mathcal{P}_2).
  \end{displaymath}
\end{lemma}

\begin{proof}
  By Lemma~\ref{lem:pumpkinchain} we may assume that $\Graph$ is already a 
  pumpkin chain and $\mu(\Graph)$ has an eigenfunction which is monotonic along 
  the chain. Now apply Theorem~\ref{thm:transferring_vol}(\ref{item:symm_parallel}) to all constituent 
  pumpkins simultaneously with $k=m$ in each case.
\end{proof}

Before proceeding, we wish to give some basic properties of locally equilateral pumpkin chains.
First, we look at a decomposition of the corresponding eigenspaces. 
Given a locally equilateral pumpkin chain $\mathcal P$ 
with terminal vertices $v_-,v_+$, we stipulate the following:
\begin{enumerate}
\item a function on $\mathcal P$ is called \emph{longitudinal} if it depends only on $\dist (\,\cdot\,,v_-)$;
\item a function on $\mathcal P$ is called \emph{transversal} if it is supported on exactly one pumpkin.
\end{enumerate}

\begin{lemma}
\label{lem:long-trans-decomp}
Let $\Graph$ be a locally equilateral pumpkin chain, and assume that no vertices apart 
possibly from the terminal vertices are equipped with a Dirichlet condition. Then
\begin{enumerate}
\item \label{item:long-trans-decomp} $L^2(\mathcal G)$ has an orthonormal basis 
consisting of longitudinal and transversal eigenfunctions, such that each transversal 
eigenfunction is supported on exactly one pair of parallel edges; and
\item \label{item:long-properties} there is an infinite sequence of eigenvalues having 
longitudinal eigenfunctions, and for each such eigenvalue the span of the longitudinal 
eigenfunctions in the corresponding eigenspace is one-dimensional.
\end{enumerate}
\end{lemma}

\begin{proof}
(\ref{item:long-trans-decomp}) Suppose $\psi$ is any eigenfunction. Define its 
\emph{longitudinal part} $\psi^{\textrm{lon}}$  by averaging the value of $\psi$ over all 
parallel edges, i.e., if $\mathcal{P}_i$ is any constituent pumpkin of $\Graph$, which 
itself consists of the edges $e_j$, $j=1,\ldots,{k_i}$, then we set
\begin{displaymath}
	\psi^{\textrm{lon}}|_{e_j}(x) := \frac{1}{k_i}\sum_{e \in \mathcal{P}_i} \psi|_{e}(x),\quad x\in \mathcal P_i.
\end{displaymath}
Then $\psi^{\textrm{lon}}$, if it is non-zero, is still an eigenfunction with the same 
eigenvalue as $\psi$ since it still satisfies the eigenvalue equation pointwise, 
$\psi^{\textrm{lon}}(v)=\psi(v)$ at every vertex $v$ of $\Graph$, and all three vertex 
conditions, Dirichlet, natural and $\delta$, are preserved by the averaging process. 
Indeed, it follows immediately from the definition that at any vertex $v$, supposing 
that $\mathcal{P}_i$ is an incident pumpkin with $k_i$ edges, then
\begin{displaymath}
	\sum_{e \in \mathcal{P}_i} \partial_\nu \psi^{\textrm{lon}}|_e (v)
	= \sum_{e \in \mathcal{P}_i} \partial_\nu \left(\frac{1}{k_i}\sum_{e \in \mathcal{P}_i}\psi|_e \right) (v)
	= \frac{1}{k_i} \sum_{e \in \mathcal{P}_i} \sum_{e \in \mathcal{P}_i} \partial_\nu \psi|_e (v)
	= \sum_{e \in \mathcal{P}_i} \partial_\nu \psi|_e (v),
\end{displaymath}
and thus Kirchhoff and $\delta$ conditions remain satisfied in the strong sense. 
Moreover, by construction, $\psi^{\textrm{lon}}$ is longitudinal. Since the function 
$\psi - \psi^{\textrm{lon}}$ vanishes at all vertices of $\Graph$, it is within the span of
transversal eigenfunctions each of which is supported on just one pair of parallel edges. 
A Gram--Schmidt process completes the proof of (\ref{item:long-trans-decomp}).

(\ref{item:long-properties}) The existence of infinitely many such eigenvalues follows 
since the problem corresponds to a (one-dimensional) Sturm--Liouville problem with 
non-smooth but piecewise constant weight function, possibly with a finite number of 
$\delta$ potentials, cf.~\cite[Section~5.2]{KeKuMaMu_ahp16}. In particular, the 
simplicity of each eigenvalue within the space of longitudinal functions follows from 
basic Sturm--Liouville theory.
\end{proof}

Similar ideas have been developed, for example, in~\cite[Theorem~8.30
and \S~8.3.1]{Mug14} in a more general context that may lack a
longitudinal direction and therefore not allow for a one-dimensional
reduction. Longitudinal eigenfunctions correspond to the trivial representation 
of the symmetry of exchanging edges within each pumpkin, cf.~\cite{BaBeJoLi17}.

\begin{lemma}
\label{lem:simple-chain}
Let $\Graph$ be a locally equilateral pumpkin chain, such that 
no vertices other than the terminal ones may be equipped with Dirichlet conditions. 
Additionally, assume that $\Graph$ is not a pumpkin with all natural conditions. 
Then the first non-trivial eigenvalue $\mu(\Graph)$ is simple. 
The corresponding eigenfunction $\psi$ is the first non-constant 
longitudinal eigenfunction. In particular, if all vertices of $\Graph$ 
are equipped with natural conditions except possibly one of the terminal vertices, 
which may be equipped with an arbitrary $\delta$-potential $\gamma \in (-\infty, 
\infty]$, then $\psi$ may be chosen to be monotonically increasing along the chain.
\end{lemma}

\begin{proof}
Let $\lambda^\ast = \lambda^\ast (\Graph)$ be the smallest eigenvalue having a 
non-constant longitudinal eigenfunction, call it $\psi$. If not all vertex conditions are 
natural, i.e., we are considering $\mu = \lambda_1$, then $\psi$ does not change 
sign, meaning $\lambda^\ast = \lambda_1$, and this eigenvalue is simple in the 
spectrum of $\Graph$.

Now consider the case of only natural conditions, i.e., $\mu = \lambda_2^N$. In 
this case, by invoking the decomposition of the spectrum and simplicity of longitudinal 
eigenvalues established in Lemma~\ref{lem:long-trans-decomp}, it suffices to show 
that $\lambda^\ast = \lambda_2^N$ and that no transversal eigenfunction has the 
same eigenvalue.

To this end, first observe that the eigenvalue of any transversal eigenfunction is always 
a non-zero eigenvalue of one of the constituent pumpkins of $\Graph$. If we denote the 
longest edge length in $\Graph$ by $|e_{\max}|$, then the smallest of these is 
$\pi^2/|e_{\max}|^2$. Hence, to complete the proof, we merely have to show that 
$\lambda_2^N (\Graph) < \pi^2/|e_{\max}|^2$.

Denote by $\mathcal{P} \subsetneq \Graph$ any constituent pumpkin whose smallest 
non-trivial eigenvalue equals $\pi^2/|e_{\max}|^2$; call its vertices $v_1$ and $v_2$. 
Then $\Graph$ may be formed from $\mathcal{P}$ by attaching the pendant(s) 
$\Graph \setminus \mathcal{P}$ to $\mathcal{P}$ at $v_1$ and/or $v_2$ as 
appropriate. Since $\lambda_2^N(\mathcal{P})$ has an eigenfunction which is 
non-zero at $v_1$ and at $v_2$, applying 
Theorem~\ref{thm:increasing_vol}(\ref{item:attaching_a_pendant}) with $r=1$ 
yields
\begin{displaymath}
	\lambda_2^N (\Graph) < \lambda_2^N (\mathcal{P}) = \frac{\pi^2}{|e_{\max}|^2}
\end{displaymath}
Now since $\psi$ corresponds to the first non-trivial eigenfunction of a one-dimensional 
Sturm--Liouville problem with $L^\infty$- (indeed, piecewise constant) weights 
(cf., e.g.,~\cite[Section~5.2]{KeKuMaMu_ahp16}), its monotonicity with respect to 
$\dist (\,\cdot\,,v_-)$ is a routine statement from Sturm--Liouville theory.
\end{proof}

With this background, we will now give three particular examples of special classes of 
pumpkin chains. On the one hand, this is a further illustration of what results can be 
obtained using the tools presented in Section~\ref{sec:tools}, in particular both the 
unfolding principles and the Hadamard principle (Remark~\ref{rem:hadamard}). At 
the same time, the examples show how the spectral gap $\lambda_2^N$ is reduced if 
more ``mass'' is concentrated symmetrically at the periphery of the graph, or as the 
pumpkin chain becomes ``thinner'' (more path-like). On the other hand, these examples 
will be needed for our principal application, in Section~\ref{sec:sizeof}.

We remark that the Hadamard-type formula for quantum graphs was used for a comparable 
but complementary analysis in \cite[Section~5]{BanLev_ahp17}, where the goal was to 
study properties of graphs which represented ``critical points'' with respect to this 
formula for a given graph topology, i.e., given an underlying discrete graph, 
to study those graphs whose every edge length was a critical point for $\lambda_2^N$. Here, 
the goal is to see how monotonic behaviour of the eigenfunction can be used to show 
that a continuous change in edge lengths can transform a graph into another one, 
such that $\lambda_2^N$ always increases or decreases under this transformation.

\subsection{Pumpkin-on-a-stick graphs}
\label{sec:pumpkin-on-a-stick}

Here we assume $\mathcal{V}=\mathcal{V}_N$ and consider $\mu=\lambda_2^N$. 
We will consider the following class of graphs.

\begin{definition}
\label{def:pumpkin-on-a-stick}
A pumpkin chain $\Graph$ shall be called a 
\emph{pumpkin-on-a-stick} if it is a $[1,m,1]$-pumpkin chain for some $m\geq 1$, and 
the $m$-pumpkin is equilateral. Here we allow the terminal one-pumpkins to be 
degenerate, i.e.\ have zero length.
\end{definition}

In any case, we refer to the $m$-pumpkin as the (non-trivial) pumpkin, and the union
of the $1$-pumpkins as the stick (cf.~Figure~\ref{fig:pumpkin-on-a-stick}). Note that 
equilateral pumpkins, tadpoles (lassos) and even path graphs are all special cases. We are 
interested in the following parameters, and the behaviour of $\lambda_2^N$ with 
respect to them:
\begin{enumerate}
\item the total length $L$;
\item the number of edges $m$ of the pumpkin;
\item the lengths $\ell_1$ and $\ell_2$ of the $1$-pumpkins (which we think of as the 
``left'' and the ``right'' ones, respectively), as well as the length 
$\ell := \ell_1 + \ell_2$ of the stick.
\end{enumerate}
For a given $L$, which will be fixed throughout, we will denote by 
$\pstick{\ell_1}{m}{\ell_2}$ the pumpkin-on-a-stick whose $1$-pumpkins have 
length $\ell_1$ and $\ell_2$, and whose central pumpkin has $m$ edges. If we 
denote by $T \subset \R^2$ the closed triangle whose vertices are $(0,0)$, 
$(L,0)$ and $(0,L)$, our condition on the $1$-pumpkins reads $(\ell_1,\ell_2) \in 
T$. Up to rigid transformations, any pumpkin-on-a-stick of total length $L$ is 
determined uniquely by the parameters $m\geq 1$ and $(\ell_1,\ell_2) \in T$.

The following proposition gives a complete description of how $\lambda_2^N$ 
depends on these parameters. The proof is based principally on 
Theorem~\ref{thm:transferring_vol}(\ref{item:symm_parallel}) (local symmetrisation) and 
\eqref{eq:hadamard_formula} (the Hadamard-type formula). We emphasise that 
the proof does \emph{not} involve any explicit calculations; in particular, we do not 
use any properties of the corresponding secular equations for the eigenvalues. 
We exclude the trivial case $m=1$ from our considerations.

\begin{figure}[H]
\begin{tikzpicture}[scale=0.8]
\coordinate (a) at (0,0);
\coordinate (b) at (2,0);
\coordinate (c) at (5,0);
\coordinate (d) at (6,0);
\draw[fill] (0,0) circle (2pt);
\draw[fill] (2,0) circle (2pt);
\draw[fill] (5,0) circle (2pt);
\draw (a) -- (b);
\draw[bend left=54] (b) edge (c);
\draw[bend right=54] (b) edge (c);
\draw[bend left=18] (b) edge (c);
\draw[bend right=18] (b) edge (c);
\draw[bend left=90] (b) edge (c);
\draw[bend right=90] (b) edge (c);
\draw (c) -- (d);
\draw[fill] (6,0) circle (2pt);
\coordinate (e) at (9,0);
\coordinate (f) at (12,0);
\coordinate (g) at (15,0);
\draw[fill] (9,0) circle (2pt);
\draw[fill] (12,0) circle (2pt);
\draw[bend left=90] (f) edge (g);
\draw[bend right=90] (f) edge (g);
\draw (e) -- (f);
\node at (a) [anchor=north east] {$v_-$};
\node at (d) [anchor=north west] {$v_+$};
\node at (e) [anchor=north east] {$v_-$};
\node at (g) [anchor=north west] {$v_+$};
\node at (b) [anchor=north east] {$v_1$};
\node at (c) [anchor=north west] {$v_2$};
\node at (1,0) [anchor=south] {$e_1$};
\node at (5.5,0) [anchor=south] {$e_2$};
\end{tikzpicture}
\caption{A ``generic'' pumpkin-on-a-stick (left), with $m=6$ and stick $e_1\cup e_2$, 
where $|e_1|=\ell_1$ and $|e_2|=\ell_2$; and a tadpole (right), with $m=2$ and 
$\ell_2=0$.}
\label{fig:pumpkin-on-a-stick}
\end{figure}
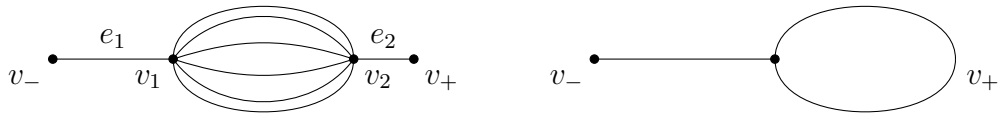

\begin{proposition}
  \label{prop:pumpkin-on-a-stick}
  Let $L$ be fixed and suppose $m \geq 2$ and  $(\ell_1,\ell_2) \in T$. Then 
  $\lambda_2^N (\pstick{\ell_1}{m}{\ell_2})$ is simple with corresponding eigenfunction 
  monotonic from $v_-$ to $v_+$, unless the graph is a pumpkin. For each fixed $m \geq 2$, 
  the function $(\ell_1,\ell_2) \mapsto \lambda_2^N (\pstick{\ell_1}{m}{\ell_2})$ is 
  continuous on the closed triangle $T$. Moreover,
  \begin{enumerate}
  \item \label{item:pumpkin-stick-size} for fixed $m \geq 2$ we have
  \begin{displaymath}
	\lambda_2^N (\pstick{\ell_1}{m}{\ell_2}) > \lambda_2^N (\pstick{\ell_1'}{m}{\ell_2'})
  \end{displaymath}
  whenever $\ell_1 \leq \ell_1'$ and $\ell_2 \leq \ell_2'$ with at least one inequality strict;
  \item \label{item:pumpkin-stick-thickness} for fixed $(\ell_1,\ell_2) \in T$ the function 
  $m \mapsto \lambda_2^N (\pstick{\ell_1}{m}{\ell_2})$ is strictly monotonically decreasing 
  in $m \geq 2$;
  \item \label{item:pumpkin-stick-distance} for fixed $m \geq 2$ and fixed $\ell=\ell_1 + 
  \ell_2$, the function $\ell_1 \mapsto \lambda_2^N (\pstick{\ell_1}{m}{\ell-\ell_1})$ is 
  strictly monotonically increasing in $\ell_1 \in [0,\ell/2]$;
  \item \label{item:pumpkin-stick-min} in particular, among all pumpkin-on-a-stick graphs 
  with fixed $\ell$ and $m$, the minimum of $\lambda_2^N$ is achieved at $\ell_1=0$, and 
  among all pumpkin-on-a-stick graphs with fixed $\ell$, the minimum is achieved at $\ell_1=0$ 
  and $m=2$, i.e.\ at the tadpole whose tail has length $\ell$.
  \end{enumerate}
  Finally, for given $m\geq 2$, $\lambda_2^N (\pstick{\ell_1}{m}{\ell_2})$ satisfies the bound
  \begin{equation}
  \label{eq:tadpole_apriori_bounds}
    \frac{\pi^2}{L^2} \leq \lambda_2^N (\pstick{\ell_1}{m}{\ell_2}) \leq \frac{\pi^2 m^2}{L^2},
  \end{equation}
  with equality in the lower estimate if and only if $\pstick{\ell_1}{m}{\ell_2}$ is a path (i.e.\ $\ell=L$) 
  and in the upper one if and only if $\pstick{\ell_1}{m}{\ell_2}$ is a pumpkin (i.e.\ $\ell=0$).
\end{proposition}

We observe that (\ref{item:pumpkin-stick-size}) contains the statement that the spectral gap of a tadpole is a strictly 
increasing function of the length of its loop (if the total length is fixed); in particular, it runs from 
$\pi^2/L^2$ if the loop has length $0$ to $4\pi^2/L^2$ if the loop has length $L$.

\begin{proof}[Proof of Proposition~\ref{prop:pumpkin-on-a-stick}]
The statements about simplicity and the corresponding eigenfunction were proved 
in Lemma~\ref{lem:simple-chain}. The statements about continuity 
follow from general results about the stability of the spectrum with respect to changes 
in the edge lengths, including in the degenerate case when an edge contracts to zero; 
see for example~\cite{BerKuc_sg12} or~\cite[Appendix~A]{BanLev_ahp17} (or
\cite{BeLaSu_prep18} for general vertex conditions including $\delta$-type). 
For the rest, we will rely primarily on the unfolding 
principles from Theorem~\ref{thm:transferring_vol} and the Hadamard-type formula in the 
form \eqref{eq:hadamard_formula}.

(\ref{item:pumpkin-stick-size})  Consider two pumpkin-on-a-chain graphs 
$\mathcal P = \pstick{\ell_1}{m}{\ell_2}$ and $ \mathcal P' = \pstick{\ell_1'}{m}{\ell_2'}$ 
with, say, $\ell_1 < \ell_1'$ and $\ell_2 \leq \ell_2'$. 
In the graph $\mathcal{P}$, on each parallel edge insert dummy vertices at 
the distance $(\ell_1'-\ell_1)/m$ from the left terminal vertex and at the distance 
$(\ell_2'-\ell_2)/m$ from the right terminal vertex, cf.~Figure~\ref{fig:shifting-pumpkins-2}.

We glue the dummy vertices in such a way as to obtain a locally equilateral 
pumpkin chain with three $m$-pumpkins. This leaves $\lambda_2^N$ unchanged by 
Corollary~\ref{cor:joining_points}.  The two side pumpkins can be unfolded,
leading to the graph $\mathcal{P}'$; 
Theorem~\ref{thm:transferring_vol}(\ref{item:unfolding_parallel}) yields 
the inequality $\lambda_2^N (\mathcal{P}) > \lambda_2^N (\mathcal{P}')$, 
which is strict because the longitudinal eigenfunction corresponding to $\lambda_2^N 
(\mathcal{P})$ is strictly monotonic in the longitudinal direction, and $\ell_1 < \ell_1'$.

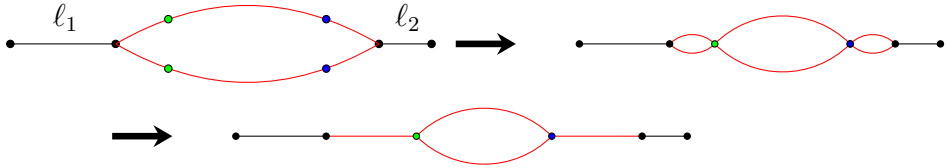
\begin{figure}[H]
\begin{minipage}[l]{5.8cm}
\begin{tikzpicture}[scale=0.7]
\coordinate (a) at (0,0);
\coordinate (b) at (2,0);
\coordinate (b1) at (3,0.47);
\coordinate (b2) at (3,-0.47);
\coordinate (c1) at (6,0.47);
\coordinate (c2) at (6,-0.47);
\coordinate (c) at (7,0);
\coordinate (d) at (8,0);
\draw[fill] (a) circle (2pt);
\draw[fill] (b) circle (2pt);
\draw[fill] (c) circle (2pt);
\draw[fill] (d) circle (2pt);
\draw (a) -- (b);
\draw[red,bend left=30] (b) edge (c);
\draw[red,bend right=30] (b) edge (c);
\draw (c) -- (d);
\draw (1.5,.5) node[left]{$\ell_1$};
\draw (8,.5) node[left]{ $\ell_2$};
\draw[fill=green] (b1) circle (2pt);
\draw[fill=green] (b2) circle (2pt);
\draw[fill=blue] (c1) circle (2pt);
\draw[fill=blue] (c2) circle (2pt);
\end{tikzpicture}
\end{minipage}
\begin{minipage}[l]{1.5cm}
\begin{tikzpicture}[scale=0.8]
\draw[-{Stealth[scale=0.5,angle'=60]},line width=2.5pt] (8,0) -- (9,0);
\end{tikzpicture}
\end{minipage}
\begin{minipage}[l]{5.8cm}
\begin{tikzpicture}[scale=0.6]
\coordinate (a) at (0,0);
\coordinate (b) at (2,0);
\coordinate (b1) at (3,0);
\coordinate (c) at (5,0);
\coordinate (c1) at (6,0);
\coordinate (d) at (7,0);
\coordinate (e) at (8,0);
\draw (a) -- (b);
\draw[red,bend left=45] (b) edge (b1);
\draw[red,bend right=45] (b) edge (b1);
\draw[red,bend right=45] (b1) edge (c1);
\draw[red,bend left=45] (b1) edge (c1);
\draw[red,bend left=45] (c1) edge (d);
\draw[red,bend right=45] (c1) edge (d);
\draw (d) -- (e);
\draw[fill] (a) circle (2pt);
\draw[fill] (b) circle (2pt);
\draw[fill] (c1) circle (2pt);
\draw[fill] (d) circle (2pt);
\draw[fill] (e) circle (2pt);
\draw[fill=green] (b1) circle (2pt);
\draw[fill=blue] (c1) circle (2pt);
\end{tikzpicture}
\end{minipage}

\begin{minipage}[l]{1.5cm}
\begin{tikzpicture}[scale=0.8]
\draw[-{Stealth[scale=0.5,angle'=60]},line width=2.5pt] (8,0) -- (9,0);
\end{tikzpicture}
\end{minipage}
\begin{minipage}[l]{9cm}
\begin{tikzpicture}[scale=0.6]
\coordinate (a) at (-1,0);
\coordinate (b) at (1,0);
\coordinate (b1) at (3,0);
\coordinate (c) at (5,0);
\coordinate (c1) at (6,0);
\coordinate (d) at (8,0);
\coordinate (e) at (9,0);
\draw (a) -- (b);
\draw[red] (b) -- (b1);
\draw[red,bend right=45] (b1) edge (c1);
\draw[red,bend left=45] (b1) edge (c1);
\draw[red] (c1) -- (d);
\draw (d) -- (e);
\draw[fill] (a) circle (2pt);
\draw[fill] (b) circle (2pt);
\draw[fill=green] (b1) circle (2pt);
\draw[fill=blue] (c1) circle (2pt);
\draw[fill] (d) circle (2pt);
\draw[fill] (e) circle (2pt);
\end{tikzpicture}
\end{minipage}
\caption{The first graph $\mathcal{P} = \pstick{\ell_1}{m}{\ell_2}$ is turned 
into an auxiliary graph and finally into $\mathcal{P}' = \pstick{\ell_1'}{m}{\ell_2'}$.}
\label{fig:shifting-pumpkins-2}
\end{figure}

(\ref{item:pumpkin-stick-thickness}) This follows directly from an application of 
Theorem~\ref{thm:transferring_vol}(\ref{item:symm_parallel}) 
to the non-trivial pumpkin, again noting that the eigenfunction does not vanish on the 
pumpkin in question.

(\ref{item:pumpkin-stick-distance}) Here we use our Hadamard formula \eqref{eq:simple-hadamard}. 
First some notation: we label the internal vertices as $v_1$ and $v_2$ as depicted in 
Figure~\ref{fig:pumpkin-on-a-stick}, so that $v_1$ is closer to $v_-$ and $v_2$ is 
closer to $v_+$. We also write $e_1 \sim v_- v_1$ and $e_2 \sim v_2 v_+$, and 
denote by $e_p$ any of the $m$ parallel edges of the central pumpkin.

Fix a pumpkin-on-a-stick and assume that for this graph $\ell_1>\ell_2$. We claim 
that to prove (\ref{item:pumpkin-stick-distance}) it is sufficient to show that
\begin{equation}
\label{eq:to-do-energetically}
	\mathscr{E}_{e_1} > \mathscr{E}_{e_2},
\end{equation}
where $\mathscr{E}_{e}$ is the Pr\"ufer amplitude defined in 
\eqref{eq:hadamard_formula}. Indeed, if this holds \emph{whenever} $\ell_1>\ell_2$, then 
by \eqref{eq:simple-hadamard}, further lengthening $e_1$ and shortening $e_2$ 
\emph{always} decreases $\lambda_2^N$. Since $\lambda_2^N$ is continuous also at 
$\ell_2=0$, this yields strict monotonicity on the entire range of possible values of $\ell_1$.

To prove \eqref{eq:to-do-energetically}, we first observe that the set $\{\psi=0\}$ of zeros 
of the eigenfunction (which we recall consists of all points of the form $\{x : \dist (x,v_-) 
= c\}$ for some $c>0$ depending on the graph) is closer to $v_1$ than $v_2$ (we do not 
rule out the possibility that it is on the edge $e_1$). Indeed, if it were not, then by 
reflecting $\{\psi \geq 0\}$ across the set $\{\psi=0\}$ (i.e.~creating a new graph which 
is reflection symmetric across $\{\psi=0\}$, such that each half is a copy of $\{\psi \geq 0 
\}$) we would obtain a new pumpkin-on-a-stick $\widetilde{\mathcal{P}}$ with a shorter 
pumpkin and a shorter edge replacing $e_1$ (since $v_2$ is closer to $v_+$ than $v_1$ 
is to $v_-$). But then the odd extension of $\psi$ from $\{\psi \geq  0\}$ to 
$\widetilde{\mathcal{P}}$ is still an eigenfunction but with the same eigenvalue. Since at 
least one edge of $\widetilde{\mathcal{P}}$ is strictly shorter than its original counterpart, 
we have created a (possibly degenerate) pumpkin chain with shorter edges and the same 
$\lambda_2^N$: this is an obvious contradiction to 
Corollary~\ref{cor:increasing_vol}(\ref{item:lengthening_an_edge}) (lengthening edges).

In particular, since $\psi|_{e_p}$ is sinusoidal and monotonic, the fact that the zero is closer to 
$v_1$ than $v_2$ means that $|\psi(v_1)| < |\psi(v_2)|$.

Armed with this, we now consider the Pr\"ufer amplitudes: since $\mathscr{E}_e$ 
depends only on the edge $e$ in question, writing $\lambda$ for $\lambda_2^N 
(\pstick{\ell_1}{m}{\ell_2})$,
\begin{equation}
\label{eq:constant-energy}
	\left[\nd{\psi}{e_p}{v_1}\right]^2 + \lambda \psi|_{e_p}(v_1)^2 = 
	\left[\nd{\psi}{e_p}{v_2}\right]^2 + \lambda \psi|_{e_p}(v_2)^2.
\end{equation}
Now we use the vertex conditions to translate this into a comparison 
between $\mathscr{E}_{e_1}$ and $\mathscr{E}_{e_2}$: continuity implies 
$\psi|_{e_p}(v_i) = \psi|_{e_i} (v_i) \equiv \psi(v_i)$ for $i=1,2$, while the Kirchhoff 
condition and the fact that $\psi$ is identical on each of the $m$ parallel edges of the 
pumpkin mean that $\nd{\psi}{e_i}{v_i} = m \nd{\psi}{e_p}{v_i}$ for $i=1,2$. Inserting 
these into \eqref{eq:constant-energy} yields
\begin{equation}
\label{eq:comparison-energy}
	\frac{1}{m}\left[\nd{\psi}{e_1}{v_1}\right]^2 + \lambda \psi|_{e_1}(v_1)^2
	= \frac{1}{m}\left[\nd{\psi}{e_2}{v_2}\right]^2 + \lambda \psi|_{e_2}(v_2)^2.
\end{equation}
Combining this with what we showed earlier, viz.\ $|\psi|_{e_1}(v_1)|<|\psi|_{e_2} 
(v_2)|$, we deduce that $\left|\nd{\psi}{e_1}{v_1}\right| > \left|\nd{\psi}{e_2}{v_2}
\right|$. Multiplying this latter inequality by $(1-1/m)$ and adding it to 
\eqref{eq:comparison-energy} now gives
\begin{displaymath}
	\mathscr{E}_{e_1} = \left[\nd{\psi}{e_1}{v_1}\right]^2 + 
	\lambda \psi|_{e_1}(v_1)^2 > \left[\nd{\psi}{e_2}{v_2}\right]^2 + 
	\lambda \psi|_{e_2}(v_2)^2 = \mathscr{E}_{e_2},
\end{displaymath}
which was to be proved.

(\ref{item:pumpkin-stick-min}) This follows immediately from 
(\ref{item:pumpkin-stick-thickness}) and (\ref{item:pumpkin-stick-distance}).

Finally, the bounds \eqref{eq:tadpole_apriori_bounds} follow directly from the (strict) 
monotonicity results of (\ref{item:pumpkin-stick-size}), since decreasing $\ell$ to $L$ 
produces a pumpkin corresponding to the upper bound, increasing $\ell$ to $0$ yields 
a path, and the behaviour of $\lambda_2^N$ is (strictly) monotonic between the two 
extremities.
\end{proof}

\subsection{Pumpkin dumbbells}
\label{sec:pumpkin-dumbbell}

We again assume $\mathcal{V} = \mathcal{V}_N$ and $\mu = \lambda_2^N$. By a 
\emph{dumbbell}, we understand a graph consisting of an edge, or \emph{handle}, $e_0$, 
with a loop attached to each end; in other words, it is a locally equilateral 
$[2,1,2]$-pumpkin chain. Here we consider a slightly more general class:

\begin{definition}
\label{def:pumpkin-dumbbell}
  Fix $m\geq 1$. A locally equilateral $[m,1,m]$-pumpkin chain 
  will be called a \emph{pumpkin dumbbell}. Any constituent 
  pumpkin is allowed to be degenerate. The middle $1$-pumpkin will be called the 
  \emph{handle}.
\end{definition}

We will again fix the total length $L$ and denote by $\pdb{\ell_1}{\ell_2}{m}$ the 
pumpkin dumbbell of length $L$, unique up to symmetries, such that
\begin{enumerate}
\item the (``left'') pumpkin adjacent to $v_-$ has total length $\ell_1 \in [0,L]$;
\item the (``right'') pumpkin adjacent to $v_+$ has total length $\ell_2 \in [0,L]$, with 
$\ell_1 + \ell_2 \leq L$;
\item the two outer pumpkins both have $m \geq 1$ edges of length $\ell_1/m$ and 
$\ell_2/m$, respectively.
\end{enumerate}
See Figure~\ref{fig:pumpkin-dumbbell}. As before, we will also write $T \subset \R^2$ 
for the closed triangle whose vertices are $(0,0)$, $(L,0)$ and $(0,L)$, so that 
$(\ell_1,\ell_2) \in T$. The pumpkin dumbbell coincides with a pumpkin-on-a-stick if 
$\ell_1=0$ or $\ell_2=0$, and also covers the special cases of a path if $\ell_1=\ell_2=0$, 
and a figure-8 if $m=2$ and $\ell_1+\ell_2=L$. If $m=2$, then we will also write
\begin{displaymath}
	\db{\ell_1}{\ell_2} := \pdb{\ell_1}{\ell_2}{2}
\end{displaymath}
for a (conventional) dumbbell. If in addition $\ell_2=0$,  then we shall write
\begin{equation}
\label{eq:notation-toad}
	\tp{\ell_1}:=\db{\ell_1}{0}\equiv\pdb{\ell_1}{0}{2}\equiv\pstick{L-\ell_1}{2}{0}
\end{equation}
for the tadpole (or lasso) of total length $L$, whose loop is of length $\ell_1$.
\begin{figure}[H]
\begin{tikzpicture}[scale=0.8]
\coordinate (a) at (0,0);
\coordinate (b) at (2,0);
\coordinate (c) at (5,0);
\coordinate (d) at (8,0);
\draw[fill] (0,0) circle (2pt);
\draw[fill] (2,0) circle (2pt);
\draw[fill] (5,0) circle (2pt);
\draw[fill] (8,0) circle (2pt);
\draw[bend left=90] (a) edge (b);
\draw[bend right=90] (a) edge (b);
\draw[bend left=90] (c) edge (d);
\draw[bend right=90] (c) edge (d);
\draw (b) -- (c);
\node at (a) [anchor=south east] {$v_-$};
\node at (d) [anchor=south west] {$v_+$};
\node at (b) [anchor=north west] {$v_1$};
\node at (c) [anchor=north east] {$v_2$};
\node at (3.5,0) [anchor=south] {$e_0$};
\coordinate (e) at (11,0);
\coordinate (f) at (13,0);
\coordinate (g) at (16,0);
\coordinate (h) at (19,0);
\draw[fill] (11,0) circle (2pt);
\draw[fill] (13,0) circle (2pt);
\draw[fill] (16,0) circle (2pt);
\draw[fill] (19,0) circle (2pt);
\node at (e) [anchor = south east] {$v_-$};
\node at (h) [anchor = south west] {$v_+$};
\node at (f) [anchor = north west] {$v_1$};
\node at (g) [anchor = north east] {$v_2$};
\node at (14.5,0) [anchor=south] {$e_0$};
\draw (f) -- (g);
\draw[bend left=54] (e) edge (f);
\draw[bend right=54] (e) edge (f);
\draw[bend left=18] (e) edge (f);
\draw[bend right=18] (e) edge (f);
\draw[bend left=90] (e) edge (f);
\draw[bend right=90] (e) edge (f);
\draw[bend left=54] (g) edge (h);
\draw[bend right=54] (g) edge (h);
\draw[bend left=18] (g) edge (h);
\draw[bend right=18] (g) edge (h);
\draw[bend left=90] (g) edge (h);
\draw[bend right=90] (g) edge (h);
\end{tikzpicture}
\caption{A dumbbell consisting of a handle $e_0$ joining two 
loops, which may (as here) be imagined as being 2-pumpkins (left); a 
more general pumpkin dumbbell with $m=6$ (right), for which the pumpkin 
on the left has total length $\ell_1$ and the one on the right total length 
$\ell_2$.}\label{fig:pumpkin-dumbbell}
\end{figure}
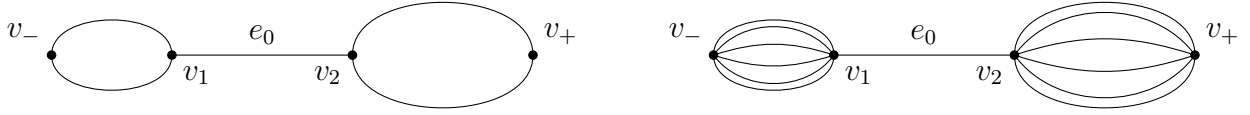

Our result, in addition to monotonicity statements analogous to those in 
Proposition~\ref{prop:pumpkin-on-a-stick}, states that ``balancing'' the 
two pumpkins, i.e., making them more equal in size, lowers the spectral 
gap $\lambda_2^N$. The tools used will be essentially the same.

\begin{proposition}
\label{prop:pumpkin-dumbbell}
Suppose that $\pdb{\ell_1}{\ell_2}{m}$ is a pumpkin dumbbell of 
fixed total length $L\geq \ell_1 + \ell_2$ with $m\geq 2$. Then 
$\lambda_2^N (\pdb{\ell_1}{\ell_2}{m})$ is simple with corresponding 
eigenfunction monotonic from $v_-$ to $v_+$, unless $\ell_1=L$ or 
$\ell_2=L$, i.e., unless it is a pumpkin. The function $(\ell_1,\ell_2) 
\mapsto \lambda_2^N (\pdb{\ell_1}{\ell_2}{m})$ is continuous on the 
closed triangle $T$. Moreover,
\begin{enumerate}
\item \label{item:pumpkin-dumbbell-size} 
for each fixed $m\geq 2$ and $\ell_2 \in [0,L]$, the function $\ell_1 
\mapsto \lambda_2^N (\pdb{\ell_1}{\ell_2}{m})$ is a strictly 
monotonically increasing function of $\ell_1 \in [0,L-\ell_2]$. A corresponding 
statement holds {\emph{mutatis mutandis}} if the roles of $\ell_1$ and 
$\ell_2$ are interchanged;
\item \label{item:pumpkin-dumbbell-balance}
if $\ell:=\ell_1+\ell_2$ and $m\geq 2$ are fixed, then $\ell_1 \mapsto 
\lambda_2^N (\pdb{\ell_1}{\ell-\ell_1}{m})$ is strictly 
monotonically increasing in $\ell_1 \in (0,\ell/2)$;
\item \label{item:pumpkin-dumbbell-thickness}
if $\ell_1,\ell_2 \in (0,L)$ are fixed, then $m\mapsto \lambda_2^N 
(\pdb{\ell_1}{\ell_2}{m})$ is strictly monotonically increasing in $m\geq 1$;
\item \label{item:pumpkin-dumbbell-min}
in particular, among all pumpkin dumbbells $\pdb{\ell_1}{\ell_2}{m}$ 
for which $\ell=\ell_1+\ell_2$ is fixed, $\lambda_2^N (\pdb{\ell_1}{\ell_2}{m}))$ 
is uniquely minimised when $\ell_1 = \ell_2 = \ell/2$; among all such 
pumpkin dumbbells where $m\geq 2$ is also allowed to vary, the minimum 
is achieved only by the regular dumbbell ($m=2$, $\ell_1=\ell_2=\ell/2$).
\end{enumerate}
Finally, for given $m\geq 2$, $\lambda_2^N (\pdb{\ell_1}{\ell_2}{m})$ satisfies 
the bound
\begin{equation}
\label{eq:dumbbell_apriori_bounds}
	\frac{\pi^2}{L^2} \leq \lambda_2^N (\pdb{\ell_1}{\ell_2}{m}) \leq
	\frac{\pi^2 m^2}{L^2},
\end{equation}
with equality in the lower estimate if and only if $\pdb{\ell_1}{\ell_2}{m}$ is a 
path (i.e.\ $\ell_1=\ell_2=0$) and in the upper one if and only if $\ell_1+\ell_2=L$.
\end{proposition}

\begin{proof}
As in the proof of Proposition~\ref{prop:pumpkin-on-a-stick}, the statements about 
simplicity and the properties of the eigenfunction were proved in 
Lemma~\ref{lem:simple-chain}, while the statements about continuity are standard.

(\ref{item:pumpkin-dumbbell-size}) The proof is essentially the same as the proof of 
monotonicity in Proposition~\ref{prop:pumpkin-on-a-stick}(\ref{item:pumpkin-stick-size}). 
Suppose $0 \leq \ell_1 < \ell_1' \leq L-\ell_2$. We consider $\pdb{\ell_1'}{\ell_2}{m}$ 
and glue the $m$ points (treated as dummy vertices) on the 
pumpkin of length $\ell_1'$ at distance $(\ell_1'-\ell_1)/m > 0$ from its vertex. Since 
the eigenfunction is longitudinal, this does not change $\lambda_2^N$. An application of 
Theorem~\ref{thm:transferring_vol}(\ref{item:unfolding_parallel}) or 
(\ref{item:symm_parallel}) to the $m$ obtained parallel edges 
transforms the graph into $\pdb{\ell_1}{\ell_2}{m}$ and decreases the eigenvalue 
strictly since the eigenfunction only has isolated zeros. This proves the statement. 
Obviously, we may interchange the roles of $\ell_1$ and $\ell_2$ if we wish.

(\ref{item:pumpkin-dumbbell-balance}) Here we will use the Hadamard formula in the form 
of \eqref{eq:simple-hadamard} as in the proof of 
Proposition~\ref{prop:pumpkin-on-a-stick}(\ref{item:pumpkin-stick-distance}). Suppose 
that $m$ and $\ell_1+\ell_2$ are fixed with $\ell_1 < \ell_2$ and denote by $e_1$ and 
$e_2$ any of the $m$ edges of the pumpkins adjacent to $v_-$ and $v_+$, respectively, 
so that $|e_1|=\ell_1/m$ and $|e_2|=\ell_2/m$. As depicted in Figure~\ref{fig:pumpkin-dumbbell}, 
we will write $e_0$ for the handle joining the two pumpkins and $v_1$, $v_2$ for its incident 
vertices, where $v_1$ is closer to $v_-$ and $v_2$ is closer to $v_+$.

Exactly as in the proof of Proposition~\ref{prop:pumpkin-on-a-stick}, by the Hadamard 
formula, it suffices to prove
\begin{equation}
\label{eq:dumbbell-energy}
	\mathscr{E}_{e_1} > \mathscr{E}_{e_2}
\end{equation}
for the graph $\pdb{\ell_1}{\ell_2}{m})$ if $0<\ell_1<\ell_2$, where, again, 
$\mathscr{E}_{e_i}$ is defined in \eqref{eq:hadamard_formula}. Indeed, this implies 
that shortening \emph{all} edges of the longer pumpkin and lengthening \emph{all} 
the edges of the shorter pumpkin by the same amount will always (strictly) lower 
$\lambda_2^N$.

Now a similar symmetry argument to the one given in the proof of 
Proposition~\ref{prop:pumpkin-on-a-stick}(\ref{item:pumpkin-stick-distance}) shows 
that the zero set $\{\psi=0\}$ 
is closer to $v_2$ than to $v_1$. To this end, we first claim that $\dist(\{\psi=0\},v_2) 
< \dist(\{\psi=0\},v_1)$: since $\lambda_2^N(\pdb{\ell_1}{\ell_2}{m}) = \lambda_1 
(\{\psi\leq 0\})$, if this were not true we could reflect $\{\psi\leq 0\}$ across 
$\{\psi=0\}$ to create a dumbbell with strictly shorter handle and/or second loop 
$e_2$ but the same eigenvalue $\lambda_2^N$. This is then a contradiction to the 
fact that lengthening an edge strictly decreases $\lambda_2^N$.

As before, it follows from the fact that $\psi$ is a monotonic sinusoidal curve on each 
edge that $|\psi(v_1)|>|\psi(v_2)|$. Using the definition of the Pr\"ufer amplitude 
$\mathscr{E}_e$, the vertex conditions and the independence of $\psi$ from the parallel 
edges in question as before, we then obtain, writing $\lambda$ for the eigenvalue,
\begin{multline*}
	m\left[\nd{\psi}{e_1}{v_1}\right]^2 + \lambda\psi|_{e_1}(v_1)^2 
	= \left[\nd{\psi}{e_0}{v_1}\right]^2 + \lambda\psi|_{e_0}(v_1)^2 \\
	= \left[\nd{\psi}{e_0}{v_2}\right]^2 + \lambda\psi|_{e_0}(v_2)^2 
	= m\left[\nd{\psi}{e_2}{v_2}\right]^2 + \lambda\psi|_{e_2}(v_2)^2
\end{multline*}
which, when combined with $\psi|_{e_1}(v_1)^2 > \psi|_{e_2} 
(v_2)^2$ as shown earlier, implies \eqref{eq:dumbbell-energy}.

(\ref{item:pumpkin-dumbbell-thickness}) This is, again, a direct consequence of 
Theorem~\ref{thm:transferring_vol}(\ref{item:symm_parallel}) 
applied to each of the pumpkins.

(\ref{item:pumpkin-dumbbell-min}) This follows immediately from 
(\ref{item:pumpkin-dumbbell-balance}) together with the observations that, for fixed $\ell 
\in (0,L)$, $\lambda_2^N (\pdb{\ell}{\ell-\ell_1}{m}) = \lambda_2^N (\pdb{\ell - 
\ell_1}{\ell}{m})$, and that $(\ell_1,\ell-\ell_1) \mapsto \lambda_2^N (\pdb{\ell}{\ell 
-\ell_1}{m})$ is continuous as $\ell_1 \to 0$ or $\ell_1 \to \ell/2$.

Finally, the bounds \eqref{eq:dumbbell_apriori_bounds} follow, analogously to the 
proof of Proposition~\ref{prop:pumpkin-on-a-stick}, from the (strict) monotonicity in 
(\ref{item:pumpkin-dumbbell-size}): starting from an arbitrary given dumbbell, 
unfolding to produce a path yields the first inequality, while expanding the pumpkins 
until the handle disappears yields the second.
\end{proof}

\subsection{Pumpkin chains with a Dirichlet vertex}
\label{sec:dirichlet-chain}

Our third result, another application of the Hadamard-type formula, is for a slightly 
larger class of graphs. We will consider locally equilateral pumpkin chains $\mathcal{P}$ 
having constituent pumpkins $\mathcal{P}_1,\ldots, \mathcal{P}_n$ (given in 
sequence, i.e., such that $\mathcal{P}_i$ has $\mathcal{P}_{i-1}$ and 
$\mathcal{P}_{i+1}$ as its neighbours). We will assume that one of the terminal 
vertices, say $v_-$, the one adjacent to $\mathcal{P}_1$, has a Dirichlet condition, 
while at all the others we have the usual natural condition; see 
Figures~\ref{fig:shifting-pumpkins} and~\ref{fig:shifting-pumpkins-conclusion}. Here 
we consider the associated smallest eigenvalue $\lambda_1^D 
(\mathcal{P})>0$.\footnote{Recall our notational convention given in 
\eqref{eq:dirichlet-eigenvalues}.} The following result states that swapping 
the order of pumpkins to move the fatter ones further away from the Dirichlet point 
decreases $\lambda_1^D$.
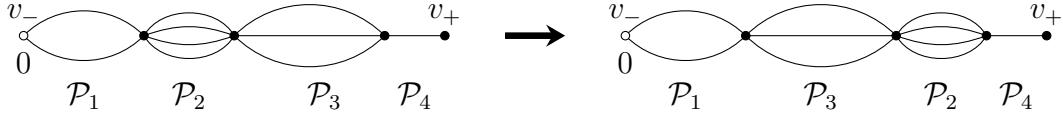
\begin{figure}[H]
\begin{tikzpicture}[scale=0.8]
\coordinate (a) at (0,0);
\coordinate (b) at (2,0);
\coordinate (c) at (3.5,0);
\coordinate (d) at (6,0);
\coordinate (e) at (7,0);
\coordinate (f) at (10,0);
\coordinate (g) at (12,0);
\coordinate (h) at (14.5,0);
\coordinate (i) at (16,0);
\coordinate (j) at (17,0);
\draw[fill] (b) circle (2pt);
\draw[fill] (c) circle (2pt);
\draw[fill] (d) circle (2pt);
\draw[fill] (e) circle (2pt);
\draw[fill] (g) circle (2pt);
\draw[fill] (h) circle (2pt);
\draw[fill] (i) circle (2pt);
\draw[fill] (j) circle (2pt);
\draw[bend left=45] (a) edge (b);
\draw[bend right=45] (a) edge (b);
\draw[bend left=20] (b) edge (c);
\draw[bend left=60] (b) edge (c);
\draw[bend right=20] (b) edge (c);
\draw[bend right=60] (b) edge (c);
\draw[bend left=45] (c) edge (d);
\draw (c) -- (d);
\draw[bend right=45] (c) edge (d);
\draw (d) -- (e);
\node at (0,-0.1) [anchor=north] {$0$};
\draw[bend left=45] (f) edge (g);
\draw[bend right=45] (f) edge (g);
\draw[bend left=20] (h) edge (i);
\draw[bend left=60] (h) edge (i);
\draw[bend right=20] (h) edge (i);
\draw[bend right=60] (h) edge (i);
\draw[bend left=45] (g) edge (h);
\draw (g) -- (h);
\draw[bend right=45] (g) edge (h);
\draw (i) -- (j);
\node at (10,-0.1) [anchor=north] {$0$};
\node at (1,-0.6) [anchor=north] {$\mathcal{P}_1$};
\node at (2.75,-0.6) [anchor=north] {$\mathcal{P}_2$};
\node at (5,-0.6) [anchor=north] {$\mathcal{P}_3$};
\node at (6.5,-0.6) [anchor=north] {$\mathcal{P}_4$};
\node at (11,-0.6) [anchor=north] {$\mathcal{P}_1$};
\node at (15.25,-0.6) [anchor=north] {$\mathcal{P}_2$};
\node at (13.25,-0.6) [anchor=north] {$\mathcal{P}_3$};
\node at (16.5,-0.6) [anchor=north] {$\mathcal{P}_4$};
\draw[-{Stealth[scale=0.5,angle'=60]},line width=2.5pt] (8,0) -- (9,0);
\node at (a) [anchor=south] {$v_-$};
\node at (e) [anchor=south] {$v_+$};
\node at (f) [anchor=south] {$v_-$};
\node at (j) [anchor=south] {$v_+$};
\draw[black,fill=white] (a) circle (2pt);
\draw[black,fill=white] (f) circle (2pt);
\end{tikzpicture}
\caption{Swapping the pumpkins $\mathcal{P}_2$ on four edges and 
$\mathcal{P}_3$ on three edges decreases $\lambda_1^D$\ldots}
\label{fig:shifting-pumpkins}
\end{figure}

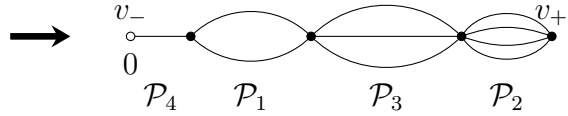
\begin{figure}[H]
\begin{tikzpicture}[scale=0.8]
\draw[-{Stealth[scale=0.5,angle'=60]},line width=2.5pt] (-2,0) -- (-1,0);
\coordinate (a) at (0,0);
\coordinate (b) at (1,0);
\coordinate (c) at (3,0);
\coordinate (d) at (5.5,0);
\coordinate (e) at (7,0);
\draw[fill] (b) circle (2pt);
\draw[fill] (c) circle (2pt);
\draw[fill] (d) circle (2pt);
\draw[fill] (e) circle (2pt);
\draw (a) -- (b);
\draw[bend left=45] (b) edge (c);
\draw[bend right=45] (b) edge (c);
\draw[bend left=45] (c) edge (d);
\draw (c) -- (d);
\draw[bend right=45] (c) edge (d);
\draw[bend left=20] (d) edge (e);
\draw[bend left=60] (d) edge (e);
\draw[bend right=20] (d) edge (e);
\draw[bend right=60] (d) edge (e);
\node at (0,-0.1) [anchor=north] {$0$};
\node at (a) [anchor=south] {$v_-$};
\node at (e) [anchor=south] {$v_+$};
\node at (0.5,-0.6) [anchor=north] {$\mathcal{P}_4$};
\node at (2,-0.6) [anchor=north] {$\mathcal{P}_1$};
\node at (4.25,-0.6) [anchor=north] {$\mathcal{P}_3$};
\node at (6.25,-0.6) [anchor=north] {$\mathcal{P}_2$};
\draw[black,fill=white] (a) circle (2pt);
\end{tikzpicture}
\caption{\ldots and so on.  The $0$ indicates the unique Dirichlet vertex, $v_-$, drawn in white.}
\label{fig:shifting-pumpkins-conclusion}
\end{figure}

\begin{proposition}
\label{prop:shifting-pumpkins}
Suppose that the locally equilateral pumpkin chain $\mathcal{P}$ consists 
of pumpkins $\mathcal{P}_1,\ldots,\mathcal{P}_n$, $n\geq 2$, such that 
$\mathcal{P}_1$ and $\mathcal{P}_n$ are the terminal pumpkins with a 
Dirichlet condition being imposed at the terminal vertex $v_-$ of 
$\mathcal{P}_1$, natural conditions at all the other vertices, 
and the $\mathcal{P}_i$ are ordered by increasing 
distance from $v_-$ (as depicted in Figure~\ref{fig:shifting-pumpkins}). 
Suppose further that for some $i=1,\ldots,n-1$, $\mathcal{P}_i$ has $m_i 
\geq 2$ parallel edges and $\mathcal{P}_{i+1}$ has $1 \leq m_{i+1} < m_i$ 
parallel edges. Denote by $\widetilde{\mathcal{P}}$ the corresponding pumpkin 
chain obtained from $\mathcal P$ by exchanging $\mathcal P_i$ and $\mathcal P_{i+1}$, 
i.e., consisting of the pumpkins $\mathcal{P}_1,\ldots,\mathcal{P}_{i-1}, 
\mathcal{P}_{i+1}, \mathcal{P}_i, \mathcal{P}_{i+2}, \ldots, \mathcal{P}_n$ 
listed in order of increasing distance from $v_-$. Then we have
\begin{displaymath}
	\lambda_1^D (\mathcal{P}) > \lambda_1^D (\widetilde{\mathcal{P}}).
\end{displaymath}
\end{proposition}

Thus if the set of constituent pumpkins is fixed, then the pumpkin chain 
minimising $\lambda_1^D$ orders them by increasing thickness away from the 
Dirichlet vertex, cf.\ Figure~\ref{fig:shifting-pumpkins-conclusion}.

\begin{remark}
\begin{enumerate}
\item The assertion of Proposition~\ref{prop:shifting-pumpkins} still holds if we 
replace the Dirichlet condition at $v_-$ by a $\delta$ condition with positive strength 
$\gamma$. The proof is identical and we do not go into details.
\item If natural conditions are imposed on all vertices of $\mathcal P$, then one 
can apply Proposition~\ref{prop:shifting-pumpkins} to either of the nodal domains 
$\mathcal P^\pm$ of the eigenfunction associated with $\lambda_2^N (\mathcal P)$ 
and deduce an analogous result as in Proposition~\ref{prop:shifting-pumpkins}, 
namely comparison with a pumpkin chain $\mathcal P'$ where fatter pumpkins have 
been moved towards the endpoints $v_-,v_+$.
\end{enumerate}
\end{remark}

\begin{proof}[Proof of Proposition~\ref{prop:shifting-pumpkins}]
The eigenvalue $\lambda_1^D (\mathcal{P})$ is simple, and by 
Lemma~\ref{lem:simple-chain}, if 
its associated eigenfunction $\psi$ is chosen positive, then it is monotonically 
increasing from $v_-$ to $v_+$ and invariant under permutations of the edges 
within any pumpkin, i.e., $\psi(x)$ depends only on $\dist (x,v_-)$.

We claim that it is sufficient to show that if $\mathcal{P}_{i-1}, \mathcal{P}_i, 
\mathcal{P}_{i+1}$ are consecutive pumpkins in such a pumpkin chain with 
$m_{i-1}, m_i, m_{i+1}$ edges, respectively, such that $m_{i-1}=m_{i+1} < m_i$, then 
shortening the edges of $\mathcal{P}_{i-1}$ and lengthening those of 
$\mathcal{P}_{i+1}$ for a fixed total length always (strictly) decreases $\lambda_1^D$ 
(and vice versa). Indeed, we may view the graphs $\mathcal{P}$ and 
$\widetilde{\mathcal{P}}$ as those obtained respectively by passing to the limit as 
$\mathcal{P}_{i+1}$ shrinks to a point, and as $\mathcal{P}_{i-1}$ shrinks to a point. 
Since the eigenvalue is continuous with respect to passing to these limits (we again refer 
to \cite[Appendix~A]{BanLev_ahp17} or \cite{BeLaSu_prep18}), it follows that 
$\lambda_1^D (\mathcal{P}) > \lambda_1^D (\widetilde{\mathcal{P}})$.

\begin{figure}[H]
\begin{tikzpicture}[scale=0.6]
\coordinate (a) at (0,0);
\coordinate (b) at (2,0);
\coordinate (c) at (3.5,0);
\coordinate (d) at (6,0);
\coordinate (e) at (7,0);
\draw[fill] (b) circle (2pt);
\draw[fill] (c) circle (2pt);
\draw[fill] (d) circle (2pt);
\draw[fill] (e) circle (2pt);
\coordinate (f) at (10,0);
\coordinate (g) at (12,0);
\coordinate (z) at (12.7,0);
\coordinate (h) at (15,0);
\coordinate (i) at (16,0);
\coordinate (j) at (17,0);
\draw[fill] (g) circle (2pt);
\draw[fill] (h) circle (2pt);
\draw[fill] (i) circle (2pt);
\draw[fill] (j) circle (2pt);
\draw[fill] (z) circle (2pt);
\coordinate (l) at (20,0);
\coordinate (m) at (22,0);
\coordinate (n) at (24.5,0);
\coordinate (o) at (26,0);
\coordinate (p) at (27,0);
\draw[fill] (m) circle (2pt);
\draw[fill] (n) circle (2pt);
\draw[fill] (o) circle (2pt);
\draw[fill] (p) circle (2pt);
\draw[bend left=45] (a) edge (b);
\draw[bend right=45] (a) edge (b);
\draw[bend left=20] (b) edge (c);
\draw[bend left=60] (b) edge (c);
\draw[bend right=20] (b) edge (c);
\draw[bend right=60] (b) edge (c);
\draw[bend left=45] (c) edge (d);
\draw (c) -- (d);
\draw[bend right=45] (c) edge (d);
\draw (d) -- (e);
\node at (0,-0.1) [anchor=north] {$0$};
\draw[bend left=20] (g) edge (z);
\draw[bend left=60] (g) edge (z);
\draw[bend right=20] (g) edge (z);
\draw[bend right=60] (g) edge (z);
\draw[bend left=45] (f) edge (g);
\draw[bend right=45] (f) edge (g);
\draw[bend left=20] (h) edge (i);
\draw[bend left=60] (h) edge (i);
\draw[bend right=20] (h) edge (i);
\draw[bend right=60] (h) edge (i);
\draw[bend left=45] (z) edge (h);
\draw (z) -- (h);
\draw[bend right=45] (z) edge (h);
\draw (i) -- (j);
\draw[bend left=45] (l) edge (m);
\draw[bend right=45] (l) edge (m);
\draw[bend left=20] (n) edge (o);
\draw[bend left=60] (n) edge (o);
\draw[bend right=20] (n) edge (o);
\draw[bend right=60] (n) edge (o);
\draw[bend left=45] (m) edge (n);
\draw (m) -- (n);
\draw[bend right=45] (m) edge (n);
\draw (o) -- (p);
\node at (10,-0.1) [anchor=north] {$0$};
\node at (1,-0.6) [anchor=north] {$\mathcal{P}_1$};
\node at (2.75,-0.6) [anchor=north] {$\mathcal{P}_2$};
\node at (5,-0.6) [anchor=north] {$\mathcal{P}_3$};
\node at (6.5,-0.6) [anchor=north] {$\mathcal{P}_4$};
\draw[-{Stealth[scale=0.5,angle'=60]},line width=2.5pt] (8,0) -- (9,0);
\node at (11,-0.6) [anchor=north] {$\mathcal{P}_1$};
\node at (14,-0.6) [anchor=north] {$\mathcal{P}_3$};
\node at (12.45,-0.6) [anchor=north] {$\mathcal{P}_2$};
\node at (15.45,-0.6) [anchor=north] {$\mathcal{P}'_2$};
\node at (16.5,-0.6) [anchor=north] {$\mathcal{P}_4$};
\draw[-{Stealth[scale=0.5,angle'=60]},line width=2.5pt] (18,0) -- (19,0);
\node at (20,-0.1) [anchor=north] {$0$};
\node at (21,-0.6) [anchor=north] {$\mathcal{P}_1$};
\node at (23.25,-0.6) [anchor=north] {$\mathcal{P}_3$};
\node at (25.2,-0.6) [anchor=north] {$\mathcal{P}'_2$};
\node at (26.5,-0.6) [anchor=north] {$\mathcal{P}_4$};
\node at (a) [anchor=south] {$v_-$};
\node at (e) [anchor=south] {$v_+$};
\node at (f) [anchor=south] {$v_-$};
\node at (j) [anchor=south] {$v_+$};
\node at (l) [anchor=south] {$v_-$};
\node at (p) [anchor=south] {$v_+$};
\draw[black,fill=white] (a) circle (2pt);
\draw[black,fill=white] (f) circle (2pt);
\draw[black,fill=white] (l) circle (2pt);
\end{tikzpicture}
\caption{The pumpkin $\mathcal P'_2$ is a clone of $\mathcal P_2$.}\label{fig:shifting-pumpkins-3}
\end{figure}
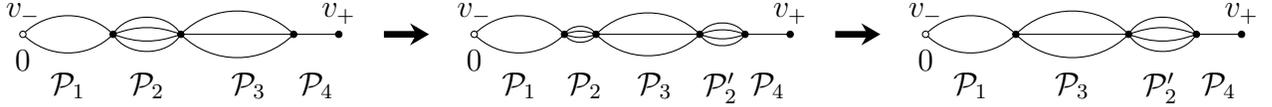

Denote by $e_{i-1}$, $e_i$, $e_{i+1}$ any of the parallel edges of 
$\mathcal{P}_{i-1}$, $\mathcal{P}_{i}$, $\mathcal{P}_{i+1}$, respectively, and 
suppose that $v_i$ is the vertex between $\mathcal{P}_{i-1}$ and $\mathcal{P}_i$, 
and $v_{i+1}$ is the vertex between $\mathcal{P}_i$ and $\mathcal{P}_{i+1}$. Then 
by \eqref{eq:simple-hadamard}, to prove the proposition we only need to show that, 
in the situation just described,
\begin{displaymath}
	\mathscr{E}_{e_{i-1}} > \mathscr{E}_{e_{i+1}}.
\end{displaymath}
The argument is thus reduced to one similar to those given in the proofs of 
Proposition~\ref{prop:pumpkin-on-a-stick}(\ref{item:pumpkin-stick-distance}) and 
Proposition~\ref{prop:pumpkin-dumbbell}(\ref{item:pumpkin-dumbbell-balance}): we 
note that, firstly, $0 < \psi (v_{i-1}) < \psi (v_{i+1})$ since $\psi$ is monotonically 
increasing and non-constant along the pumpkin chain, and secondly, writing $\lambda$ 
for the eigenvalue $\lambda_1^D$ of the pumpkin chain in question,
\begin{equation}
\label{eq:pumpkin-in-the-middle}
	\left[\nd{\psi}{e_i}{v_{i}}\right]^2 + \lambda \psi (v_{i})^2 = \mathscr{E}_{e_i}
	= \left[\nd{\psi}{e_i}{v_{i+1}}\right]^2 + \lambda \psi (v_{i+1})^2.
\end{equation}
The Kirchhoff condition at $v_i$ and $v_{i+1}$, together with the fact that $\psi$ 
is equal on all edges of a given pumpkin, then reads
\begin{displaymath}
	m_{i-1} \left|\nd{\psi}{e_{i-1}}{v_i}\right| = m_i \left|\nd{\psi}{e_i}{v_i}
	\right|, \qquad m_i \left|\nd{\psi}{e_i}{v_{i+1}}\right| 
	= m_{i+1} \left|\nd{\psi}{e_{i+1}}{v_{i+1}}\right|;
\end{displaymath}
inserting this into \eqref{eq:pumpkin-in-the-middle} implies
\begin{displaymath}
	\frac{m_{i-1}}{m_i} \left[\nd{\psi}{e_{i-1}}{v_i}\right]^2 + 
	\lambda \psi(v_i)^2 = \frac{m_{i+1}}{m_i} 
	\left[\nd{\psi}{e_{i+1}}{v_{i+1}}\right]^2 + \lambda \psi(v_{i+1})^2.
\end{displaymath}
Recalling that $0 < m_{i-1}=m_{i+1} < m_i$ and $\psi(v_i)^2 < \psi(v_{i+1})^2$, this 
implies that $\left[\nd{\psi}{e_{i-1}}{v_i}\right]^2 > 
\left[\nd{\psi}{e_{i+1}}{v_{i+1}}\right]^2$ and
\begin{displaymath}
	\mathscr{E}_{e_{i-1}} = \left[\nd{\psi}{e_{i-1}}{v_i}\right]^2 + 
	\lambda \psi(v_i)^2 > \left[\nd{\psi}{e_{i+1}}{v_{i+1}}\right]^2 + 
	\lambda \psi(v_{i+1})^2 = \mathscr{E}_{e_{i+1}},
\end{displaymath}
as required.
\end{proof}

%%%%%%%%%%%%%%%%%%%%%%%%%%%%%%%%%%%%%%%%%%%%%%%%%%%%%%%%%%%%%%%%%%%%%%%%%%%%%%
\section{The size of the doubly connected part}
\label{sec:sizeof}

We are finally in a position to give the principal application of the paper, namely the 
quantitative lower bound on $\lambda_2^N$. It involves the following quantity.

\begin{definition}
\label{def:doubly-connected-part}
  The \emph{doubly connected part} $\dcp$ of a graph $\Graph$ is the closed 
  subgraph consisting of all $x \in \Graph$ for which there is a 
  non-self-intersecting path in $\Graph$ starting and ending at $x$.
\end{definition}

The doubly connected part $\dcp$ can be obtained by deleting every bridge 
(including pendant edges), followed by the removal of any isolated vertices. It may 
also be characterised as being the largest subgraph of $\Graph$ whose every 
connected component is itself doubly edge connected.

\begin{example}
Suppose $\mathcal{D} = \db{\ell_1}{\ell_2}$ is a dumbbell having loops 
$e_1$ and $e_2$ of length $\ell_1$ and $\ell_2$, respectively. Then 
$\dcpgraph{\mathcal{D}} = e_1 \cup e_2$. More generally, if $\Graph$ is a 
pumpkin chain, then $\dcp$ consists of the (possibly disjoint) union of the 
non-trivial constituent pumpkins of $\Graph$. 
\end{example}

Our goal is to use the tools of Section~\ref{sec:tools} and the results of 
Section~\ref{sec:pumpkins} to derive a lower bound on $\lambda_2^N 
(\Graph)$ in terms of the total length $|\dcp|$ of $\dcp$ (as well as the 
length $L:=|\Graph|$ of $\Graph$).
 This bound will interpolate between 
the inequalities of Nicaise \cite[Th\'eor\`eme~3.1]{Nic_bsm87} and 
Band--L\'evy \cite[Theorem~2.1]{BanLev_ahp17}. We recall that the former 
states that for any compact graph $\Graph$ of total length $L$, we have
\begin{equation}
\label{eq:nicaise}
	\lambda_2^N (\Graph) \geq \lambda_2^N (\mathcal{I}) = 
	\frac{\pi^2}{L^2},
\end{equation}
where $\mathcal{I}$ is a path graph (interval) of length $L$; the latter 
is for \emph{doubly connected} compact graphs $\Graph$ of total length $L$, i.e., 
graphs $\Graph$ for which $\dcp = \Graph$, and reads
\begin{equation}
\label{eq:band-levy}
	\lambda_2^N (\Graph) \geq \lambda_2^N (\mathcal{C}) = 
	\frac{4\pi^2}{L^2},
\end{equation}
where $\mathcal{C}$ is a loop of the same total length 
(the same inequality was proved earlier in \cite{KurNab_jst14} 
for Eulerian graphs). Our theorem is as follows.

\begin{theorem}
\label{thm:doubly-connected}
  Suppose that the compact and connected graph 
  $\Graph$ has total length $L$, and its doubly connected part has total length 
  $V:=|\dcp| \in [0,L]$. Let $\mathcal{D} = \db{\frac{V}{2}}{\frac{V}{2}}$ be the 
  dumbbell of length $L$ having both loops of length $V/2$. Then
\begin{displaymath}
	\lambda_2^N (\Graph) \geq \lambda_2^N (\mathcal{D}).
\end{displaymath}
\end{theorem}

Variants are possible: see Theorem~\ref{thm:doubly-connected-component} 
and Corollary~\ref{cor:circumference}. Also note that since 
$\lambda_2^N(\db{\frac{V}{2}}{\frac{V}{2}})$ is monotonically increasing with respect to 
$V$ (see Proposition~\ref{prop:pumpkin-dumbbell}(\ref{item:pumpkin-dumbbell-size})), 
for any non-tree Theorem~\ref{thm:doubly-connected} yields a strictly better estimate 
than \eqref{eq:nicaise}. Moreover, it contains \eqref{eq:nicaise} and 
\eqref{eq:band-levy} as special cases (the former corresponds to $V=0$, 
the latter to $V=L$), and thus interpolates smoothly and monotonically 
between them as $V$ ranges from $0$ to $L$.

\begin{remark}
\label{rem:nicaise}
  Actually, using Lemma~\ref{lem:pumpkinchain} and 
  Theorem~\ref{thm:transferring_vol}(\ref{item:unfolding_parallel}), we can 
  immediately sketch simpler short proofs of \eqref{eq:nicaise} and \eqref{eq:band-levy}: 
  to prove \eqref{eq:nicaise}, assume without loss of generality that $\Graph$ is a 
  pumpkin chain, and apply 
  Theorem~\ref{thm:transferring_vol}(\ref{item:unfolding_parallel}) repeatedly until 
  every pumpkin is turned into a path. To prove \eqref{eq:band-levy}, note that after 
  applying Lemma~\ref{lem:pumpkinchain} every constituent pumpkin has at least 
  two edges. Hence apply 
  Theorem~\ref{thm:transferring_vol}(\ref{item:unfolding_parallel}) to reduce each 
  to a pumpkin on two edges. The resulting graph is a $[2,2,\ldots,2]$-pumpkin 
  chain. Cutting through all the vertices of degree four lowers $\lambda_2^N$ and 
  produces $\mathcal{C}$.
\end{remark}

\begin{proof}[Proof of Theorem~\ref{thm:doubly-connected}]
Fix any eigenfunction $\psi$ of $\Graph$ associated with $\lambda_2^N (\Graph)$. 
We may assume without loss of generality that $\psi$ does not vanish identically on 
any edge of $\Graph$. Indeed, if it did, say on the edge $e \sim v_1 v_2$, we could 
form a new graph $\NewGraph$ by deleting $e$ and gluing $v_1$ and $v_2$; then 
$\lambda_2^N(\NewGraph) = \lambda_2^N(\Graph)$ by 
Corollary~\ref{cor:increasing_vol}(\ref{item:shrinking_redundant}). Moreover, 
$|\NewGraph|\leq L$, and the bounds in the theorem are decreasing functions of $L$.

\textbf{Step 1.} We start by creating a pumpkin chain $\mathcal{P}$ out of $\Graph$ 
in accordance with Lemma~\ref{lem:pumpkinchain}: then $\mathcal{P}$ still has 
length $L$, $\lambda_2^N (\mathcal{P}) = \lambda_2^N (\Graph)$, and up to the usual 
identification $\psi$ is still an eigenfunction on $\mathcal{P}$, monotonically increasing 
on each pumpkin. Now gluing together vertices can only shorten the paths used in 
Definition~\ref{def:doubly-connected-part}, while new paths could be created; hence 
the size of the doubly connected part can only increase, i.e., $|\mathscr D_{\mathcal P}| 
\geq V$.

\textbf{Step 2.} We now apply 
Theorem~\ref{thm:transferring_vol}(\ref{item:symm_parallel}) with $m=2$ to 
each constituent pumpkin of $\mathcal{P}$, using $\psi$ as the eigenfunction with the 
necessary properties, to obtain a locally equilateral pumpkin chain 
$\widetilde{\mathcal{P}}$ with $\lambda_2^N (\widetilde{\mathcal{P}}) \leq 
\lambda_2^N (\mathcal{P})$, each of whose pumpkins has either one or two edges, 
and the sum of the lengths of the two-pumpkins is still $|\mathscr D_{\mathcal P}|$.

\textbf{Step 3.} Let $\tilde\psi$ be an eigenfunction corresponding to 
$\lambda_2^N(\widetilde{\mathcal{P}})$: then, as established in 
Lemma~\ref{lem:simple-chain}, $\tilde\psi$ is longitudinal (in particular 
invariant with respect to permutations of the edges of each two-pumpkin), 
monotonic between the two terminal vertices of $\widetilde{\mathcal{P}}$, and 
does not vanish on any edge. In particular, 
the two sets $\widetilde{\mathcal{P}}^+ := \{x \in \widetilde{\mathcal{P}}: \tilde\psi 
(x) \geq 0 \}$ and $\widetilde{\mathcal{P}}^- := \{x \in \widetilde{\mathcal{P}}: 
\tilde \psi (x) \leq 0 \}$ are connected and, up to identifying the (at most two) points 
where $\tilde\psi = 0$, creating the vertex $ v_0$,  are themselves locally 
equilateral pumpkin chains as in Figure~\ref{fig:generalised-dumbbell}.
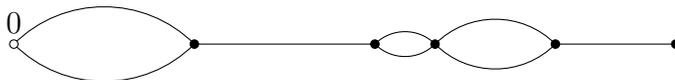
\begin{figure}[H]
\begin{tikzpicture}[scale=0.8]
\coordinate (b) at (2,0);
\coordinate (c) at (5,0);
\coordinate (d) at (8,0);
\coordinate (e) at (9,0);
\coordinate (f) at (11,0);
\coordinate (g) at (13,0);
\draw[fill] (5,0) circle (2pt);
\draw[fill] (8,0) circle (2pt);
\draw[fill] (9,0) circle (2pt);
\draw[fill] (11,0) circle (2pt);
\draw[fill] (13,0) circle (2pt);
\draw[bend left=45] (b) edge (c);
\draw[bend right=45] (b) edge (c);
\draw[bend left=45] (d) edge (e);
\draw[bend right=45] (d) edge (e);
\draw[bend left=45] (e) edge (f);
\draw[bend right=45] (e) edge (f);
\draw (f) -- (g);
\draw (c) -- (d);
\node at (b) [anchor=south] {$0$};
\draw[black,fill=white] (2,0) circle (2pt);
\end{tikzpicture}
\caption{A depiction of the set $\widetilde{\mathcal{P}}^+$; the point $0$ refers to 
the set $\{\tilde\psi=0\}$, which is taken as a Dirichlet condition at the vertex.}
\label{fig:generalised-dumbbell}
\end{figure}
We note that $\tilde\psi$ continues to be an eigenfunction on 
$\widetilde{\mathcal{P}}^\pm$ and, since it does not change sign, we have 
$\lambda_2^N (\widetilde{\mathcal{P}}) = \lambda_1^D 
(\widetilde{\mathcal{P}}^\pm)$, the latter graphs being equipped with a Dirichlet 
condition at the vertex corresponding to $\{\tilde\psi=0\}$. We suppose that the sum 
of the lengths of the two-pumpkins of $\widetilde{\mathcal{P}}^\pm$ is $\ell_\pm$; 
then $\ell_-+\ell_+  = | \mathscr D_{\mathcal P} |$.

We may thus apply Proposition~\ref{prop:shifting-pumpkins} to each of 
$\widetilde{\mathcal{P}}^\pm$ separately to shift the two-pumpkins away from the 
vertex $v_0$: suppose that $\widetilde{\mathcal{P}}^+$ consists of pumpkins 
$\mathcal{P}_1,\ldots, \mathcal{P}_m$, each having either one or two edges, and denote 
by $\widehat{\mathcal{P}}^+$ the pumpkin chain with a Dirichlet condition in which 
all the one-pumpkins have been lined up after the vertex $ v_0 $ and the two-pumpkins 
are in a row after them. Cutting through the vertices of degree four 
(cf.~Figure~\ref{fig:a-pumpkin-in-the-sun}), which each separate two of the 
neighbouring two-pumpkins of $\widehat{\mathcal{P}}^+$, we are left with a 
tadpole graph (which we will still denote by $\widehat{\mathcal{P}}^+$) having a 
loop of length $\ell_+$ and a pendant edge of length $|\widetilde{\mathcal{P}}^+| 
- \ell_+$, such that $\lambda_1^D (\widetilde{\mathcal{P}}^+) \geq \lambda_1^D 
(\widehat{\mathcal{P}}^+)$.

Similarly, we obtain a tadpole graph $\widehat{\mathcal{P}}^-$ having the same 
total length as $\widetilde{\mathcal{P}}^-$, a loop of length $\ell_-$, a pendant edge 
with a Dirichlet vertex, and an eigenvalue $\lambda_1 (\widehat{\mathcal{P}}^-) \leq 
\lambda_1 (\widetilde{\mathcal{P}}^-)$.

Gluing $\widehat{\mathcal{P}}^\pm$ together at their Dirichlet vertices, we obtain a 
dumbbell $\db{\ell_-}{\ell_+}$ with the same total length as $\Graph$, and such that $
\lambda_2^N (\db{\ell_-}{\ell_+}) \leq \max \lambda_1(\widehat{\mathcal{P}}^\pm)$ 
(since the eigenfunctions on the latter two graphs may be glued to
  create a valid non-trivial eigenfunction on $\db{\ell_-}{\ell_+}$.

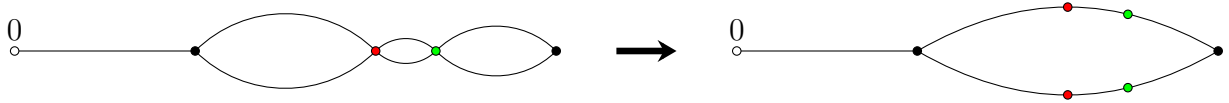
\begin{figure}[H]
\begin{tikzpicture}[scale=0.8]
\coordinate (a) at (-1,0);
\coordinate (b) at (2,0);
\coordinate (c) at (5,0);
\coordinate (d) at (6,0);
\coordinate (e) at (8,0);
\draw[fill] (2,0) circle (2pt);
\draw[fill] (8,0) circle (2pt);
\draw (a) -- (b);
\draw[bend left=45] (b) edge (c);
\draw[bend right=45] (b) edge (c);
\draw[bend left=45] (c) edge (d);
\draw[bend right=45] (c) edge (d);
\draw[bend left=45] (d) edge (e);
\draw[bend right=45] (d) edge (e);
\node at (a) [anchor=south] {$0$};
\draw[black,fill=white] (-1,0) circle (2pt);
\draw[fill=red] (5,0) circle (2pt);
\draw[fill=green] (6,0) circle (2pt);
\draw[-{Stealth[scale=0.5,angle'=60]},line width=2.5pt] (9,0) -- (10,0);
\coordinate (f) at (11,0);
\coordinate (g) at (14,0);
\coordinate (h) at (19,0);
\draw[fill] (14,0) circle (2pt);
\draw[fill] (19,0) circle (2pt);
\draw (f) -- (g);
\draw[bend left=30] (g) edge (h);
\draw[bend right=30] (g) edge (h);
\node at (f) [anchor=south] {$0$};
\draw[black,fill=white] (11,0) circle (2pt);
\draw[fill=red] (16.5,0.73) circle (2pt);
\draw[fill=red] (16.5,-0.73) circle (2pt);
\draw[fill=green] (17.5,0.61) circle (2pt);
\draw[fill=green] (17.5,-0.61) circle (2pt);
\end{tikzpicture}
\caption{The graph created from $\widetilde{\mathcal{P}}^+$ by shifting all the 
two-pumpkins away from the Dirichlet vertex denoted by $0$ (left); the graph 
created by cutting through all the vertices of degree four, to produce the tadpole 
$\widehat{\mathcal{P}}^+$ (right).}
\label{fig:a-pumpkin-in-the-sun}
\end{figure}

\textbf{Step 4.} We have thus found a dumbbell $\mathcal{D} = \db{\ell_-}{\ell_+}$ 
with $\ell_-+\ell_+ = | \mathscr D_{\mathcal P} | \geq  V$ such that $\lambda_2^N 
(\mathcal{D}) \leq \lambda_2^N (\widetilde{\mathcal{P}})$. By 
Proposition~\ref{prop:pumpkin-dumbbell}(\ref{item:pumpkin-dumbbell-balance}), we 
finally have $\lambda_2^N (\db{\frac{V}{2}}{\frac{V}{2}}) \leq \lambda_2^N (\mathcal{D}) \leq 
\lambda_2^N (\widetilde{\mathcal{P}}) \leq \lambda_2^N (\mathcal{P}) = 
\lambda_2^N (\Graph)$, proving Theorem~\ref{thm:doubly-connected}.
\end{proof}

We shall now give a variant of Theorem~\ref{thm:doubly-connected} which is stronger 
if the doubly connected part $\dcp$ is connected, or even if one of its connected 
components is sufficiently large compared with the rest: here, our object of comparison 
will be a tadpole graph rather than a dumbbell. We recall our notation from 
\eqref{eq:notation-toad}: for fixed $L$, $\tp{V}$ is the tadpole with total length $L$ 
and loop length $V \in [0,L]$.

\begin{theorem}
\label{thm:doubly-connected-component}
Suppose the compact and connected graph $\Graph$ has total length $L$, and its 
doubly connected part has a connected component of length $V \in [0,L]$. Then
\begin{displaymath}
	\lambda_2^N (\Graph) \geq \lambda_2^N (\tp{V}).
\end{displaymath}
\end{theorem}

Note that, for given $V>0$, $\lambda_2^N (\tp{V}) > \lambda_2^N (\db{\frac{V}{2}}{\frac{V}{2}})$, 
as follows from 
Proposition~\ref{prop:pumpkin-dumbbell}(\ref{item:pumpkin-dumbbell-thickness}). 
Thus if for example $\dcp$ is connected, or even if $\dcp$ has a connected component 
whose total length is sufficiently close to $V$, then 
Theorem~\ref{thm:doubly-connected-component} provides a better estimate than 
Theorem~\ref{thm:doubly-connected}.

\begin{proof}[Sketch of proof of Theorem~\ref{thm:doubly-connected-component}]
The proof is a simple modification of the proof of Theorem~\ref{thm:doubly-connected}, 
so we do not go into much detail. Assuming without loss of generality that the eigenfunction 
does not vanish identically on any edge, the largest doubly connected component of 
the pumpkin chain $\mathcal P$ formed from $\Graph$ is at least as
large as the largest doubly connected component in $\Graph$. Locally symmetrising, 
unfolding as necessary to create a pumpkin-on-a-stick with $m=2$ and finally invoking 
Proposition~\ref{prop:pumpkin-on-a-stick}(\ref{item:pumpkin-stick-min}), we get that 
$\lambda_2^N (\Graph)$ can estimated from below by the second eigenvalue of the tadpole 
whose loop is equal to the size of the largest doubly connected component in $ \Graph$.
\end{proof}

Finally, to illustrate the strength of the above theorems, we give a comparison 
with what is known for \emph{discrete (combinatorial) graph Laplacians}. 
We recall one of the principal results in this direction, which is 
in terms of the (discrete) \emph{girth} $s$ of a combinatorial graph $\mG$, 
defined as the shortest cycle length in the graph.

\begin{proposition}
\label{prop:discrete-tadpole}
  Among all connected combinatorial graphs $\mG$ on $n$ vertices with girth $s\geq 3$, the 
  algebraic connectivity (smallest non-trivial eigenvalue of the combinatorial 
  Laplacian) is minimised by the tadpole graph consisting of a cycle of length $s$ 
  attached at one vertex to a path of length $n-s$.
\end{proposition}

This is the principal result of \cite{Guo_dm08}, following a conjecture of \cite{FalKir_ejla98}.
By way of comparison, Theorem~\ref{thm:doubly-connected-component} implies 
a corresponding statement in terms of the \emph{circumference} $\cy = \cy 
(\Graph)$ of the metric graph $\Graph$, which we define to be the 
\emph{maximum cycle length within $\Graph$}. (By \emph{cycle} of a metric 
graph, we mean any closed path within $\Graph$ in which no edge appears twice, 
although vertices may be crossed multiple times. The assumption that $\Graph$ 
has a finite number of edges of finite length guarantees that this maximum is 
well defined.) If $\Graph$ is a tree, we define $\cy(\Graph):=0$. We will also 
use $s=s(\Graph)$ to denote the \emph{metric girth}, i.e., shortest cycle length 
in $\Graph$.

\begin{corollary}
\label{cor:circumference}
Suppose the compact and connected graph $\Graph$ has total length 
$L$, its circumference is $\cy \in [0,L]$ and its girth is 
$s \in [0,\cy]$. Then
\begin{displaymath}
	\lambda_2^N (\Graph) \geq \lambda_2^N (\tp{\cy}) \geq \lambda_2^N (\tp{s}).
\end{displaymath}
\end{corollary}

Here $\tp{V}$ is, as before, the tadpole graph of total length $L$ and loop 
length $V\in [0,L]$. Actually, the discrete equivalent of Corollary~\ref{cor:circumference} 
has just been proved; see \cite{XuLiSh_tcs18}.

\begin{proof}[Proof of Corollary~\ref{cor:circumference}]
Let $\mathcal{C} \subset \Graph$ be a cycle of length $\cy$. Then it is immediate 
from the definition that $\mathcal{C}$ is contained in a connected component 
of $\dcp$, meaning that this component has length at least as large as $\cy$.
Now apply Theorem~\ref{thm:doubly-connected-component} and, if necessary, use 
that the mapping $V \mapsto \lambda_2^N (\tp{V})$ is an increasing function of $V$ 
for fixed $L$ by Proposition~\ref{prop:pumpkin-on-a-stick}(\ref{item:pumpkin-stick-size}).
\end{proof}

%%%%%%%%%%%%%%%%%%%%%%%%%%%%%%%%%%%%%%%%%%%%%%%%%%%%%%%%%%%%%%%%%%%%%%%%%%%%%%
\def\cprime{$'$}

\end{document}